%% file: main.tex
\documentclass[final,12pt]{colt2023} 

\input{shortcuts.sty}
\usepackage{comment}
\usepackage{mathrsfs}
\usepackage{amssymb,amsmath,ntheorem,ifthen,amsfonts}
\usepackage{wrapfig}

\newcommand{\TT}{\ensuremath{\mathsf{\tiny{T}}}}
\newcommand{\T}{^{\TT}}

\usepackage{cleveref,aliascnt}
\renewtheorem{theorem}{Theorem}
\crefname{theorem}{theorem}{Theorems}
\Crefname{Theorem}{Theorem}{Theorems}

\newaliascnt{lemma}{theorem}
\renewtheorem{lemma}[lemma]{Lemma}
\aliascntresetthe{lemma}
\crefname{lemma}{lemma}{lemmas}
\Crefname{Lemma}{Lemma}{Lemmas}

\newaliascnt{corollary}{theorem}
\renewtheorem{corollary}[corollary]{Corollary}
\aliascntresetthe{corollary}
\crefname{corollary}{corollary}{corollaries}
\Crefname{Corollary}{Corollary}{Corollaries}

\newaliascnt{proposition}{theorem}
\renewtheorem{proposition}[proposition]{Proposition}
\aliascntresetthe{proposition}
\crefname{proposition}{proposition}{propositions}
\Crefname{Proposition}{Proposition}{Propositions}

\newaliascnt{definition}{theorem}
\renewtheorem{definition}[definition]{Definition}
\aliascntresetthe{definition}
\crefname{definition}{definition}{definitions}
\Crefname{Definition}{Definition}{Definitions}

\newaliascnt{remark}{theorem}
\renewtheorem{remark}[remark]{Remark}
\aliascntresetthe{remark}
\crefname{remark}{remark}{remarks}
\Crefname{Remark}{Remark}{Remarks}

\renewtheorem{example}{Example}
\crefname{example}{example}{examples}
\Crefname{Example}{Example}{Examples}

\crefname{figure}{figure}{figures}
\Crefname{Figure}{Figure}{Figures}

\newtheorem{assumption}{\textbf{A}\hspace{-3pt}}
\Crefname{assumption}{\textbf{A}\hspace{-3pt}}{\textbf{A}\hspace{-3pt}}
\crefname{assumption}{\textbf{A}}{\textbf{A}}

\title[Orthogonal Directions Constrained Gradient Method]{Orthogonal Directions Constrained Gradient Method: from non-linear equality constraints to Stiefel manifold}
\usepackage{times}

\coltauthor{%
 \Name{Sholom Schechtman} \Email{sholom.schechtman@telecom-sudparis.eu}\\
 \addr CMAP, SAMOVAR, Télécom Sudparis, Institut Polytechnique de Paris, 91120 Palaiseau, France
 \AND
 \Name{Daniil Tiapkin} \Email{dtyapkin@hse.ru}\\
 \addr HSE University, Pokrovsky Blvd, 11, Moscow, Russia, 109028%
 \AND
 \Name{Michael Muehlebach} \Email{michaelm@tuebingen.mpg.de} \\
 \addr{Max Planck Ring 4, 72076 Tuebingen, Germany}
 \AND
 \Name{{\'E}ric Moulines} \Email{eric.moulines@polytechnique.edu}\\
 \addr CMAP, {\'E}cole Polytechnique, Route de Saclay, 91128, Palaiseau
}

\begin{document}

\maketitle

\begin{abstract}
We consider the problem of minimizing a non-convex function over a smooth manifold $\cM$. We propose a novel algorithm, the \emph{Orthogonal Directions Constrained Gradient Method} (\algo)  which only requires computing  a projection onto a vector space. \algo\ is infeasible but the iterates are constantly pulled towards the manifold, ensuring the convergence of \algo\ towards $\cM$. 
\algo\ is much simpler to implement than the classical methods which require the computation of a retraction.  Moreover, we show that \algo\ exhibits the  near-optimal oracle complexities $\cO(1/\varepsilon^2)$ and $\cO(1/\varepsilon^4)$ in the deterministic and stochastic cases, respectively. Furthermore, we establish that, under an appropriate choice of the projection metric, our method recovers the \texttt{landing} algorithm of \citet{ablin2022fast}, a recently introduced algorithm for optimization over the Stiefel manifold. As a result, we significantly extend the analysis of \citet{ablin2022fast}, establishing near-optimal rates both in deterministic and stochastic frameworks. Finally, we perform numerical experiments which shows the efficiency of \algo\ in a high-dimensional setting.

\begin{keywords}%
  constrained optimization, non-convex optimization, Riemannian optimization, stochastic optimization, Stiefel manifold
\end{keywords}
\end{abstract}

\input{intro}

\input{prel}

\input{main_res}

\input{comp_cheap}
\input{stiefel}

\input{numerics}
\input{conclusion}

\clearpage
\newpage

\bibliography{math}

\clearpage
\newpage
\appendix
\input{appendix_numerics}
\input{proof_main}
\input{proof_reduced}
\input{proof_stiefel}
\end{document}

%% file: intro.tex
\section{Introduction}




Given a continuously differentiable function $f\colon \bbR^n \rightarrow \bbR$, we consider the following optimization problem:
\begin{equation}\label{eq:main_opt_prob}
  \min_{x \in \cM} f(x) , \quad \textrm{ with } \cM := \{x \in \bbR^n : h(x) = 0 \} \, ,
\end{equation}
where $h: \bbR^n \rightarrow \bbR^{n_h}$ is continuously differentiable, non-convex, $n_h>0$ represents the number of constraints and $\cM$ denotes the feasible set. Optimization problems with nonlinear constraints naturally arise in a number of areas in machine learning,  with a specific emphasis on matrix manifold optimization (see \cite{li2019orthogonal,yang2007globally,sato2021riemannian}). Examples include independent component analysis \citep{hyvarinen2009independent,ablin2018faster}, Procrustes estimation  \citep{bojanczyk1999procrustes,turaga2008statistical,turaga2011statistical}
and the orthogonally normalized neural networks in deep learning \citep{bengio_unitaryNN16,li2019orthogonal,bansal18,qi2020deep}.

When the projection to $\cM$ is computationally tractable, projected gradient method -- in which a gradient descent step on $f$ is combined with the projection to $\cM$ --  is often the preferred option. The convergence guarantees for projected gradient methods are similar to those for an unconstrained gradient descent. Moreover, projected gradients are a first-order procedure and efficiently handle the stochastic case where only one estimator of $\nabla f$ is known (see, e.g., \cite{gha_lan13, gha_lan16}). When $\cM$ is a submanifold, a typical approach is to determine a search direction in the tangent space and then apply a retraction (see e.g. \cite{absil2012projection, bonn_riem13,boumal2019global,boumal2020intromanifolds,sato2021riemannian}). Similarly, retraction-based gradient algorithms have optimal convergence rates in both deterministic and stochastic settings (see \cite{riem_first_zha16, riem_sto_zha19}). These methods are feasible, i.e., the iterates always belong to $\cM$. In most cases, however, computing the retraction is expensive and requires solving a nontrivial optimization problem.

Infeasible methods (i.e. the iterates do not remain on $\cM$) such as augmented Lagrangian and proximally guided methods seek a solution to~\eqref{eq:main_opt_prob} by solving a sequence of optimization problems (see \cite{zichong2020rate, lin_ma_xu22, xie_wright_PAL19, hong17a}). Here, the iterates are not  feasible but are gradually pushed towards $\cM$. Nevertheless, each of the optimization problems in the inner loop might be computationally involved. Moreover, these methods are sensitive to the choice of hyperparameters both in theory (often the sub-problems in the inner loop are required to be convex) and in practice.  

In this work, we propose \algo, which stands for \emph{Orthogonal Directions Constrained Gradient Method}, a new class of algorithms that are both easy to implement and computationally inexpensive while retaining the good convergence properties of gradient descent. \algo\ realizes a trade-off between two opposite goals: minimizing $f$ and guaranteeing feasibility  of solutions. In order to set up the stage, we define for each $x \in \bbR^n$ \emph{i)}  $\nabla H(x) := (1/2) \nabla (\norm{h(x)}^2)$, and \emph{ii)}  $V(x)$ the vector space  orthogonal  to $\operatorname{span}(\{\nabla h_i(x) \}_{i=1}^{n_h})$. Then, a vanilla version of \algo\ produces iterates as follows:
\begin{equation}\label{eq:alg_int} x_{k+1} = x_k - \gamma_k \nabla H(x_k) - \gamma_k \nabla_V f(x_k) \, ,
\end{equation}
where $\gamma_k > 0$ is a step size and $\nabla_V f$ denotes the orthogonal projection of $\nabla f$ onto $V(x)$. Since $\nabla H(x)$ is orthogonal to $V(x)$ by construction, the iterates, even if allowed to be infeasible, are constantly shifted in the direction of $\cM$. Moreover, $-\nabla_V f$ strives to be as close as possible to $-\nabla f$, which is the direction of descent for $f$, and thus tends to minimize $f$; see \Cref{fig:orth_dir}.

\begin{figure}[h]
  \centering{

  \resizebox{60mm}{!}{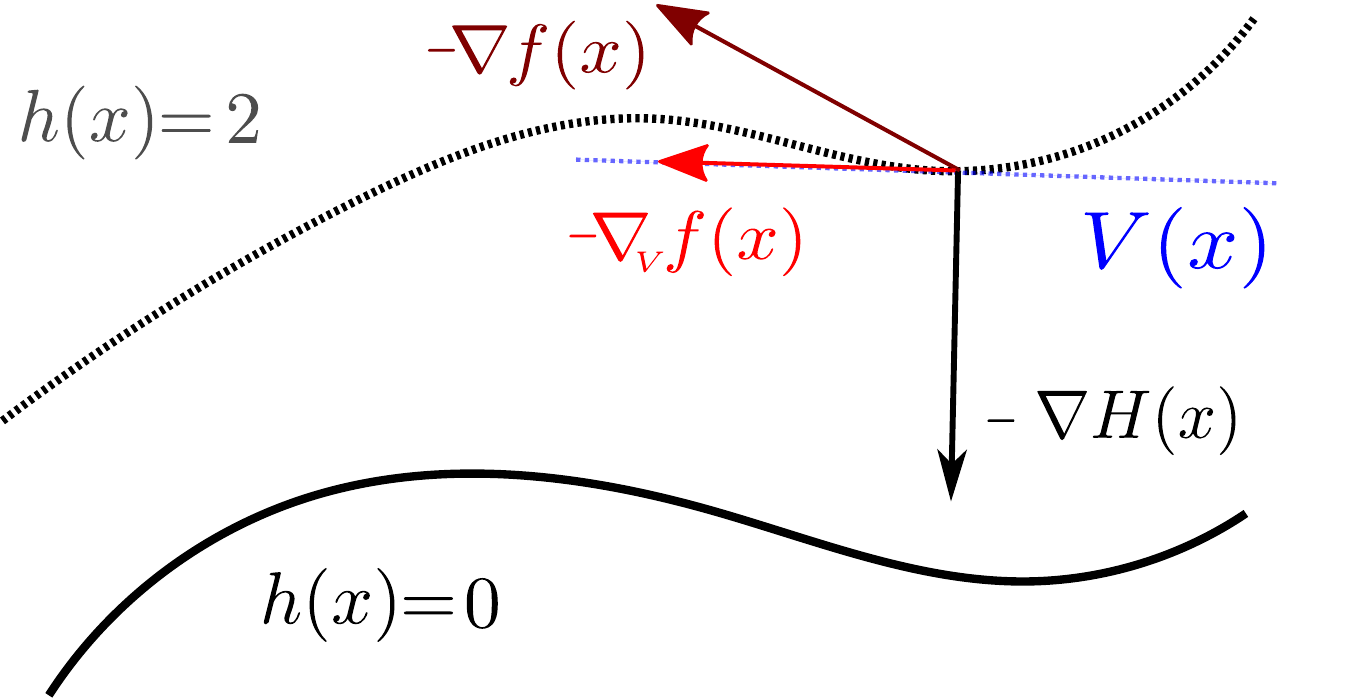}\hspace{30pt}
\resizebox{60mm}{!}{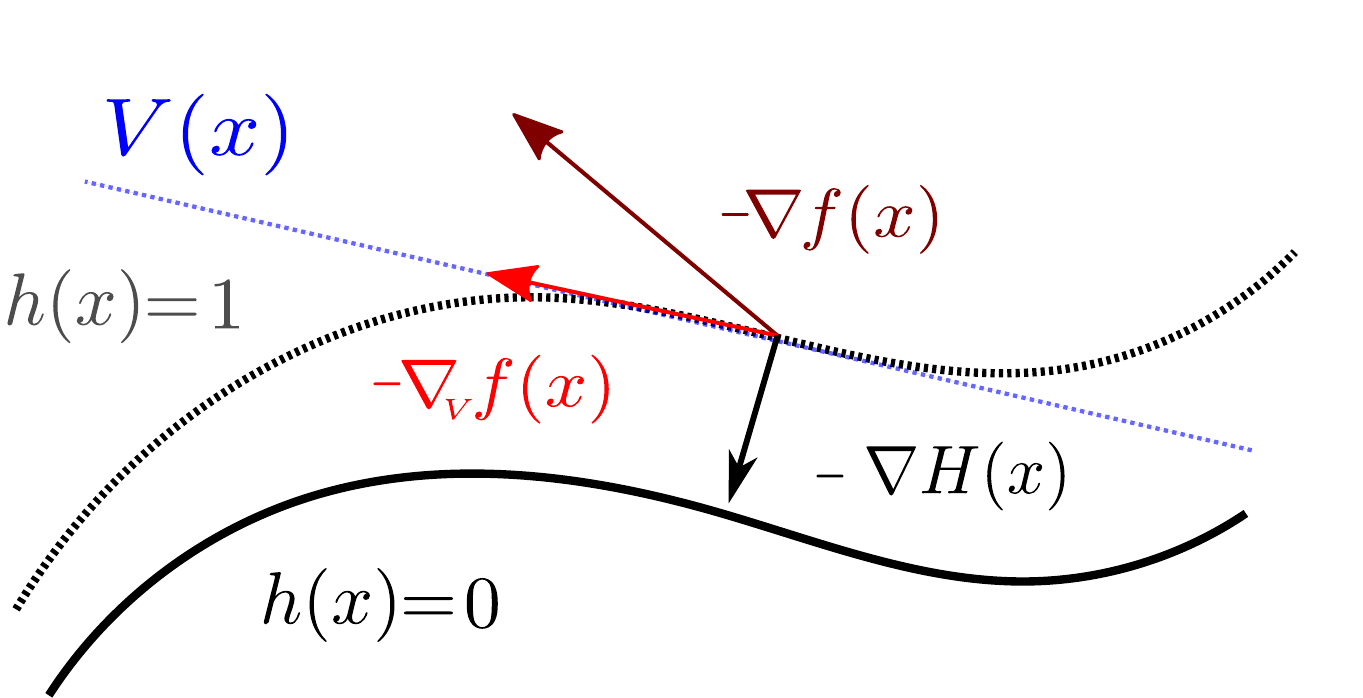} 
\caption{Construction of the orthogonal directions}\label{fig:orth_dir}
  }
\end{figure}

\algo\ is a first-order algorithm that only requires a projection onto the vector space $V(x)$ at each iteration. It is scalable, simple to implement, and can be easily generalized to the stochastic setting. In addition, we provide \algo\ with strong theoretical guarantees: we establish convergence bounds that are equivalent to those of (unconstrained) gradient descent in both deterministic and stochastic settings: $\cO(\varepsilon^{-2})$ and $\cO(\varepsilon^{-4})$, respectively. We also present \redalgo\, a computationally cheaper version of \algo\ where $V(x)$ is replaced by a hyperplane orthogonal to $\nabla H(x)$. The advantage \redalgo\ is that the projection (i.e. the computation of $\nabla_V f(x)$) now comes essentially for free, which has the potential to efficiently solve high-dimensional problems, $n-n_h \gg 1$.
This version of \algo\ is inherently non-smooth and we obtain a $\cO(\varepsilon^{-3})$ convergence rate in the deterministic setting and $\cO(\varepsilon^{-4})$ in the stochastic setting. 

\algo\ is closely related to two recently proposed methods. First, the algorithm developed in \citet{muehlebach2022constraints}, when applied to equality constraints, is a special instance of \algo. Our convergence finite-time complexity analysis extends \citet{muehlebach2022constraints} to the non-convex setting (see also \cite{stmjm_ifac22, askew_lsm22}). 
Second, \algo\ is closely related to the \texttt{landing} algorithm proposed by \citet{ablin2022fast,gao2022optimization}. The \texttt{landing} algorithm deals with the case where $\cM$ is the Stiefel (or orthogonal) manifold: it avoids retractions and requires only a few matrix multiplications at each iteration. For the orthogonal manifold case (and in the deterministic setting), \cite{ablin2022fast} provides convergence guarantees, however, with a suboptimal convergence rate. Following \cite{gao2022optimization}, we show that by choosing an appropriate metric for the projection on $V(x)$, we obtain a closed-form solution for $\nabla_V f$, and we recover \texttt{landing} as a specific instance of \algo. As a consequence, when $\cM$ is the Stiefel manifold, we significantly extend the analysis of \cite{ablin2022fast} by establishing near-optimal rates both in the deterministic and stochastic framework. In particular, we show that \texttt{landing} indeed converges to $\cM$, which was only conjectured in \cite{ablin2022fast}.\\
\vspace{5pt}
\noindent\textbf{Main contributions.}
\vspace{-10pt}
\begin{itemize}
    \item We propose \algo, a novel family of algorithms that do not require projections or retractions to the feasible set $\cM$.
    \item We establish convergence rates that coincide with the one of gradient descent in the non-convex setting: $\cO(\varepsilon^{-2})$ in the deterministic and $\cO(\varepsilon^{-4})$ in the stochastic cases; see \Cref{sec:Main}.
    \item We propose \redalgo\ which significantly decreases the computational cost per iteration. The cost of this computational reduction is a slightly degraded convergence rate: $\cO(\varepsilon^{-3})$ in the deterministic case and $\cO(\varepsilon^{-4})$ in the stochastic case; see \Cref{sec:reduced}.
    \item We introduce \geomalgo, a geometry-aware version of \algo, which is applicable when an underlying geometrical structure of the problem is available. In particular, the \texttt{landing} method of \cite{ablin2022fast} is a particular version of \geomalgo. Convergence guarantees of \geomalgo\ are identical to the one \algo; see \Cref{sec:stiefel}.
    \item We perform various numerical experiments on high-dimensional problems that highlight the claim on efficiency of our method; see \Cref{sec:numerics}.
\end{itemize}

\noindent\textbf{Notations.} For a smooth function $f : \bbR^n \rightarrow \bbR$, $\nabla f(x) \in \bbR^n$ denotes its gradient. For a smooth function $h: \bbR^n \rightarrow \bbR^{n_h}$, we denote $\nabla h(x) \in \bbR^{n \times n_h}$ the matrix in which the $i$-th column is $\nabla h_i(x)$. Given a matrix $A$, $\ker A$ denotes its kernel.
 Given a probability space $(\Omega, \cA, \bbP)$ and a filtration $(\mcF_k)$, $\bbE_k[\cdot]$ is denoted as $\bbE[\cdot | \mcF_k]$. $P_V$ denotes the orthogonal projector on the linear subspace $V$.

\noindent \textbf{Submanifolds.} A set $\cM \subset \bbR^n$ is called a submanifold of dimension $n-n_h$, with $n_h \leq n$, if for every point $x\in \cM$ there is a neighborhood $U \subset \bbR^n$ of $x$ and a smooth function $h: U \rightarrow \bbR^{n_h}$ such that $h^{-1}(0) = U \cap \cM$ and $\nabla h$ is of full rank on $U$. 
The tangent plane of $\cM$ at $x$ is $\cT_{x}\cM = \ker (\nabla h(x)^{\top})$.
For a smooth function $f : \bbR^n \rightarrow \bbR$ and $x \in \cM$, $\Grad f(x) = P_{\cT_x \cM} \nabla f(x)$ denotes the Riemannian gradient of $f$ at $x$ in the case when the Riemannian metric is inherited from the ambient space.
More generally, for  $(\cM, g)$ a manifold equipped with a Riemannian metric $g$, $\Grad_{\cM} f(x) \in \cT_x \cM$ denotes the Riemannian gradient: a vector in the tangent plane such that for any $\xi \in \cT_x \cM$,  $g_x(\Grad_{\cM} f(x), \xi) = \nabla f(x)^{\top} \xi$.

%% file: drawing_test2.pdf_tex
\begingroup%
  \makeatletter%
  \providecommand\color[2][]{%
    \errmessage{(Inkscape) Color is used for the text in Inkscape, but the package 'color.sty' is not loaded}%
    \renewcommand\color[2][]{}%
  }%
  \providecommand\transparent[1]{%
    \errmessage{(Inkscape) Transparency is used (non-zero) for the text in Inkscape, but the package 'transparent.sty' is not loaded}%
    \renewcommand\transparent[1]{}%
  }%
  \providecommand\rotatebox[2]{#2}%
  \newcommand*\fsize{\dimexpr\f@size pt\relax}%
  \newcommand*\lineheight[1]{\fontsize{\fsize}{#1\fsize}\selectfont}%
  \ifx\svgwidth\undefined%
    \setlength{\unitlength}{392.48860517bp}%
    \ifx\svgscale\undefined%
      \relax%
    \else%
      \setlength{\unitlength}{\unitlength * \real{\svgscale}}%
    \fi%
  \else%
    \setlength{\unitlength}{\svgwidth}%
  \fi%
  \global\let\svgwidth\undefined%
  \global\let\svgscale\undefined%
  \makeatother%
  \begin{picture}(1,0.51187389)%
    \lineheight{1}%
    \setlength\tabcolsep{0pt}%
    \put(0,0){\includegraphics[width=\unitlength,page=1]{drawing_test2.pdf}}%
  \end{picture}%
\endgroup%

%% file: drawing_test.pdf_tex
\begingroup%
  \makeatletter%
  \providecommand\color[2][]{%
    \errmessage{(Inkscape) Color is used for the text in Inkscape, but the package 'color.sty' is not loaded}%
    \renewcommand\color[2][]{}%
  }%
  \providecommand\transparent[1]{%
    \errmessage{(Inkscape) Transparency is used (non-zero) for the text in Inkscape, but the package 'transparent.sty' is not loaded}%
    \renewcommand\transparent[1]{}%
  }%
  \providecommand\rotatebox[2]{#2}%
  \newcommand*\fsize{\dimexpr\f@size pt\relax}%
  \newcommand*\lineheight[1]{\fontsize{\fsize}{#1\fsize}\selectfont}%
  \ifx\svgwidth\undefined%
    \setlength{\unitlength}{392.48860517bp}%
    \ifx\svgscale\undefined%
      \relax%
    \else%
      \setlength{\unitlength}{\unitlength * \real{\svgscale}}%
    \fi%
  \else%
    \setlength{\unitlength}{\svgwidth}%
  \fi%
  \global\let\svgwidth\undefined%
  \global\let\svgscale\undefined%
  \makeatother%
  \begin{picture}(1,0.51187389)%
    \lineheight{1}%
    \setlength\tabcolsep{0pt}%
    \put(0,0){\includegraphics[width=\unitlength,page=1]{drawing_test.pdf}}%
  \end{picture}%
\endgroup%

%% file: prel.tex
\section{Problem formulation and preliminaries}\label{Sec:pf}
We  consider submanifolds of $\bbR^n$ defined by a single function $h\colon \bbR^n \rightarrow \bbR^{n_h}$. A point $x \in \cM$ is a \emph{critical point} of~\eqref{eq:main_opt_prob} if:
\begin{equation}
\label{eq:definition-Grad}
  \Grad f(x) := P_{\cT_x \cM} \nabla f(x) = 0 \, .
\end{equation}
In particular, any local minimum of~\eqref{eq:main_opt_prob} is a critical point. To each $x \in \bbR^n$ we associate a vector space $V(x)= \{ v \in \bbR^n : \nabla h(x)^{\top} v = 0\}$. Note that, for any $x \in \cM$,  $V(x)= \cT_{x} \cM$. If $x$ is such that $\nabla h(x)$ has full rank, then $V(x)$ is the tangent plane of the manifold $\{ y \in \bbR^n : h(y) = h(x)\}$ (perhaps restricted to some neighborhood of $x$). Therefore, $V(x)$ extends the tangent plane outside of $\cM$.
The \emph{orthogonal directions field} is defined as:
\begin{align}\label{eq:orth_field}
\OF (x) = - \nabla h(x) A(x) h(x) - \nabla_V f(x)\, ,
\end{align}
where for all $x \in \bbR^n$, $\nabla_V f(x)$ is the orthogonal projection of $\nabla f(x)$ onto $V(x)$ and $A(x) \in \bbR^{n_h \times n_h}$ is chosen such that $\nabla h(x)^{\top}\nabla h(x) A(x)$ is a symmetric positive definite matrix. As we will see in the next sections, such an assumption enforces  that the directions along $\OF(x)$ tend to decrease $\norm{h}$. Note also that the term in $\nabla h(x)$ is orthogonal to $\nabla_V f(x)$ by construction. Before discussing the possible choices of $x \mapsto A(x)$, we show in the following lemma, that $\OF(x)$ is a meaningful way to measure the closeness of $x$ to a critical point. In particular, it is consistent with the notions of $\varepsilon$-1o point of \cite{xie_wright_PAL19} and $\varepsilon$-KKT point of \cite{birgin2018augmented,haeser2019optimality}. 
\begin{lemma}\label{lm:OF_crit}
For $x\in \bbR^n$ and $\varepsilon >0$, denote $\lambda$ the minimal singular value of $\nabla h(x) A(x)$. If $\norm{\OF(x)} \leq \varepsilon$, then $\norm{\nabla_V f(x)} \leq \varepsilon$ and $\norm{h(x)} \leq \lambda^{-1}\varepsilon $. In particular, if $\OF(x) = 0$, then $x$ is a critical point of Problem~\ref{eq:main_opt_prob}.
\end{lemma}

\begin{example}[Vanilla orthogonal directions field]\label{ex:vanil}
The first natural example is to put $A(x) \equiv \alpha \cI$, where $\cI \in \bbR^{n_h \times n_h}$ is the identity matrix. In this case, denoting $H(x) = 1/2 \norm{h(x)}^2$, it holds that $\OF(x) = - \alpha \nabla H(x) - \nabla_V f(x)$. An adaptive version of the method is obtained by choosing $A(x) =\alpha(x) \cI$, with  $\alpha : \bbR^{n} \rightarrow \bbR_{+}$ a strictly positive function.
\end{example}

\begin{example}[MJ orthogonal directions field]\label{ex:mj_flow}
For $x$ such that $\nabla h(x)$ is of full rank, another natural example is obtained by setting
$A(x)  = \alpha (\nabla h(x)^{\top} \nabla h(x))^{-1}$, where $\alpha >0$.
In this case, it turns out that $\OF$ is an instance of \cite{muehlebach2022constraints}. Denote $V_{\alpha}(x) := \{v \in \bbR^n : \nabla h(x)^{\top} v = - \alpha h(x)\}$. Note that for $x \in \cM$, $V_{\alpha}(x) = V(x) = \cT_{x} \cM$ and that $V_{\alpha}(x)$ is non-empty as soon as $\nabla h(x)$ is of full rank. A direct calculation (see \Cref{lm:aff_proj}) shows that
  \begin{equation}\label{eq:MJ_flow}
    \OF(x) = \argmin_{v \in V_{\alpha}(x)}\frac{1}{2} \norm{v + \nabla f(x)}^2  \, .
  \end{equation}
Since the computational cost of the projection on $V_{\alpha}$ and $V$ is similar, it might be interesting to compute this vector field by directly solving~\eqref{eq:MJ_flow}. However, we will see in~\Cref{sec:stiefel} that for important examples of Stiefel and orthogonal manifolds we can modify the geometry of the ambient space to obtain a computationally tractable projection onto $V$.
\end{example}

%% file: main_res.tex
 \section{Main results}\label{sec:Main}
 \subsection{Continuous-time flow}\label{subsec:cont_time}
 In this section, we analyze the ordinary differential equation $\dot{\sx}(t) = \OF(\sx(t))$.  In all the remainder, we fix $r_1 >0$ and $K \subset \bbR^{n}$ with $ K = \{ x \in \bbR^n : \norm{h(x)} \leq r_1\}$.
Consider the following assumption:
 \begin{assumption}\label{hyp:cont_model}
   \begin{enumerate}[label=\roman*), nosep,leftmargin=15pt]
     \item\label{hyp:k_comp} The set $K$ is compact and $\nabla h$ is of full rank on $K$.
     \item\label{hyp:k_fullr}  It holds that $\nabla h^{\top} \nabla h A \in \bbR^{n_h \times n_h}$ is symmetric positive definite on $K$.
     \item\label{hyp:A_loclip} The function $A : K \rightarrow \bbR^{n_h \times n_h}$ can be extended to a locally Lipschitz continuous function on some neighborhood of $K$.
     \item\label{hyp:hA_eigenv}  There is $\alpha_m >0$ such that  $\inf_{x \in K} \lambda_m(x) > \alpha_m$, where $\lambda_m(x)$ is the minimal eigenvalue of $\nabla h^{\top}(x) \nabla h(x) A(x)$
   \end{enumerate}
 \end{assumption}
Note that as soon as $\cM$ is compact, there is always some $r_1 > 0$ such that \Cref{hyp:cont_model}-\ref{hyp:k_comp} holds. Moreover, \Cref{hyp:cont_model}-\ref{hyp:k_fullr}--\ref{hyp:A_loclip} are satisfied for the matrices $A$ given in~\Cref{ex:vanil,ex:mj_flow}.   As is often the case, to analyze the trajectory of an ordinary differential equation we need to find an energy (or Lyapunov) function. For $M > 0$, we define $\Lambda_M :\bbR^{n} \rightarrow \bbR$ as:
\begin{equation}\label{eq:def_LambdaM} \Lambda_M = f + M \norm{h} \, .
\end{equation}
The following theorem is our first main result, it shows that for $M$ large enough, $\Lambda_M$ decreases along any trajectory. This observation immediately implies the convergence of any bounded trajectory to the set of critical points. 
\begin{theorem}\label{th:cont_time}
  Assume \Cref{hyp:cont_model}.
  For any $x_0$ such that $\norm{h(x_0)} \leq r_1$ there is $\sx:\bbR_{+} \rightarrow \bbR^n$ a unique solution to 
  \begin{equation}
      \label{eq:orth_flow}
      \dot{\sx}(t) = \OF(\sx(t))
  \end{equation}
  starting at $x_0$. In addition, it holds that:
  \begin{enumerate}[nosep]
    \item For any $t \geq 0$, $\norm{h(\sx(t))} \leq \rme^{-\alpha_m t} \norm{h(x_0)}$,  where $\alpha_m$ is defined in \Cref{hyp:cont_model}-\ref{hyp:hA_eigenv}.
  \item For all $M \geq \overline{M}= M_1/\alpha_m$, with $M_1 = \sup_{x \in K} \norm{A^{\top} \nabla h^{\top}(\nabla f - \nabla h A h)}$, we get
  \begin{equation*}
    \inf_{0 \leq t \leq T} \norm{\OF(\sx(t))}^2= \inf_{0 \leq t \leq T} \norm{\dot{\sx}(t)}^2 \leq \frac{1}{T} \int_{0}^{T} \norm{\dot{\sx}(t)}^2 \rmd t \leq \frac{\Lambda_{M}(\sx(0)) - \Lambda_{M}(\sx(T))}{T} \, .
  \end{equation*}
  \item Let $x^*$ be in the limit set of $\sx$, i.e. there is $t_n \rightarrow + \infty$ such that $\sx(t_n) \rightarrow x^*$. Then $x^*$ is a critical point of \eqref{eq:main_opt_prob}.
  \end{enumerate}
\end{theorem}

\begin{proof}
 The existence and uniqueness of a local solution of \eqref{eq:orth_flow} follows from the fact that $\OF$ is locally Lipschitz continuous. As we shall see, such a solution must lie in $K$, which is compact by \Cref{hyp:cont_model}. This implies that the domain of a local solution can be extended to $\bbR_{+}$. Indeed, let $\sx$ be such a solution. Since for all $v \in V$, it holds that $\nabla h^{\top} v = 0$, we get using \Cref{hyp:cont_model}-\ref{hyp:hA_eigenv}:
\begin{equation}\label{eq:h_decrease_intm}
\frac{\dif}{ \dif t} \norm{h(\sx)}^2 = - 2 h^{\top}(\sx) \nabla h^{\top}(\sx) \nabla h(\sx) A(\sx) h(\sx)  \leq -2 \alpha_m \norm{h(\sx)}^2 \, ,
\end{equation}
 and Grönwall's lemma implies that $\norm{h(\sx(t))} \leq \rme^{-\alpha_m t} \norm{h(\sx(0))} $, for $t \geq 0$.
 Therefore, any local solution stays away from the boundary of $K$ and can be extended to a global solution for which the first claim holds. We now prove the second claim. Denote $D_h = (\nabla h^{\top} \nabla h)^{-1}$. In order to simplify the notations we omit the dependence on $x$ (see Lemma~\ref{lm:aff_proj}), and get
\begin{equation}\label{eq:interm_OFA}
\OF= - \nabla f +  \nabla h \left(  D_h \nabla h^{\top} \nabla f - A h \right)  \, ,
\end{equation}
where $D_h := (\nabla h^{\top} \nabla h)^{-1}$. This implies $\nabla h^{\top} \OF = - \nabla h^{\top} \nabla hA h$. Therefore, we have
\begin{align}\label{eq:err_f}
    \begin{split}
      \norm{(\OF + \nabla f)^{\top} \OF} &= \norm{\left(  D_h\nabla h^{\top} \nabla f - A h \right)^{\top} \nabla h^{\top} \OF}
      \\
      &\leq \norm{ h^{\top} A^{\top} \nabla h^{\top} \nabla h Ah - \nabla f^{\top} \nabla h A h}  \leq M_1 \norm{h} \, .
    \end{split}
\end{align}
Finally, if $\sx \not \in \cM$, we have
\begin{equation}\label{eq:f_decrease}
  \frac{\dif}{\dif t}f(\sx) = \nabla f(\sx)^{\top} \dot{\sx} = - \norm{\dot{\sx}}^2 + (\dot{\sx} + \nabla f(\sx))^{\top}\dot{\sx}(t) \leq - \norm{\dot{\sx}}^2 + M_1 \norm{h(\sx)} \, .
\end{equation}
Therefore, using~\eqref{eq:h_decrease_intm} and \eqref{eq:f_decrease} we obtain
\begin{equation}\label{eq:strict_lyap}
  \frac{\dif}{\dif t} \Lambda_M(\sx) \leq - \norm{\dot{\sx}}^2  \leq - \norm{\nabla_{V}f(\sx)}^2\, ,
\end{equation}
where the last inequality comes from the fact that the projection of $\dot{\sx}(t)$ onto $V$ is  $\nabla_V f$.
Integrating the last inequality we obtain the second claim for $\sx$.

To establish the third claim, we notice that $\OF \neq 0$ as soon as $x \notin \cM$ or $x \in \cM$ and $\Grad (f) \neq 0$. Equation~\eqref{eq:strict_lyap} then shows that $\Lambda_M$ is a strict Lyapunov function for the ODE~\eqref{eq:orth_flow} and the set of critical points of \eqref{eq:main_opt_prob}. In particular, LaSalle's invariance principle (see e.g. \cite[Theorem 2.17]{har_dynsyst91}) then implies that any limit point of $\sx$ must be contained in the set of critical points of \eqref{eq:main_opt_prob}.
\end{proof}
\vspace{-10pt}
\subsection{Algorithm}\label{sec:det_alg}
In this section we analyze the algorithms provided by the discretization of ODE~\eqref{eq:orth_flow} both in the deterministic and stochastic settings.
  Consider a filtered probability space $(\Omega, \mcF, \{\mcF_k, k >0\},  \bbP)$. Fix $x_0 \in K$ and let $(\eta_{k})_{k \geq 1}$ be a sequence of random variables adapted to $(\mcF_k)$. Our method, \algo, produces iterates as follows:
 \begin{equation}\label{eq:orth_alg}
     x_{k+1} = x_k + \gamma_k v_k + \gamma_k \eta_{k+1} , \quad{} \textrm{ with } v_k = \OF(x_k)
 \end{equation} 
 and with $(\gamma_k)$ a sequence of positive step sizes. The perturbation $(\eta_k)$ allows to capture the case where $\nabla f(x)$ (and hence $\nabla_V f(x)$) is unknown. This covers both streaming data and finite-sum problems in machine learning; see \citep{lan2020first}.
Recall that $\bbE_k$ denotes the conditional expectation given $\mcF_k$ and consider the following assumptions.
\begin{assumption}\label{hyp:disc_model}
\begin{enumerate}[label=\roman*), nosep]
    \item\label{hyp:fh_Lipgrad}The function $f$ (respectively $h$) has $L_f$ (respectively $L_h$) Lipschitz gradients on $K$.
    \item\label{hyp:iter_bound}The iterates $(x_k)$ remain in $K$, $\bbP$-almost surely.
    \item\label{hyp:zer_mean} For every $k \in \bbN$, it holds that $\eta_{k+1} \in V(x_k)$ and $\bbE_k[\eta_{k+1}] = 0$.
    \item\label{hyp:var_bound} There is a constant $\sigma \geq 0$ such that for all $k \in \bbN$, $\bbE_k[\norm{\eta_{k+1}}^2] \leq \sigma^2$.
\end{enumerate}
\end{assumption}
 \begin{example}\label{ex:SA}
 In the stochastic approximation framework, it is assumed that there is a probability space $(\Xi, \mcT, \mu)$ and a $\mu$-integrable function $g: \bbR^n \times \Xi \rightarrow \bbR^n$ such that for each $x \in \bbR^n$, $\int g(x, s) \mu(\rmd s)  = \nabla f(x)$. Let $(\xi_k)_{k \geq 1}$ be a sequence of i.i.d random variables defined on $(\Omega, \mcF, \bbP)$, taking values in $\Xi$ and such that the distribution of $\xi_k$ is $\mu$. We consider the following recursion
 \begin{equation*} 
 x_{k+1} = x_k - \gamma_k \nabla h(x_k) A(x_k) h(x_k) - \gamma_k g_V(x_k, \xi_{k+1}) \, , \end{equation*}
where $g_V(x, \xi)$ denotes the orthogonal projection of $g(x, \xi)$ onto $V(x)$. Thus, if we denote $\eta_{k+1} :=\nabla_V f(x_k) - g_V(x_k, \xi_{k+1})$ and $\mcF_k:= \sigma(\xi_1, \dots, \xi_k)$, we obtain \eqref{eq:orth_alg}. Note also that in this case $\eta_{k+1} \in V(x_k)$, $\bbE_k[\eta_{k+1}] = 0$, and if for some $\sigma >0$, it holds that $\sup_{x \in \bbR^n} \bbE[\norm{g(x,\xi) - \nabla f(x)}^2] \leq \sigma^2$, then $\bbE_k[\norm{\eta_{k+1}}^2] \leq \sigma^2$. 
 \end{example}
The deterministic setting is recovered by setting $\sigma = 0$. If $A$ is defined only on $K$ (see  \Cref{ex:mj_flow}), then \Cref{hyp:disc_model}-\ref{hyp:iter_bound} is required for the recursions to be properly defined. However, for $A$ as in~\Cref{ex:vanil}, this assumption is not needed. Nevertheless, it is necessary for our convergence analysis, and we show in \Cref{proof:safe_step}, that, under mild assumptions, if the step-sizes are small enough \Cref{hyp:disc_model}-\ref{hyp:iter_bound} is automatically satisfied.

The following theorem is the discrete counterpart of \Cref{th:cont_time}. It shows that \algo\ converges to the set of the critical points essentially at the same rate than (unconstrained) gradient descent. 

\begin{theorem}\label{th:gen_rates}
  Assume \Cref{hyp:cont_model}--\ref{hyp:disc_model}. For any $M \geq \overline{M}$, where $\overline{M}$ is defined in \Cref{th:cont_time}, denote $D_M := \Lambda_M(x_0) - \inf_{x \in K} \Lambda_M(x)$ and let $\gamma \leq \gamma_{\max}= \min\left(\alpha_m^{-1}, (L_f + M L_h)^{-1} \right)$. Then, the following holds.
  \begin{enumerate}[nosep,leftmargin=15pt]
      \item If $\sigma = 0$, and for all $k$, $\gamma_k \equiv \gamma$, then:
      \begin{equation}\label{eq:det_rates}
  \inf_{ 0 \leq k \leq N-1}  \norm{\OF(x_k)}^2= \inf_{0\leq k \leq N-1} \norm{v_k}^2 \leq \frac{2 D_M}{N\gamma} \, .
  \end{equation}
   Furthermore, it holds that $\OF(x_k) \rightarrow 0$ and any accumulation point $x^*$ of $(x_k)$ is a critical point of Problem~\eqref{eq:main_opt_prob}.
  \item Otherwise, fix some constant $\bar D >0$, $N >0$ and $\gamma := \min(\gamma_{\max}, \bar D(\sigma \sqrt{N})^{-1})$. If $\gamma_k \equiv \gamma$, and $\hat k$ is uniformly sampled in $\{0, \dots, N-1\}$, then:
  \begin{equation}\label{eq:sto_rate}
  \bbE\left[\norm{\OF(x_{\hat k})}^2\right] \leq \frac{2D_M(L_f + M L_h+ \alpha_m)}{N} + \frac{\sigma}{\sqrt{N}}\left( \bar D (L_f + M L_h) + \frac{2 D_M}{\bar D}\right) \, .
  \end{equation}
  \end{enumerate}
\end{theorem}
\begin{proof}
Using a Taylor expansion of $\Lambda_M$ and using the upper-bound on $\gamma_k$, we obtain
    \begin{equation}\label{eq:rem_decr}
    2\left(\bbE_k[\Lambda_M(x_{k+1})] - \Lambda_M(x_k)\right) \leq - \gamma\norm{v_k}^2 + L_f + M L_h \sigma^2 \gamma^2 \, .
  \end{equation}
  Our claims then follow by telescoping this inequality and applying a standard proof technique (see e.g. \citet[Chapter~6]{lan2020first}) both in the deterministic and stochastic framework. Further details are given in \Cref{proof:gen_rates}.
\end{proof}

The preceding theorem shows that the rate of convergence of our algorithm, measured through $\OF$, is identical to the one obtained by gradient descent in a non-convex framework: $\cO(\varepsilon^{-2})$ in the deterministic setting and $\cO(\varepsilon^{-4})$ in the stochastic setting. As recently shown in \cite{CarmonLowerBF, CarmonLowerBF_sto}, these rates are tight, which makes our algorithm near-optimal in both cases.
  
   The term $(L_f + M L_h)$ in the definition of $\gamma_{\max}$ is  the Lipschitz constant of $\nabla f + M \nabla h$, hence our bound on the step sizes is reminiscent of the $L_f^{-1}$ bound required for convergence of standard gradient descent.
Note also that only an upper bound on $\overline{M}$ is required to achieve such rates. Indeed, in the deterministic setting, we can combine our method with line search; see \Cref{rm:safe_step}.
In the stochastic framework, performing line search is not an option, but we note that the discussion of \citet[Corollary~2.2.]{gha_lan13} applies here as well. In particular, we can make an error of the order of $\sqrt{N}$ in estimating $(L_f + M L_h)$ while maintaining our rate of convergence of $\cO(\varepsilon^{-4})$. If all constants are known, then the optimal $\overline{D}$ in equation~\eqref{eq:sto_rate} is $\sqrt{2D_M/(L_f + M L_h)}$. Finally, a nonconstant choice of step sizes is possible without affecting the final results; see \cite[Chapter~6]{lan2020first}. The choice of step size is further discussed in \Cref{proof:safe_step}.

%% file: comp_cheap.tex
\section{Reducing the computational costs: reduced \algo}\label{sec:reduced}
While \algo\ provides optimal theoretical guarantees, it does so by computing, at every iteration, a projection onto a $n-n_h$-dimensional vector space. For $n -n_h \gg 1$, such a projection might be computational expensive. In this section, we therefore propose a modification of \algo\ that only projects onto a hyperplane, which comes essentially for free. The main idea is to reparametrize our problem by noting that $\cM =\{ x \in \bbR^n: H(x) := \norm{h(x)}^2 / 2 = 0\}$.
Introducing the vector spaces $\tilde{V}(x) := \{v\in \bbR^n : \nabla H(x)^{\top} v = 0 \}$, the iterates of \redalgo\ are defined as follows:
\begin{equation}
  x_{k+1} = x_k - \alpha(x_k) \gamma_k \nabla H(x_k) - \gamma_k \nabla_{\tilde V}f(x_k) + \gamma_k \eta_{k+1}\, ,
\end{equation}
where, as previously, $\nabla_{\tilde V} f(x)$ denotes the projection of $\nabla f(x)$ onto $\tilde V(x)$, $(\eta_{k+1})$ is a perturbation sequence, and $\alpha(x)$ corresponds to the choice $A(x) = \alpha(x)\cI$ and is specified in \Cref{th:reduced_rates} below.

Note that, as soon as $\nabla H(x) \neq 0$, $\tilde{V}(x)$ is a hyperplane. Therefore, the computation of $\nabla_{\tilde V} f(x)$ is straightforward, preserving, at the same time, its orthogonality to $\nabla H(x)$. Thus, \redalgo\ follows the same idea as \algo\, while significantly reducing the computational costs. Unfortunately, this construction damages the continuity of $\tilde{V}(x)$ near $\cM$. Indeed, since $\nabla H(x) = \nabla h(x) \cdot h(x)$, we obtain $\nabla H(x) = 0$ and $\tilde{V}(x)= \bbR^n$ on $\cM$. This observation shows that the field associated with \redalgo\ is non-smooth. The inherent non-smoothness of \redalgo\ deteriorates its convergence properties, but we can still derive a $\cO(\epsilon^{-3})$ rate of convergence in deterministic environments and a $\cO(\varepsilon^{-4})$ rate of convergence in stochastic environments. The latter is reminiscent of the convergence rate of subgradient methods in non-smooth environments (see \citet{dav_dru_weakconv_rate19}). 

To properly analyze \redalgo, and due to a non-smooth choice of $\alpha(x)$, we consider assumptions that are slightly different from \Cref{hyp:cont_model}. More precisely, we assume \Cref{hyp:cont_model} for $A(x) = \cI$. We will call this set of assumptions \Cref{hyp:cont_model}', and we denote the smallest eigenvalue of $\nabla h(x)^\top \nabla h(x)$ as $\mu_h^2$. 

We note that the compactness of $K$ and Lipchitz-continuity of $\nabla f$ and $\nabla h$ (\Cref{hyp:disc_model}-\ref{hyp:fh_Lipgrad}) implies that $f$, $h$, and $\nabla H = \nabla h \cdot h$ are Lipschitz-continuous with Lipchitz constants $C_f$, $C_h$, and $L_H$ respectively. Moreover, since $\nabla h$ is continuous and $K$ is compact, we have $\sup_{x \in K}\norm{\nabla h(x)}_2 \leq M_h$. 



\begin{theorem}\label{th:reduced_rates}
    Assume \Cref{hyp:cont_model}'-\ref{hyp:disc_model}. Let $\bar D, \alpha >0$ and $\alpha(x) = \alpha \cdot H(x) / \norm{\nabla H(x)}^2$. Denote $D_0 = f(x_0) - \inf_{x \in K} f(x)$, $\gamma_{\max} = \min(\alpha^{-1}, (L_f + \alpha \mu_h^{-2} L_H)^{-1})$ and $\tilde C = B_f M_h \mu_h^{-2}$. Finally, assume that $x_0 \in \cM$ and fix $N>0$ the number of iterations. The following holds:
     \begin{enumerate}
        \item If $\sigma^2=0$  and for all $k \in \bbN: \gamma_k \equiv \gamma$ for $\gamma = \min(\gamma_{\max}, \bar D \cdot N^{-1/3})$, then choosing $\alpha = \gamma$, we obtain
        \[
            \inf_{k=0,\ldots,N-1} \left\{\norm{\nabla_{\tV} f(x_k)}^2 + \frac{1}{2} \norm{h(x_k)}^2\right\} \leq \frac{8D_0 (L_f + L_h \mu_h^{-2})}{N} + \frac{\left(\frac{8 D_0}{\bar D} + 8 \widetilde{C} L_H \cdot \bar D \right)}{N^{2/3}}\,.
        \]
        \item Otherwise, if for all $k \in \bbN: \gamma_k \equiv \gamma$, with $\gamma = \min(\gamma_{\max}, \bar D \cdot N^{-1/2})$, we obtain by choosing $\alpha = \gamma$ and $\hat k$ uniformly sampled in $\{0,\ldots, N-1\}$
        \begin{align*}
            \bbE\left[\norm{\nabla_{\tV} f(x_{\hat k})}^2 + \frac{1}{2}\norm{h(x_{\hat k})}^2\right] &\leq \frac{4D_0 (L_f + L_h \mu_h^{-2})}{N} + \frac{4 D_0}{\bar D \cdot \sqrt{N}}  +  \frac{4 \widetilde{C}^2 \bar D^2 \cdot L_H}{N}\\
            &+ \frac{\bar D}{\sqrt{N}}\left( 2\left(L_f + \gamma L_H \mu_h^{-2}  \right) \cdot \sigma^2 + 2 \widetilde{C} \cdot \sqrt{ \frac{L_H \sigma^2}{2}} \right)\, .
        \end{align*}
    \end{enumerate}
\end{theorem}
The main difficulty in establishing this result relies in the lack of a suitable Lyapunov function for \redalgo. The latter comes from its inherent non-smoothness and the fact that the Lagrange multipliers that arise in the problem of projection on $\tV$ are unbounded. A complete proof of this theorem is provided in \Cref{app:reduced_proofs}.

In the deterministic setting, \redalgo\ outputs $\hat x$ such that $\norm{h(\hat x)} \leq \varepsilon$ and $\norm{\nabla_{\tV} f(\hat x)} \leq \varepsilon$ in $\cO(\varepsilon^{-3})$ iterations. In the stochastic setting, \redalgo\ outputs a point $\hat x = x_{\hat k}$ such that $\bbE[\norm{h(\hat x)}] \leq \varepsilon$ and $\bbE[\norm{\nabla_{\tV} f(\hat x)}] \leq \varepsilon$ in $\cO(\varepsilon^{-4})$ iterations. One drawback of such a method is that we are no longer guaranteed to converge towards the feasible set. Nevertheless, the condition $\norm{\nabla_{\tV} f(\hat x)} \leq \varepsilon$, could be rewritten as $\varepsilon$-1o point with appropriate Lagrange multipliers proportional to $h(\hat x)$ (see \cite{xie_wright_PAL19} for the definition of an $\varepsilon$-1o point). 

%% file: stiefel.tex
\section{A geometry aware version of \algo}\label{sec:stiefel}
As mentioned earlier, a drawback of \algo\ lies in the fact that at each iteration the method evaluates a projection on $V(x)$.  \redalgo\ requires only one projection onto a hyperplane but does not exhibit optimal convergence guarantees. In fact, since the main feature of our analysis was to exploit the orthogonality of $\nabla h(x)$ and $V(x)$, one might think that the projection onto $V(x)$ (and thus $\nabla_V f(x)$) is not necessarily defined through the canonical metric. This observation is the main idea behind our \textit{Orthogonal Directions Riemannian Gradient Method} (\geomalgo), where the type of projection might depend on $x$. This implicitly provides the ambient space with a Riemannian metric and turns out to be particularly interesting for optimization over the Stiefel manifold. In fact, by a specific choice of metric, the projection has a closed form which recovers the \texttt{landing} algorithm of \citet{ablin2022fast}.  In particular, our results imply near-optimal rates of \texttt{landing}, significantly improving the ones presented in \cite{ablin2022fast}.

Before proceeding, let us introduce some notations. Let $Q : \bbR^n \rightarrow \bbR^{n \times n}$ be such that for all $x \in \bbR^n$, $Q(x)$ is a positive definite matrix. Given $v, u \in \bbR^n$, we set $q_x(u,v) = u^{\top} Q(x)v$. As a result, we change the geometry of $\bbR^n$ and transform it into a Riemannian manifold with $q_x$ as the Riemannian inner product. For $v \in \bbR^n$, we will denote $\norm{v}_{q_x} =\sqrt{ q_x(v,v)}$.
We are now ready to present the \emph{geometry-aware orthogonal directions field}
\begin{equation}\label{eq:geometric_algorithm}
  \GOF(x) = - \nabla h(x) A(x) h(x) + \argmin_{v \in V(x)} \frac{1}{2} \norm{ v + Q^{-1}(x)\nabla f(x)}^2_{q_x} \, .
\end{equation}
By replacing $\OF$ with $\GOF$ in the algorithms of \Cref{sec:Main}, we obtain geometry-aware deterministic and stochastic algorithms. 

To describe a more geometric viewpoint on this algorithm, let us define a family of manifolds $\cM_{h(x)} = \{ y \in \bbR^n : h(y) = h(x) \}$ with Riemannian metric $g^{h(x)}$ such that $g^{h(x)}_x = q_x$ parameterized by $x \in \bbR^n$. In this case, we can prove that the projection in \eqref{eq:geometric_algorithm} exactly corresponds to the negative Riemannian gradient (see Lemma~\ref{lem:riemannian_gradient_by_projection} in Appendix~\ref{app:geometric}):
\begin{equation}\label{eq:riemannian_gradient_by_projections}
    -\Grad_{\cM_{h(x)}} f(x) = \argmin_{v \in V(x)} \frac{1}{2} \norm{ v + Q^{-1}(x)\nabla f(x)}^2_{q_x} \, .
\end{equation}
In particular, if the problem at hand has a geometrical structure, one might hope that a particular choice of $Q$ might reduce the computational costs (or even exhibit a closed form solution) of the right-hand side of \eqref{eq:riemannian_gradient_by_projections}. This idea explains the ``geometry-aware" nature of the algorithm. 

The main motivation for \geomalgo\ is the example of the orthogonal, or, more generally, Stiefel manifold. In this case, for $X \in \bbR^{n \times p}$, following the recent work of \citet{gao2022optimization}, the constraints are defined by $h(X) = X^\top X - \cI$ and the manifolds $\cM_{h(X)}$ correspond to $\mathrm{St}_{X^\top X}(p,n)$. For any of such $\cM_{h(X)}$, we obtain a natural Riemannian metric inherited from the Stiefel manifold $\mathrm{St}(p,n)$ through a family of diffeomorphisms. This provides us with a natural way of defining $Q$ and we obtain (see \cite{gao2022optimization} for a detailed discussion):
\[
    \Grad_{\cM_{h(X)}} f(X)  = \psi(X) X, \quad \text{where } \psi(X) = \left(\nabla f(X) X^\top - X (\nabla f(X))^\top\right).
\]
In particular, by setting $A(x) = \lambda \cI$, our algorithm exactly recovers the \texttt{landing} algorithm \citep{ablin2022fast,gao2022optimization}
\[
    X_{k+1} = X_k - \lambda \gamma_k \nabla H(X_k) - \gamma_k \psi(X_k) X_k.
\]
In other words, our approach is a generalization of the \texttt{landing} algorithm beyond the orthogonal and Stiefel manifolds.

Next we analyze \geomalgo\ under the following assumption. 
\begin{assumption}\label{hyp:riem_proj}
  There is a constant $C_q >0$ such that
  \begin{equation*}
    \sup_{x \in K}\max(\norm{Q^{-1}(x)}, \norm{Q(x)}) \leq C_q \, .
  \end{equation*}
\end{assumption}
The following theorem shows that \geomalgo\ exhibits the same type of rates than \algo. We emphasize that all our analysis automatically holds for the \texttt{landing} algorithms as a special case. In particular, we obtain new and better rates for \texttt{landing}, where only an $\cO(\epsilon^{-6})$ rate was previously proven for the deterministic (and with decreasing step-sizes) version of the algorithm. Furthermore, we establish the convergence of the deterministic version of \texttt{landing} to the Stiefel manifold, which was only conjectured in \cite{ablin2022fast}. A full proof is provided in \Cref{app:geometric}.
\begin{theorem}\label{th:stief_rates}
  Let \Cref{hyp:cont_model}--\ref{hyp:riem_proj} hold. Then there exists $\overline{M}_q$ (detailed in the proof) such that for all $M \geq \overline{M}_q$, denoting $\gamma_{\max}>0$ as $\gamma_{\max} = \min(\alpha_m^{-1}, C_q^{-1}(L_f + M L_h)^{-1})$, the following holds:
  \begin{enumerate}
      \item If $\sigma = 0$ and $\gamma_k \equiv \gamma$, with $\gamma \leq \gamma_{\max}$, then:
        \begin{equation*}
\inf_{0 \leq k \leq N-1} \norm{\GOF(x_k)}^2 \leq    \inf_{0 \leq k \leq N-1}{\norm{v_k}^2} \leq 2 C_q\frac{\Lambda_M(x_0) - \Lambda_{M}(x_n)}{N \gamma} \, .
  \end{equation*}
  Furthermore, $\norm{\GOF(x_k)} \rightarrow 0$ and any accumulation point of $(x_k)$ is a critical point of Problem~\ref{eq:main_opt_prob}.
  \item Otherwise, fix some constant $\bar D$, $N >0$ and $\gamma = \min(\gamma_{\max}, \bar D (\sigma \sqrt{N})^{-1})$. If $\gamma_k \equiv \gamma $ and $\hat k$ is uniformly sampled in $\{0, \dots, N-1\}$, then 
  \begin{equation*}
      \bbE_k [\norm{\GOF(x_{\hat k})}] \leq \frac{2C_qD_M(L_f + M L_h+ \alpha_m)}{N} + \frac{C_q\sigma}{\sqrt{N}}\left( \bar D (L_f + M L_h) + \frac{2 D_M}{\bar D}\right) \, .
  \end{equation*}
  \end{enumerate}

\end{theorem}

%% file: numerics.tex
\section{Numerical experiments}\label{sec:numerics}
We showcase the efficiency of the proposed algorithms on different optimization problems.
\paragraph{Procrustes problem}
Let $A,B$ be matrices with $A \in \bbR^{q \times q}$ and $B \in \bbR^{p \times q}$, where $p \geq q$. We consider the orthogonal Procrustes problem of finding a matrix $X \in \bbR^{p \times q}$ with orthonormal columns solving the minimization problem
$\min _{X^{\top} X= \cI_q}  \left\|A X-B\right\|^2_{\mathrm{F}}$,
where $\|\cdot\|_{\mathrm{F}}$ is the Frobenius norm. This is referred to as the Procrustes problem on the Stiefel manifold; see \citep{elden1999procrustes}.
We compare \algo, \redalgo, \geomalgo\ with Riemannian gradient descent with two different choices of Riemannian metric: Euclidean and Canonical. The results are shown in \Cref{fig:procrustes_small_scale} in log-log scale for $p=60, q=40$. The results are averaged over $n=100$ draws for the matrices $A$ and $B$ [the entries of the matrices are sampled from a standard normal distributions]. For this experiment, we choose $A(x) = 5 \cI$; we use a constant step size $\gamma_k = 10^{-2}$ for \algo\ and \geomalgo, and decreasing step size for \redalgo\ $\gamma_k = 10^{-2} \cdot k^{-1/3}$.
\begin{figure}
    \centering
    \includegraphics[width=0.98\linewidth]{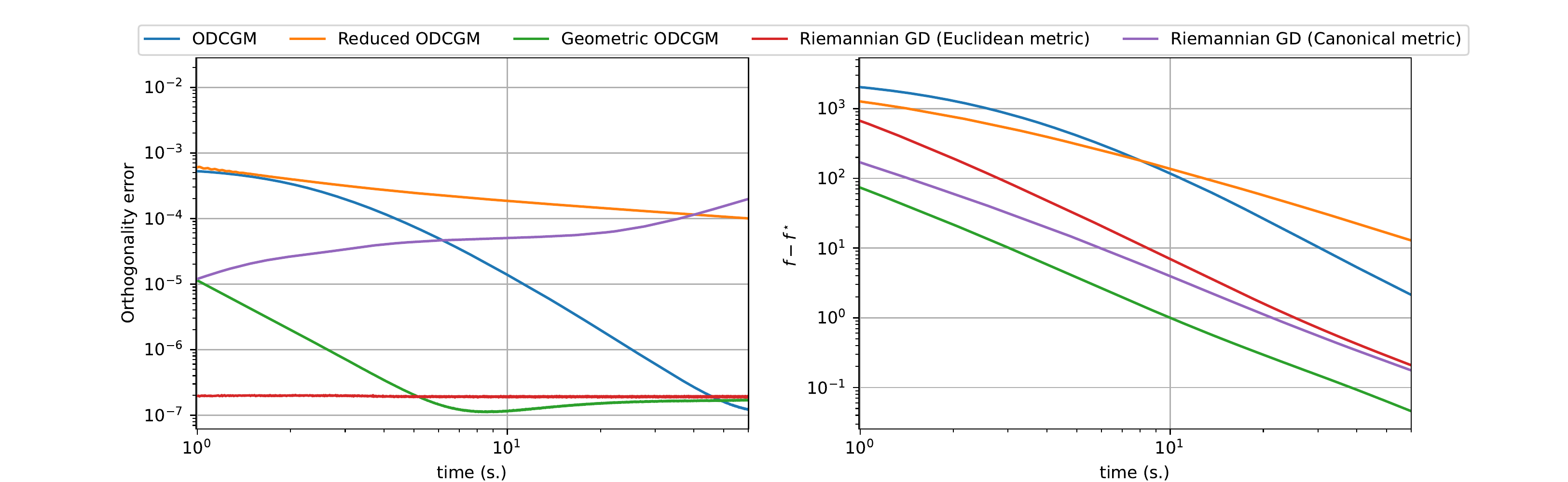}
    \caption{Comparison of \algo\ (blue), \redalgo\ (orange), and \geomalgo\ (green) with Riemannian gradient descent with two different Riemannian metrics (red and purple). The upper plot shows the orthogonality error for $X$, the lower plot shows the convergence of the objective function (averages over 100 seeds).}
    \label{fig:procrustes_small_scale}
\end{figure}
In particular, we find that \geomalgo\ outperforms the Riemannian gradient descent methods for both the Euclidean and canonical Riemannian metrics, and achieves the orthogonality error at the level of machine accuracy. We also see numerical confirmation of the $\cO(\varepsilon^{-2})$ convergence of \algo\ and \geomalgo\ and the slower convergence of \redalgo. Additional experiments on a large instance of the problem are presented in \Cref{app:numerics}.
\paragraph{Hanging chain}
As a second non-convex and nonlinear example, we compute the shape of a hanging chain. The problem can be formulated as follows:
\begin{multline}
\min_{(\xi_1,\dots,\xi_N)\in\mathbb{R}^{2N}} \frac{1}{N^3} \sum_{i=1}^N \left( \frac{k_\text{s}}{r^4}(\xi_{i-1}-\xi_i)\T (\xi_{i+1}-\xi_i) + y_i \right) \\ \text{s.t.} \quad \sqrt{(\xi_{k-1}-\xi_k)\T (\xi_{k-1}-\xi_k)}\leq r,~ k=1,2,\dots,N+1, \label{eq:numEx}
\end{multline}
where $\xi_k=(x_k,y_k)$ denotes the $xy$-position of the $k$-th element, $k=1,\dots N$, and $\xi_{0}=(0,0)$ and $\xi_{N+1}=(9,0)$ are the two endpoints. Further details are given in \Cref{app:numerics}. We compare the results of \algo~ with $A(x) = \alpha (\nabla h(x)^{\top} \nabla h(x))^{-1}$ (hereafter abbreviated as \algo),  \redalgo~, and an augmented Lagrangian method. The results are summarized in \Cref{fig:objfunconst2} for the case $N=10,000$, which leads to $20,000$ decision variables and $10,001$ nonlinear constraints.
We note that  \redalgo~ and augmented Lagrangian converge much more slowly than \algo. We also find that fine-tuning the augmented Lagrangian method is quite difficult, while the time steps of \algo~ and  \redalgo~ are easy to set (see \Cref{app:numerics} for details). The execution time per iteration of \algo~ is about five times that of \redalgo~ and the augmented Lagrangian. To demonstrate the potential of \redalgo~, we run the same example for $N=2 \times 10^5$, resulting in a large optimization problem with $4 \times 10^5$ decision variables and $2\times10^5$ nonlinear constraints. Under these conditions, solving the Karush-Kuhn-Tucker system becomes challenging at each iteration, which is required for \algo~. However, the  \redalgo~ still performs well, requiring only about 0.85 seconds to execute a single iteration.
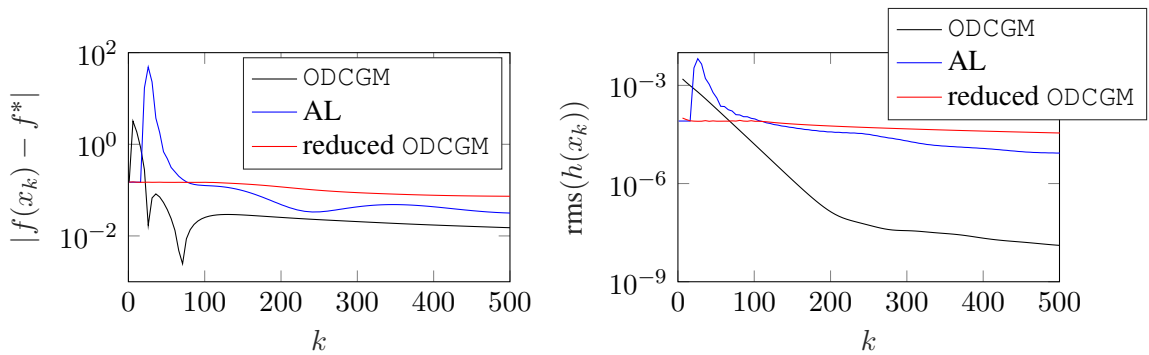
\begin{figure}
\newlength{\figurewidth}
\newlength{\figureheight}
\setlength{\figurewidth}{.35\columnwidth}
\setlength{\figureheight}{.2\columnwidth}
\input{media/fvalueUpdated.tikz} %
\input{media/gnormFinal.tikz}
\vspace{-10pt}
\caption{The figure compares the results of \algo~ (black) with \redalgo~ (red) and an augmented Lagrangian method (blue). \redalgo~ is performed with a decreasing step size of $\cO(k^{-1/2})$. The left plot shows the convergence in objective function, while the right plot gives the mean square error (denoted by rms) of the constraint violations. We note that \algo~ converges much faster than \redalgo~ and the augmented Lagrangian.}
\label{fig:objfunconst2}
\end{figure}

%% file: media/fvalueUpdated.tikz
%
%
\begin{tikzpicture}

\begin{axis}[%
width=0.951\figurewidth,
height=\figureheight,
at={(0\figurewidth,0\figureheight)},
scale only axis,
xmin=0,
xmax=500,
xlabel style={font=\color{white!15!black}},
xlabel={$k$},
ymode=log,
ymin=0.001,
ymax=100,
yminorticks=true,
ylabel style={font=\color{white!15!black}},
ylabel={$|f(x_k)-f^*|$},
axis background/.style={fill=white},
legend style={legend cell align=left, align=left, draw=white!15!black}
]
\addplot [color=black]
  table[row sep=crcr]{%
1	0.148936627077608\\
6	3.19644062127997\\
11	1.85468142366626\\
16	0.849614045696708\\
21	0.279223093257881\\
26	0.0169838133470177\\
31	0.0674057487313153\\
36	0.0825892429096275\\
41	0.0708015234256438\\
46	0.053679203147138\\
51	0.0374649873632293\\
56	0.0239047765020888\\
61	0.013215608457398\\
66	0.00478567341409164\\
71	0.00246210064058183\\
76	0.00864590799828091\\
81	0.0136928568834772\\
86	0.0177966990262116\\
91	0.0210905945559624\\
96	0.0236578611108866\\
101	0.0256066700495011\\
106	0.0270454756266691\\
111	0.0280668980077422\\
116	0.0287482287888687\\
121	0.0291584830187542\\
126	0.0293573280197496\\
131	0.029394330690665\\
136	0.0293105296950229\\
141	0.029138222171552\\
146	0.0289018650293172\\
151	0.028620587892192\\
156	0.0283093596728043\\
161	0.0279793199127716\\
166	0.0276387458844459\\
171	0.0272940074540093\\
176	0.0269496585199132\\
181	0.0266083574510287\\
186	0.0262715050101843\\
191	0.0259400858065167\\
196	0.0256149548127438\\
201	0.0252967832696634\\
206	0.0249860989359524\\
211	0.0246834100439347\\
216	0.0243892731218274\\
221	0.0241042657591958\\
226	0.0238288501027112\\
231	0.0235631945662402\\
236	0.0233070932837696\\
241	0.0230600187148861\\
246	0.0228212194207388\\
251	0.0225897906810235\\
256	0.0223647305564645\\
261	0.0221450189626626\\
266	0.0219297113349246\\
271	0.0217180127644552\\
276	0.0215093172080254\\
281	0.0213032195778147\\
286	0.0210995129045022\\
291	0.0208981713399354\\
296	0.0206993105002991\\
301	0.0205031244700737\\
306	0.0203098165840289\\
311	0.0201195480644812\\
316	0.0199324171202285\\
321	0.0197484648424011\\
326	0.0195676977613384\\
331	0.0193901174077815\\
336	0.0192157469727041\\
341	0.0190446461885353\\
346	0.0188769106893134\\
351	0.0187126581333659\\
356	0.0185520066885062\\
361	0.0183950518033241\\
366	0.0182418460427172\\
371	0.0180923855616912\\
376	0.017946605768409\\
381	0.0178043869955247\\
386	0.0176655681970412\\
391	0.0175299642166248\\
396	0.0173973817858115\\
401	0.0172676311881569\\
406	0.0171405330616469\\
411	0.0170159216811181\\
416	0.0168936465984328\\
421	0.0167735738518847\\
426	0.0166555868248673\\
431	0.0165395861008991\\
436	0.0164254877188083\\
441	0.01631321983531\\
446	0.0162027184436029\\
451	0.0160939231121957\\
456	0.0159867736210074\\
461	0.0158812080103343\\
466	0.0157771621203882\\
471	0.0156745703437212\\
476	0.015573367123513\\
481	0.0154734887129903\\
486	0.0153748748170006\\
491	0.0152774698938156\\
496	0.0151812240383313\\
501	0.0150860934656186\\
506	0.0149920406610203\\
511	0.0148990342767129\\
516	0.014807048850667\\
521	0.0147160644145283\\
526	0.0146260660451558\\
531	0.0145370434005428\\
536	0.0144489902637298\\
541	0.0143619040994017\\
546	0.0142757856138347\\
551	0.0141906383030188\\
556	0.0141064679822946\\
561	0.0140232823081272\\
566	0.0139410903207476\\
571	0.0138599020447376\\
576	0.0137797281782851\\
581	0.01370057988211\\
586	0.0136224686560786\\
591	0.0135454062738568\\
596	0.0134694047390004\\
601	0.0133944762274614\\
606	0.0133206329856815\\
611	0.013247887156299\\
616	0.0131762505057432\\
621	0.0131057340369928\\
626	0.0130363474899697\\
631	0.0129680987612559\\
636	0.0129009933061182\\
641	0.0128350336052901\\
646	0.0127702187759935\\
651	0.0127065443759224\\
656	0.0126440024005668\\
661	0.0125825814262209\\
666	0.0125222668233761\\
671	0.0124630409670832\\
676	0.0124048833977394\\
681	0.012347770925175\\
686	0.0122916777045557\\
691	0.0122365753345141\\
696	0.0121824330303992\\
701	0.0121292179088159\\
706	0.0120768953913399\\
711	0.0120254297050867\\
716	0.0119747844350498\\
721	0.0119249230746659\\
726	0.0118758095259335\\
731	0.0118274085150682\\
736	0.0117796859073249\\
741	0.01173260892029\\
746	0.0116861462463204\\
751	0.011640268100562\\
756	0.0115949462128182\\
761	0.0115501537805264\\
766	0.0115058653979553\\
771	0.0114620569738323\\
776	0.0114187056471313\\
781	0.0113757897083114\\
786	0.011333288531012\\
791	0.0112911825173196\\
796	0.0112494530577694\\
801	0.0112080825053261\\
806	0.0111670541614064\\
811	0.0111263522706392\\
816	0.0110859620207295\\
821	0.0110458695434141\\
826	0.0110060619130918\\
831	0.010966527140347\\
836	0.0109272541583699\\
841	0.0108882328014428\\
846	0.0108494537756062\\
851	0.0108109086224876\\
856	0.0107725896781584\\
861	0.0107344900292461\\
866	0.0106966034692301\\
871	0.0106589244574545\\
876	0.0106214480836379\\
881	0.0105841700395738\\
886	0.0105470865991428\\
891	0.0105101946066537\\
896	0.0104734914721145\\
901	0.0104369751713504\\
906	0.0104006442477359\\
911	0.0103644978119556\\
916	0.0103285355362144\\
921	0.0102927576393195\\
926	0.0102571648601298\\
931	0.0102217584174202\\
936	0.0101865399554929\\
941	0.010151511476212\\
946	0.0101166752591185\\
951	0.010082033772601\\
956	0.0100475895800481\\
961	0.010013345245374\\
966	0.00997930324256369\\
971	0.0099454658738124\\
976	0.00991183519971639\\
981	0.0098784129843702\\
986	0.00984520065654337\\
991	0.00981219928688648\\
996	0.00977940958016255\\
};
\addlegendentry{\algo}

\addplot [color=blue]
  table[row sep=crcr]{%
1	0.148936627077608\\
6	0.148930803603413\\
11	0.148929582045431\\
16	0.149056230881783\\
21	16.9896187178999\\
26	48.7489031347188\\
31	22.9398964994392\\
36	3.72004221965904\\
41	1.80409306775725\\
46	0.688677650749411\\
51	0.48569734308718\\
56	0.312498399327862\\
61	0.250625227886008\\
66	0.201089307721436\\
71	0.171905488073625\\
76	0.151712039133411\\
81	0.139820697077972\\
86	0.133173960938727\\
91	0.129193792349184\\
96	0.127243315092799\\
101	0.12559232247866\\
106	0.124490775170982\\
111	0.123187013778321\\
116	0.121592133617653\\
121	0.119564423686373\\
126	0.116856607643956\\
131	0.113760997279624\\
136	0.109974810459241\\
141	0.105887480210916\\
146	0.101245938221153\\
151	0.0963729574723876\\
156	0.0911681378327043\\
161	0.0858210893718911\\
166	0.0804625328134898\\
171	0.0750392646680477\\
176	0.0698917844950533\\
181	0.0647112371037514\\
186	0.0600579340402222\\
191	0.0554052703331258\\
196	0.0514739682468323\\
201	0.0475655906959224\\
206	0.0444935613130201\\
211	0.0414404527490168\\
216	0.03924454639063\\
221	0.0370638355835304\\
226	0.0357332865263883\\
231	0.0344122991512428\\
236	0.0338659430345318\\
241	0.0333069713431586\\
246	0.0334267628207882\\
251	0.0335192846582592\\
256	0.0341753773590829\\
261	0.0347720159763836\\
266	0.0357754689564492\\
271	0.0366821258210173\\
276	0.0378492100302698\\
281	0.0389006496126828\\
286	0.0400931399681071\\
291	0.0411655633759714\\
296	0.0422996350991967\\
301	0.0433154810264668\\
306	0.0443208880569843\\
311	0.0452039656800115\\
316	0.0460124800630831\\
321	0.0466999483602678\\
326	0.0472679755779563\\
331	0.0477229972369742\\
336	0.0480347441558766\\
341	0.0482460932336868\\
346	0.0483079441436693\\
351	0.0482836992113693\\
356	0.0481163248778276\\
361	0.0478770328289179\\
366	0.0475097996609883\\
371	0.0470841230080808\\
376	0.0465521818363495\\
381	0.0459744270437842\\
386	0.0453166900223939\\
391	0.0446246963753407\\
396	0.0438815605055461\\
401	0.0431142909699093\\
406	0.0423250972683458\\
411	0.041520121434674\\
416	0.0407208251725675\\
421	0.0399121799104979\\
426	0.0391337175748266\\
431	0.038350543014526\\
436	0.0376181158944929\\
441	0.0368841316406182\\
446	0.0362172185545088\\
451	0.0355507649458288\\
456	0.0349634756301554\\
461	0.034377804371975\\
466	0.0338793434525824\\
471	0.0333830130974366\\
476	0.0329781578783798\\
481	0.0325754030488837\\
486	0.0322648288530741\\
491	0.0319557484808083\\
496	0.0317361527729846\\
501	0.0315167991894947\\
506	0.0313811022146216\\
511	0.0312438026862398\\
516	0.0311817593619026\\
521	0.0311158855155367\\
526	0.0311151106167518\\
531	0.0311081612690872\\
536	0.031155266314537\\
541	0.0311939583993356\\
546	0.0312755289155198\\
551	0.0313467087323894\\
556	0.0314500561490331\\
561	0.0315414062512803\\
566	0.0316551138772335\\
571	0.0317556457867406\\
576	0.0318698725521788\\
581	0.0319701528034335\\
586	0.0320766606055995\\
591	0.0321687749441875\\
596	0.0322607607374713\\
601	0.03233813772581\\
606	0.0324100355698389\\
611	0.0324672878451124\\
616	0.0325146682863751\\
621	0.0325475420269026\\
626	0.0325671361156394\\
631	0.0325725673056081\\
636	0.0325623685390802\\
641	0.0325385911435666\\
646	0.0324979645218174\\
651	0.0324446035005865\\
656	0.0323743368446027\\
661	0.0322924288271311\\
666	0.0321946813150826\\
671	0.0320865882157584\\
676	0.0319647212647798\\
681	0.0318339336296452\\
686	0.0316922433901593\\
691	0.0315431010057451\\
696	0.0313864960687963\\
701	0.0312238705670662\\
706	0.0310575437791484\\
711	0.0308865271000518\\
716	0.0307156590859274\\
721	0.0305412876832103\\
726	0.0303708000874097\\
731	0.0301978254479517\\
736	0.0300321820929476\\
741	0.0298648812436584\\
746	0.0297079235586063\\
751	0.0295499367626222\\
756	0.0294047423886415\\
761	0.029258932833979\\
766	0.0291277009342144\\
771	0.0289960464723088\\
776	0.0288800277344237\\
781	0.0287635647516628\\
786	0.0286630563621737\\
791	0.0285618926529144\\
796	0.0284763073748691\\
801	0.0283897074168199\\
806	0.0283177097394688\\
811	0.0282442410308067\\
816	0.0281839311103057\\
821	0.0281216508974763\\
826	0.0280707721974522\\
831	0.0280174315281235\\
836	0.0279735796380676\\
841	0.0279268244364903\\
846	0.027887639573667\\
851	0.0278451933879242\\
856	0.0278085243762471\\
861	0.0277683417371114\\
866	0.0277323726024085\\
871	0.0276927542034643\\
876	0.0276560859828784\\
881	0.0276157481195999\\
886	0.0275774302593217\\
891	0.0275355236110455\\
896	0.0274950338641411\\
901	0.0274511125264073\\
906	0.027408291956193\\
911	0.0273622419170051\\
916	0.027317200043969\\
921	0.0272691439357128\\
926	0.0272221543707671\\
931	0.0271723527055912\\
936	0.0271237590047649\\
941	0.0270725254746166\\
946	0.0270226706169082\\
951	0.0269703117132134\\
956	0.0269194955913939\\
961	0.0268662762421723\\
966	0.0268147377945427\\
971	0.0267608683770252\\
976	0.0267087849916324\\
981	0.0266544225629379\\
986	0.0266019192904\\
991	0.026547176644344\\
996	0.0264943399796269\\
};
\addlegendentry{AL}

\addplot [color=red]
  table[row sep=crcr]{%
1	0.148936627077608\\
6	0.151283331700312\\
11	0.148708675452214\\
16	0.148014812589377\\
21	0.148296878550312\\
26	0.148857806124306\\
31	0.148002473386871\\
36	0.148284222196277\\
41	0.148841895765388\\
46	0.147860410627038\\
51	0.147759226711285\\
56	0.147705385897189\\
61	0.14754605404547\\
66	0.148010540253487\\
71	0.147679035549179\\
76	0.147394399095865\\
81	0.147660843221119\\
86	0.147544972631619\\
91	0.147711622508793\\
96	0.147717539729069\\
101	0.148080137926952\\
106	0.147080165492773\\
111	0.146844682845171\\
116	0.145831984510789\\
121	0.144506156965444\\
126	0.14317499585704\\
131	0.14187792502151\\
136	0.140572134283243\\
141	0.139224024829043\\
146	0.137809592589389\\
151	0.136313187439759\\
156	0.134725679500527\\
161	0.133043109841137\\
166	0.13126602408486\\
171	0.129399218819276\\
176	0.127451587071505\\
181	0.125435810192641\\
186	0.123367723678691\\
191	0.121265305227887\\
196	0.119147403427788\\
201	0.117032477785277\\
206	0.114937631822046\\
211	0.112878050209383\\
216	0.110866745220857\\
221	0.108914459803011\\
226	0.10702965645364\\
231	0.105218600124791\\
236	0.10348554765833\\
241	0.101833020578073\\
246	0.100262112090838\\
251	0.0987727783937449\\
256	0.0973640804460123\\
261	0.0960343626527578\\
266	0.0947813715676936\\
271	0.0936023282872535\\
276	0.0924939732796203\\
281	0.0914526032975775\\
286	0.0904741176988934\\
291	0.0895540865054589\\
296	0.0886878455906577\\
301	0.0878706167352371\\
306	0.0870976435921119\\
311	0.086364330398874\\
316	0.0856663694194378\\
321	0.084999845401791\\
326	0.084361309737745\\
331	0.0837478220313534\\
336	0.0831569611181768\\
341	0.0825868104536478\\
346	0.0820359240426938\\
351	0.0815032790044974\\
356	0.0809882199586701\\
361	0.080490399188439\\
366	0.0800097153711909\\
371	0.0795462527991405\\
376	0.0791002225269782\\
381	0.0786719067457743\\
386	0.0782616077714796\\
391	0.077869603186048\\
396	0.0774961087077699\\
401	0.0771412501649567\\
406	0.0768050454496065\\
411	0.0764873965719941\\
416	0.0761880910436525\\
421	0.0759068109599963\\
426	0.0756431475193697\\
431	0.0753966184453783\\
436	0.0751666859309222\\
441	0.0749527732444998\\
446	0.0747542788831528\\
451	0.0745705879208597\\
456	0.0744010807972928\\
461	0.0742451401023411\\
466	0.0741021559214735\\
471	0.0739715300991118\\
476	0.0738526794883071\\
481	0.0737450380269772\\
486	0.0736480574077609\\
491	0.0735612062154477\\
496	0.0734839676484014\\
501	0.0734158362299552\\
506	0.073356314151679\\
511	0.073304907995432\\
516	0.0732611265201614\\
521	0.0732244799892647\\
526	0.0731944812116586\\
531	0.0731706481533848\\
536	0.0731525077207372\\
541	0.0731396001700804\\
546	0.07313148357807\\
551	0.0731277378872545\\
556	0.073127968185055\\
561	0.0731318070306712\\
566	0.0731389157787382\\
571	0.0731489849417288\\
576	0.0731617336856965\\
581	0.0731769085773459\\
586	0.0731942817098996\\
591	0.0732136483429983\\
596	0.0732348242018926\\
601	0.0732576425925585\\
606	0.0732819514961669\\
611	0.0733076108033789\\
616	0.0733344898320522\\
621	0.0733624652420045\\
626	0.0733914194196859\\
631	0.0734212393601857\\
636	0.0734518160292189\\
641	0.0734830441492608\\
646	0.0735148223260575\\
651	0.0735470534157818\\
656	0.0735796450298684\\
661	0.0736125100822258\\
666	0.0736455672987628\\
671	0.0736787416299161\\
676	0.0737119645283864\\
681	0.0737451740750293\\
686	0.0737783149530958\\
691	0.0738113382840666\\
696	0.0738442013470311\\
701	0.0738768672077119\\
706	0.073909304284392\\
711	0.0739414858764235\\
716	0.0739733896778075\\
721	0.074004997294287\\
726	0.0740362937782946\\
731	0.0740672671918657\\
736	0.0740979082042946\\
741	0.0741282097280068\\
746	0.0741581665940859\\
751	0.0741877752667374\\
756	0.0742170335950781\\
761	0.0742459405994167\\
766	0.0742744962889019\\
771	0.0743027015071882\\
776	0.0743305578024845\\
781	0.0743580673187458\\
786	0.0743852327046438\\
791	0.0744120570374471\\
796	0.074438543759073\\
801	0.0744646966218501\\
806	0.0744905196421161\\
811	0.074516017059664\\
816	0.0745411933017066\\
821	0.0745660529501735\\
826	0.0745906007114664\\
831	0.0746148413879333\\
836	0.0746387798507354\\
841	0.0746624210137597\\
846	0.0746857698085757\\
851	0.074708831160356\\
856	0.0747316099649931\\
861	0.0747541110675016\\
866	0.0747763392420017\\
871	0.0747982991734494\\
876	0.0748199954414559\\
881	0.0748414325062094\\
886	0.0748626146969673\\
891	0.0748835462028909\\
896	0.0749042310665995\\
901	0.0749246731802891\\
906	0.0749448762843915\\
911	0.0749648439688144\\
916	0.0749845796763786\\
921	0.0750040867084862\\
926	0.0750233682326179\\
931	0.075042427291569\\
936	0.0750612668138965\\
941	0.0750798896255033\\
946	0.0750982984618996\\
951	0.0751164959808529\\
956	0.0751344847751568\\
961	0.0751522673851225\\
966	0.0751698463106639\\
971	0.0751872240225263\\
976	0.0752044029725779\\
981	0.0752213856029118\\
986	0.0752381743535648\\
991	0.0752547716687697\\
996	0.0752711800015998\\
};
\addlegendentry{reduced \algo}

\end{axis}
\end{tikzpicture}%

%% file: media/gnormFinal.tikz
%
%
\begin{tikzpicture}

\begin{axis}[%
width=0.951\figurewidth,
height=\figureheight,
at={(0\figurewidth,0\figureheight)},
scale only axis,
xmin=0,
xmax=500,
xlabel style={font=\color{white!15!black}},
xlabel={$k$},
ymode=log,
ymin=1e-09,
ymax=0.01,
yminorticks=true,
ylabel style={font=\color{white!15!black}},
ylabel={$\text{rms}(h(x_k))$},
axis background/.style={fill=white},
legend style={legend cell align=left, at={(0.55,0.95)},anchor=west,align=left, draw=white!15!black},
]
\addplot [color=black]
  table[row sep=crcr]{%
1	0\\
6	0.00157647918458071\\
11	0.00123249936147852\\
16	0.000982470797344102\\
21	0.000788392957847463\\
26	0.000630625848403623\\
31	0.000498650516596113\\
36	0.000391864177736385\\
41	0.000306743673261504\\
46	0.000239519046497488\\
51	0.000186835887851158\\
56	0.000145616558288575\\
61	0.000113438130792073\\
66	8.83714221783618e-05\\
71	6.88421935659028e-05\\
76	5.36075595231446e-05\\
81	4.17435502204482e-05\\
86	3.24957380451271e-05\\
91	2.52772951192385e-05\\
96	1.96532389382437e-05\\
101	1.52745858896642e-05\\
106	1.18673615956986e-05\\
111	9.2176279556293e-06\\
116	7.15776203306625e-06\\
121	5.55698245513243e-06\\
126	4.31376384398271e-06\\
131	3.34923312816885e-06\\
136	2.60155010030765e-06\\
141	2.02217391062804e-06\\
146	1.57333781704408e-06\\
151	1.22576077871289e-06\\
156	9.56691563930765e-07\\
161	7.48487065891352e-07\\
166	5.87529244311952e-07\\
171	4.63247326044641e-07\\
176	3.67330210583864e-07\\
181	2.93304471621994e-07\\
186	2.36278598703403e-07\\
191	1.92544165187453e-07\\
196	1.59172606321664e-07\\
201	1.33830422617198e-07\\
206	1.14698614579489e-07\\
211	1.00348413968909e-07\\
216	8.96138183643705e-08\\
221	8.14953399455487e-08\\
226	7.5119541140285e-08\\
231	6.97942793155688e-08\\
236	6.5077176753792e-08\\
241	6.07467889463123e-08\\
246	5.67064993767055e-08\\
251	5.29180292859031e-08\\
256	4.93950363217132e-08\\
261	4.62038162427991e-08\\
266	4.34339867867807e-08\\
271	4.11581978005269e-08\\
276	3.94069313076211e-08\\
281	3.81627372359606e-08\\
286	3.73615616948912e-08\\
291	3.68925510190655e-08\\
296	3.66060869331257e-08\\
301	3.63436656559552e-08\\
306	3.59815971792896e-08\\
311	3.54614792737828e-08\\
316	3.47894987098828e-08\\
321	3.40130265858726e-08\\
326	3.31943295411264e-08\\
331	3.23906037347361e-08\\
336	3.16401593222526e-08\\
341	3.09553074498141e-08\\
346	3.03233362456795e-08\\
351	2.97141337557977e-08\\
356	2.9090627279027e-08\\
361	2.8418502141725e-08\\
366	2.76735370257171e-08\\
371	2.68463729220948e-08\\
376	2.59445364499775e-08\\
381	2.49908968198416e-08\\
386	2.40182562596288e-08\\
391	2.30615916708259e-08\\
396	2.21507615688447e-08\\
401	2.13060921715551e-08\\
406	2.05376830881944e-08\\
411	1.98475559067078e-08\\
416	1.92327047911476e-08\\
421	1.86872362293588e-08\\
426	1.82029900200041e-08\\
431	1.77693533343933e-08\\
436	1.737343053357e-08\\
441	1.70012131226284e-08\\
446	1.66395494025749e-08\\
451	1.62781872625546e-08\\
456	1.59111652932257e-08\\
461	1.55371802122346e-08\\
466	1.51589712354062e-08\\
471	1.47820441293526e-08\\
476	1.44131612693978e-08\\
481	1.40589791847132e-08\\
486	1.37250764785773e-08\\
491	1.34154469453431e-08\\
496	1.31323942909213e-08\\
501	1.28766926825486e-08\\
506	1.264787192339e-08\\
511	1.24445228676551e-08\\
516	1.22645664106303e-08\\
521	1.21054665482271e-08\\
526	1.19643885873912e-08\\
531	1.18383120211704e-08\\
536	1.17241112114941e-08\\
541	1.16186210454589e-08\\
546	1.15187090411418e-08\\
551	1.14213741814847e-08\\
556	1.13238811949579e-08\\
561	1.1223917755069e-08\\
566	1.11197408025932e-08\\
571	1.1010269413615e-08\\
576	1.08950918571823e-08\\
581	1.07743805603265e-08\\
586	1.06487379036135e-08\\
591	1.05190130338033e-08\\
596	1.0386126692525e-08\\
601	1.02509215964707e-08\\
606	1.0114035967783e-08\\
611	9.97579209530729e-09\\
616	9.83610430222674e-09\\
621	9.69442969072899e-09\\
626	9.54979461009973e-09\\
631	9.40091846264492e-09\\
636	9.24642601810973e-09\\
641	9.08510215348198e-09\\
646	8.91611830371587e-09\\
651	8.73916379947493e-09\\
656	8.55445010200487e-09\\
661	8.36260522691756e-09\\
666	8.16451387942818e-09\\
671	7.96116812233993e-09\\
676	7.75357528727016e-09\\
681	7.54273863565529e-09\\
686	7.32969804331592e-09\\
691	7.11560012470677e-09\\
696	6.90176114010044e-09\\
701	6.68969058183579e-09\\
706	6.48105703418949e-09\\
711	6.27759690761892e-09\\
716	6.0809850522867e-09\\
721	5.89269737931504e-09\\
726	5.71389619426079e-09\\
731	5.54536013853664e-09\\
736	5.38746673379704e-09\\
741	5.24022264068342e-09\\
746	5.10332817879861e-09\\
751	4.97625971426938e-09\\
756	4.8583549517252e-09\\
761	4.74889044929124e-09\\
766	4.64714510349796e-09\\
771	4.55244724481394e-09\\
776	4.46420566326597e-09\\
781	4.38192603619005e-09\\
786	4.3052147437025e-09\\
791	4.23377223491212e-09\\
796	4.1673781606783e-09\\
801	4.10587064957477e-09\\
806	4.04912246727618e-09\\
811	3.99701691389385e-09\\
816	3.94942614969523e-09\\
821	3.90619425911339e-09\\
826	3.86712654234386e-09\\
831	3.8319857092702e-09\\
836	3.80049463801303e-09\\
841	3.77234485038106e-09\\
846	3.74720913758324e-09\\
851	3.72475667014376e-09\\
856	3.70466880850138e-09\\
861	3.68665391728045e-09\\
866	3.6704595045822e-09\\
871	3.65588031869365e-09\\
876	3.64276130906155e-09\\
881	3.63099475367814e-09\\
886	3.62051158238633e-09\\
891	3.6112676009278e-09\\
896	3.60322606322366e-09\\
901	3.59633867664911e-09\\
906	3.59052729698437e-09\\
911	3.5856687205457e-09\\
916	3.58158427149554e-09\\
921	3.57803555537351e-09\\
926	3.57472682140361e-09\\
931	3.57131354589395e-09\\
936	3.56741647016981e-09\\
941	3.5626396248674e-09\\
946	3.5565907095894e-09\\
951	3.54890203569546e-09\\
956	3.53925023603595e-09\\
961	3.52737312574889e-09\\
966	3.51308225614372e-09\\
971	3.49627027078017e-09\\
976	3.47691245752083e-09\\
981	3.4550627191954e-09\\
986	3.43084454889839e-09\\
991	3.40443826525447e-09\\
996	3.37606598410696e-09\\
};
\addlegendentry{\algo}

\addplot [color=blue]
  table[row sep=crcr]{%
1	8.14222322115732e-05\\
6	8.14241932828419e-05\\
11	8.14245880575772e-05\\
16	8.13770199893546e-05\\
21	0.00341410087020828\\
26	0.00654773208269618\\
31	0.00448089285845348\\
36	0.00163290949840723\\
41	0.00108056302901898\\
46	0.000606473161812175\\
51	0.000410771887846056\\
56	0.00022976174589438\\
61	0.000225618455784966\\
66	0.000180586603731475\\
71	0.000164951891377301\\
76	0.000127463261839918\\
81	0.000124965164185557\\
86	0.00011043318704013\\
91	0.000110090384331227\\
96	9.86027025300204e-05\\
101	9.40235688001214e-05\\
106	8.34506948771836e-05\\
111	7.87276691012409e-05\\
116	7.16870871660451e-05\\
121	6.84199590430602e-05\\
126	6.34603376897511e-05\\
131	6.11396910223016e-05\\
136	5.74130829836357e-05\\
141	5.58068824517454e-05\\
146	5.27353996180224e-05\\
151	5.13845775730192e-05\\
156	4.85512706890178e-05\\
161	4.74272084187852e-05\\
166	4.50275190401628e-05\\
171	4.43051624382474e-05\\
176	4.22245999132281e-05\\
181	4.16701896860701e-05\\
186	3.97870312684233e-05\\
191	3.94623655388956e-05\\
196	3.79449644744194e-05\\
201	3.7898325561359e-05\\
206	3.67115408590074e-05\\
211	3.68554758656933e-05\\
216	3.5809706533104e-05\\
221	3.5917031626612e-05\\
226	3.49563162985413e-05\\
231	3.51060790500765e-05\\
236	3.41067093229261e-05\\
241	3.39170666958161e-05\\
246	3.2450259975083e-05\\
251	3.17010207376108e-05\\
256	2.98985488738607e-05\\
261	2.88385437108121e-05\\
266	2.70683255310486e-05\\
271	2.61239987578791e-05\\
276	2.47508923077719e-05\\
281	2.41115286178203e-05\\
286	2.28295959358636e-05\\
291	2.20620658911862e-05\\
296	2.07234804176658e-05\\
301	1.99218364611398e-05\\
306	1.8729626316553e-05\\
311	1.80114650862267e-05\\
316	1.70107532073519e-05\\
321	1.64400549498558e-05\\
326	1.567156090996e-05\\
331	1.52836154726164e-05\\
336	1.47295895290255e-05\\
341	1.44895681168166e-05\\
346	1.40777046414918e-05\\
351	1.39181453501973e-05\\
356	1.35789636075604e-05\\
361	1.34505031881673e-05\\
366	1.31471513092178e-05\\
371	1.3029459146084e-05\\
376	1.27469726083087e-05\\
381	1.26320201765158e-05\\
386	1.23585135915861e-05\\
391	1.22369788946896e-05\\
396	1.19582233091298e-05\\
401	1.18191958716288e-05\\
406	1.15236443004715e-05\\
411	1.13607114643352e-05\\
416	1.10457185281741e-05\\
421	1.08621520397447e-05\\
426	1.05367400745296e-05\\
431	1.03465448668247e-05\\
436	1.00291998446609e-05\\
441	9.85321160107232e-06\\
446	9.56567707787652e-06\\
451	9.42482530634587e-06\\
456	9.18524904499556e-06\\
461	9.09414158526964e-06\\
466	8.91127004021951e-06\\
471	8.872582569538e-06\\
476	8.74055411984423e-06\\
481	8.74106668150474e-06\\
486	8.63823345371955e-06\\
491	8.65144020668877e-06\\
496	8.54801165445356e-06\\
501	8.54478029416193e-06\\
506	8.41549316514827e-06\\
511	8.37554707588511e-06\\
516	8.20812760018397e-06\\
521	8.1250298711014e-06\\
526	7.92087338149842e-06\\
531	7.80043261706118e-06\\
536	7.57088967645159e-06\\
541	7.42718457586154e-06\\
546	7.1895539497497e-06\\
551	7.04118838850641e-06\\
556	6.81485879398255e-06\\
561	6.68051683905663e-06\\
566	6.48155790075044e-06\\
571	6.37402199556772e-06\\
576	6.20942667088284e-06\\
581	6.13107337576758e-06\\
586	5.99697527849479e-06\\
591	5.94054811438378e-06\\
596	5.82604432933201e-06\\
601	5.77992732881808e-06\\
606	5.67310969349425e-06\\
611	5.62681472291603e-06\\
616	5.51896868432505e-06\\
621	5.46601015026792e-06\\
626	5.35294843493179e-06\\
631	5.29125030717687e-06\\
636	5.17283800513716e-06\\
641	5.1038685923013e-06\\
646	4.9827656981845e-06\\
651	4.91013167376457e-06\\
656	4.79019528822443e-06\\
661	4.71803353783657e-06\\
666	4.60276129576093e-06\\
671	4.53438030218666e-06\\
676	4.42594796623353e-06\\
681	4.36322935537794e-06\\
686	4.26246709081093e-06\\
691	4.2061602566338e-06\\
696	4.1132631642476e-06\\
701	4.06378468177255e-06\\
706	3.97915962108174e-06\\
711	3.93727700864464e-06\\
716	3.86192179687228e-06\\
721	3.82881581154598e-06\\
726	3.76389878509353e-06\\
731	3.7404497561345e-06\\
736	3.68623418984235e-06\\
741	3.67191121520262e-06\\
746	3.62671866332392e-06\\
751	3.61888536639681e-06\\
756	3.57894207036582e-06\\
761	3.57322319863321e-06\\
766	3.5336420764985e-06\\
771	3.52521325182187e-06\\
776	3.48151293604072e-06\\
781	3.46652781833624e-06\\
786	3.41581475919186e-06\\
791	3.39225025862456e-06\\
796	3.33361197317707e-06\\
801	3.30134268253283e-06\\
806	3.2355597780706e-06\\
811	3.1959115323233e-06\\
816	3.12486679686294e-06\\
821	3.08002860577601e-06\\
826	3.00617108731828e-06\\
831	2.95875675317134e-06\\
836	2.88481496824034e-06\\
841	2.83768077017314e-06\\
846	2.76660441793746e-06\\
851	2.72281666823744e-06\\
856	2.65773144412343e-06\\
861	2.62042473939736e-06\\
866	2.56431897831135e-06\\
871	2.53620995002975e-06\\
876	2.49123193939739e-06\\
881	2.47380978383989e-06\\
886	2.44040558081051e-06\\
891	2.43318070712572e-06\\
896	2.40959350387021e-06\\
901	2.4099288976137e-06\\
906	2.39248555151856e-06\\
911	2.39623688448207e-06\\
916	2.38035968096213e-06\\
921	2.38303308257277e-06\\
926	2.36447154536785e-06\\
931	2.36234472887128e-06\\
936	2.3380932585148e-06\\
941	2.3288902370646e-06\\
946	2.29754852231901e-06\\
951	2.28057297607136e-06\\
956	2.2422113473998e-06\\
961	2.21807783388794e-06\\
966	2.17382426507933e-06\\
971	2.14401032353079e-06\\
976	2.09558712054602e-06\\
981	2.06200472204603e-06\\
986	2.01137585034084e-06\\
991	1.97608449452683e-06\\
996	1.92528102233808e-06\\
};
\addlegendentry{AL}

\addplot [color=red]
  table[row sep=crcr]{%
1	0\\
6	0.000100527894253087\\
11	9.02201687063572e-05\\
16	8.28610073113188e-05\\
21	8.07897512672054e-05\\
26	8.00656693428086e-05\\
31	8.03219526202546e-05\\
36	8.36907794066898e-05\\
41	7.9924663485585e-05\\
46	8.22316113981195e-05\\
51	8.04544140899932e-05\\
56	8.050525266801e-05\\
61	8.07447760475883e-05\\
66	8.05728410271541e-05\\
71	8.23981738365558e-05\\
76	8.10792774976084e-05\\
81	8.52007723858638e-05\\
86	8.10915861240303e-05\\
91	8.007781617471e-05\\
96	8.33687048494412e-05\\
101	7.9836897667944e-05\\
106	8.00794374284539e-05\\
111	7.8894430114971e-05\\
116	7.72182995282106e-05\\
121	7.51752967107848e-05\\
126	7.32071922279051e-05\\
131	7.14813476187087e-05\\
136	6.99355988860026e-05\\
141	6.85309739800273e-05\\
146	6.72408830191495e-05\\
151	6.60461235269322e-05\\
156	6.49323255474668e-05\\
161	6.38884336656521e-05\\
166	6.29057230668397e-05\\
171	6.19771346541473e-05\\
176	6.10968128535797e-05\\
181	6.02597786671768e-05\\
186	5.94617016837331e-05\\
191	5.86987537010163e-05\\
196	5.79675320945978e-05\\
201	5.72650324636472e-05\\
206	5.65886372647967e-05\\
211	5.59360895709834e-05\\
216	5.53054445104195e-05\\
221	5.46950136680207e-05\\
226	5.41033191680135e-05\\
231	5.3529061468417e-05\\
236	5.29710957530672e-05\\
241	5.24284113008056e-05\\
246	5.19001114751977e-05\\
251	5.13853945948027e-05\\
256	5.088353683053e-05\\
261	5.03938780378399e-05\\
266	4.99158108495651e-05\\
271	4.94487728722073e-05\\
276	4.8992241552531e-05\\
281	4.85457311499873e-05\\
286	4.81087911819503e-05\\
291	4.76810056727776e-05\\
296	4.72619925543671e-05\\
301	4.68514026679069e-05\\
306	4.64489180091379e-05\\
311	4.60542491083405e-05\\
316	4.56671316774976e-05\\
321	4.52873228285917e-05\\
326	4.49145972332537e-05\\
331	4.45487435590843e-05\\
336	4.4189561415268e-05\\
341	4.38368589147674e-05\\
346	4.34904508497715e-05\\
351	4.31501574022731e-05\\
356	4.28158032769015e-05\\
361	4.24872171417015e-05\\
366	4.21642312829852e-05\\
371	4.18466814111734e-05\\
376	4.15344065858391e-05\\
381	4.12272492523991e-05\\
386	4.09250553951601e-05\\
391	4.06276748100908e-05\\
396	4.03349614871257e-05\\
401	4.00467740713452e-05\\
406	3.97629763517454e-05\\
411	3.94834377125482e-05\\
416	3.92080334802707e-05\\
421	3.89366451122742e-05\\
426	3.8669160197168e-05\\
431	3.84054722691771e-05\\
436	3.81454804695095e-05\\
441	3.78890891099607e-05\\
446	3.76362072018364e-05\\
451	3.73867480056665e-05\\
456	3.71406286372825e-05\\
461	3.68977697414275e-05\\
466	3.66580952225912e-05\\
471	3.64215320104565e-05\\
476	3.618800983672e-05\\
481	3.59574610085579e-05\\
486	3.57298201777468e-05\\
491	3.55050241170098e-05\\
496	3.52830115227999e-05\\
501	3.50637228636751e-05\\
506	3.48471002866428e-05\\
511	3.46330875824874e-05\\
516	3.44216301993292e-05\\
521	3.42126752844804e-05\\
526	3.40061717309602e-05\\
531	3.38020702068857e-05\\
536	3.36003231527988e-05\\
541	3.3400884740864e-05\\
546	3.32037107986133e-05\\
551	3.30087587060296e-05\\
556	3.28159872774429e-05\\
561	3.2625356639149e-05\\
566	3.24368281110892e-05\\
571	3.22503640976636e-05\\
576	3.20659279899913e-05\\
581	3.18834840803687e-05\\
586	3.1702997489057e-05\\
591	3.1524434103643e-05\\
596	3.13477605313532e-05\\
601	3.11729440646625e-05\\
606	3.09999526600106e-05\\
611	3.08287549285415e-05\\
616	3.06593201368416e-05\\
621	3.04916182147826e-05\\
626	3.03256197671179e-05\\
631	3.0161296085346e-05\\
636	2.9998619156834e-05\\
641	2.98375616687815e-05\\
646	2.96780970055868e-05\\
651	2.9520199239137e-05\\
656	2.93638431123545e-05\\
661	2.92090040170779e-05\\
666	2.90556579677492e-05\\
671	2.89037815725871e-05\\
676	2.87533520038415e-05\\
681	2.86043469685581e-05\\
686	2.84567446809057e-05\\
691	2.83105238368158e-05\\
696	2.81656635913141e-05\\
701	2.80221435386345e-05\\
706	2.78799436950239e-05\\
711	2.7739044483929e-05\\
716	2.7599426723281e-05\\
721	2.74610716144542e-05\\
726	2.73239607325539e-05\\
731	2.71880760177355e-05\\
736	2.70533997672172e-05\\
741	2.69199146278448e-05\\
746	2.67876035889513e-05\\
751	2.66564499754643e-05\\
756	2.65264374410824e-05\\
761	2.63975499615471e-05\\
766	2.62697718279426e-05\\
771	2.61430876399606e-05\\
776	2.60174822992236e-05\\
781	2.5892941002565e-05\\
786	2.57694492353676e-05\\
791	2.5646992764909e-05\\
796	2.5525557633761e-05\\
801	2.54051301532584e-05\\
806	2.52856968970568e-05\\
811	2.51672446947628e-05\\
816	2.50497606257041e-05\\
821	2.493323201282e-05\\
826	2.48176464166808e-05\\
831	2.47029916296643e-05\\
836	2.45892556702923e-05\\
841	2.44764267777255e-05\\
846	2.43644934064175e-05\\
851	2.42534442209528e-05\\
856	2.41432680910399e-05\\
861	2.40339540866686e-05\\
866	2.39254914734298e-05\\
871	2.38178697079983e-05\\
876	2.37110784337559e-05\\
881	2.3605107476551e-05\\
886	2.3499946840607e-05\\
891	2.33955867045454e-05\\
896	2.32920174175315e-05\\
901	2.31892294955352e-05\\
906	2.30872136176792e-05\\
911	2.29859606227045e-05\\
916	2.28854615054974e-05\\
921	2.27857074137223e-05\\
926	2.26866896445276e-05\\
931	2.25883996413111e-05\\
936	2.24908289905659e-05\\
941	2.23939694187905e-05\\
946	2.22978127894606e-05\\
951	2.22023511000591e-05\\
956	2.21075764791673e-05\\
961	2.20134811836137e-05\\
966	2.19200575956834e-05\\
971	2.1827298220387e-05\\
976	2.17351956827734e-05\\
981	2.16437427253219e-05\\
986	2.15529322053787e-05\\
991	2.14627570926562e-05\\
996	2.13732104667934e-05\\
};
\addlegendentry{reduced \algo}

\end{axis}
\end{tikzpicture}%

%% file: conclusion.tex
\vspace{-5pt}
\section{Conclusion}
In this paper, we propose \algo\ a novel infeasible method for optimization on an immersed manifold $\cM$.  An attractive property of \algo\ is that it avoids retractions and only projections on a vector space need to be computed. \algo\ achieve near-optimal  oracle complexities $\cO(1/\varepsilon^2)$ and $\cO(1/\varepsilon^4)$ in the deterministic and stochastic cases, respectively. Various extensions of \algo\ are presented. First, we introduce \redalgo, a computationally friendly version of \algo. Here we only need to compute one projection onto a hyperplane, but at the price of a slightly worse complexity bound. Second, we introduce \geomalgo\, a geometry-based version of \algo, where the projections account for the local Riemannian metric. When specialized to the Stieffel manifold, \geomalgo\ generalizes the \texttt{landing} algorithm \citep{ablin2022fast}. We show that \geomalgo\ enjoys the same oracle complexity as \algo\:  as a result, for Stieffel manifold, we establish  oracle complexity bounds for \texttt{landing}. Numerical experiments illustrate the performance of \algo\ and its variants. 

%% file: appendix_numerics.tex
\section{Numerical experiments}\label{app:numerics}

\subsection{Procrustes problem}

We provide additional numerical experiment on matrices $A \in \bbR^{q \times q},B \in \bbR^{p\times q}$ for $p=1000$ and $q=500$. Plots are presented in Figure~\ref{fig:procrustes_large_scale}. The large scale of the problem add a lot of challenges for all algorithms and notably it affects \redalgo\, convergence for constraints. However, we see that \geomalgo\ again outperforms all the baselines. All experiments for the Procrustes problem are performed in PyTorch on CPU with Intel Core-i7 processor. 

\begin{figure}[h]
    \centering
    \includegraphics[width=0.98\linewidth]{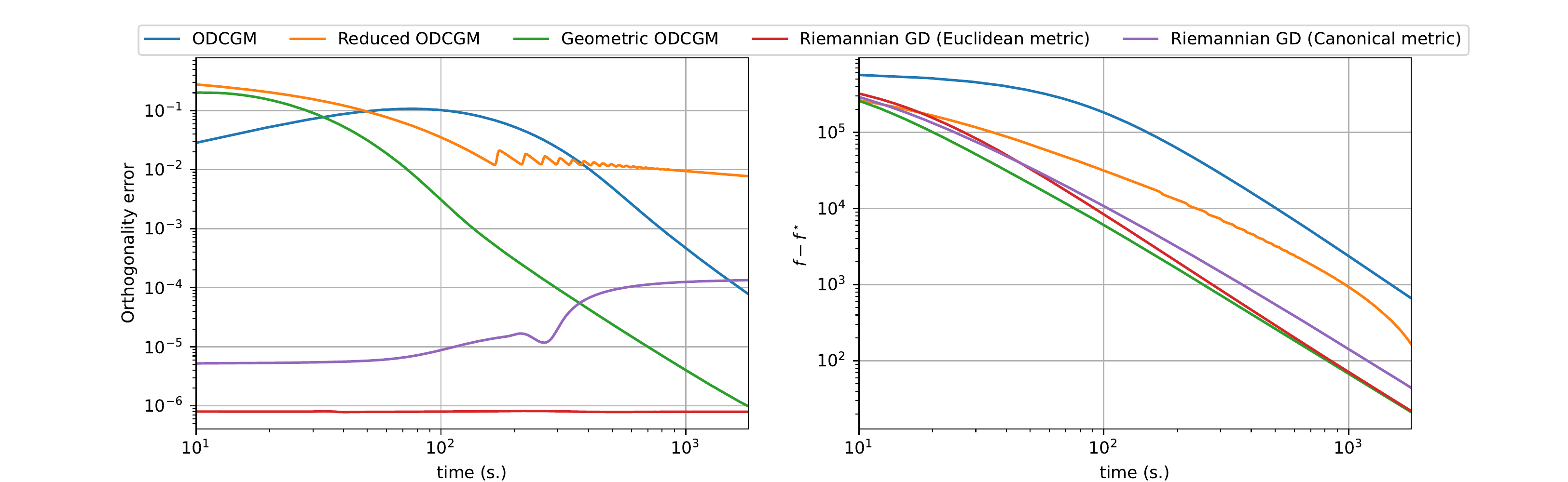}
    \caption{Comparison of \algo\ (blue), \redalgo\ (orange), and \geomalgo\ (green) with Riemannian gradient descent with two different Riemannian metrics (red and purple). The upper plot shows the orthogonality error for $X$, the lower plot shows the convergence of the objective function (averages over 128 seeds).}
    \label{fig:procrustes_large_scale}
\end{figure}

\paragraph{Euclidean projection on tangent space}

Additionally, in the case of optimization over Stiefel manifold $\cM = \{ X \in \bbR^{p \times q}: X^\top X = \cI \}$, we discuss a way of projecting onto $V(x)$ (necessary for \algo) whihc is more efficient than solving a linear system of size $pq$. 

First, we notice that the tangent space can be described as follows (see \cite{gao2022optimization}):
\[
    V(X) = \{ Y \in \bbR^{p \times q} : Y^\top X + X^\top Y = 0 \}\,.
\]
To optimize over this set, we apply the Lagrange multipliers method
\begin{equation}\label{eq:stiefel_valpha_proj}
    \min_{Y \in \bbR^{p \times q}} \frac{1}{2} \norm{Y - U}_2^2, \quad \text{s.t. } Y^\top X + X^\top Y = 0.
\end{equation}
To solve this problem, we start from reparametrization $Y = (X^+)^\top Z$, with $Z$ a skew-symmetric matrix, and where $X^+$ denotes the pseudoinverse. In this way, we obtain:
\begin{equation}\label{eq:stiefel_valpha_proj_skew_sym}
    \min_{Z \in \bbR^{n \times p}} \frac{1}{2} \norm{(X^+)^\top Z - U}_2^2, \quad \text{s.t. } Z^\top + Z = 0\,.
\end{equation}
First order optimality condition implies
$
    X^+ ( (X^+)^\top Z - U) - (\Lambda + \Lambda^\top)= 0 ,
$
where $\Lambda$ are Lagrange multipliers. By properties of pseudoinverse for full column rank matrices we have $X^+ ( (X^+)^\top = (X^\top X)^{-1}$ and thus
\[
    Z = (X^\top X) (\Lambda + \Lambda^\top) + X^\top U.
\]
Next, we have to choose $M$ to satisfy $Z^\top + Z = 0$:
\[
    (X^\top X) (\Lambda + \Lambda^\top) + (\Lambda + \Lambda^\top) (X^\top X) = - (X^\top U + U^\top X)\, .
\]
Since the right-hand side is symmetric, we only need to compute, over $P$, any solution to the following \textit{Sylvester's equation}:
\[
    (X^\top X) P + P(X^\top X) = - 2(X^\top U + U^\top X)
\]
and symmetrize it: $\Lambda = (P + P^\top)/2$. The solution to this system could be easily found using SVD of a matrix $X$. Notice that since $X^\top X \to I$ we may expect that all operations will be numerically stable. The total complexity of is equal to $\cO(pq^2)$.

\subsection{Hanging chain}
We compute the shape of a hanging chain as a numerical example. The chain has length $l=10$ and is divided into $N+1$ segments of equal length $r=10/(N+1)$. Each two segments are connected by a joint and a torsion spring, which models the stiffness of the chain. The torsion spring has a spring constant of $k_\text{s}=100$. The chain is suspended at positions $(0,0)$ and $(9,0)$, and an example with three nodes ($N=3$) is shown in \Cref{fig:chain}. The optimization variables are given by the coordinates $\xi_i=(x_i,y_i)$ of the nodes, $i=1,\dots,N$, and a non-convex distance constraint restricts the length of each segment to $r$. We compute the shape of the chain by minimizing its potential energy, i.e.,
\begin{equation}
\begin{split}
    \min_{(\xi_1,\dots,\xi_N)\in\mathbb{R}^{2N}} &\frac{1}{N^3} \sum_{i=1}^N  \left( \frac{k_\text{s}}{r^4}(\xi_{i-1}-\xi_i)\T (\xi_{i+1}-\xi_i) + y_i \right) \\ \text{s.t.}& \quad \sqrt{(\xi_{k-1}-\xi_k)\T (\xi_{k-1}-\xi_k)}\leq r,~ k=1,2,\dots,N+1,
\end{split}
 \label{eq:numEx_app}
\end{equation}
where $\xi_{0}=(0,0)$ and $\xi_{N+1}=(9,0)$ are the two endpoints. We note that the first term of the objective function contains a discrete approximation to the curvature of the chain that models the spring potential, while the second term corresponds to the gravitational potential. This example is motivated by the fact that it leads to a simple problem formulation that includes non-convex distance constraints, but also allows us to scale $N$ to values of $10^5$ or more. Finally, Euler-Bernoulli beam theory gives us a reasonable initial estimate for the start of the optimization. All calculations are performed in MATLAB on a standard laptop (Dell XPS 15 with an Intel Core-i7 processor, 32 gigabytes of RAM, and a Windows operating system).

We start with a chain of $10,001$ segments, leading to an optimization problem with $20,000$ decision variables and $10,001$ nonlinear constraints. We compare the three algorithms: \algo~ with $A(x) = \alpha (\nabla h(x)^{\top} \nabla h(x))^{-1}$ (denoted simply by \algo~), \redalgo~, and an extended Lagrangian approach. \Cref{fig:chain} (right) shows the initial estimate and the final result as computed by \algo~ (the result of the other algorithms is similar). 
The step size for \algo~ is set to $\gamma_k=T$, where $T=0.1/k_\text{s}$ and $\alpha=0.05/T$; the step size for \redalgo\ is set to  $\gamma_k=T$ for $k\leq 100$ and $\gamma_k=T/\sqrt{k-100}$ for $k > 100$ (the scaling with $1/k_\text{s}$ results from the Hessian of \eqref{eq:numEx_app}). \Cref{fig:objfunconst2} (main text) shows the value of the objective function and the root mean square error of the constraint violation over the course of the optimization.
We find that \algo~ leads to fast convergence in terms of constraint violations and function value, while the convergence of the augmented Lagrangian approach and \redalgo~ is much slower. Moreover, the performance of the augmented Lagrangian is quite sensitive to the initial value of the dual variable, which may even lead to divergence. In contrast, setting the step size of \algo~ and \redalgo~ is very simple. 
\Cref{fig:objfunconst} illustrates that \redalgo~ must be executed with decreasing step size; if a constant step size is used, the constraint violations remain as shown in the left panel, which is also consistent with our theoretical analysis. 
The right panel shows the execution time per iteration of the different algorithms by computing a moving average over past iterations. We conclude that  \redalgo~ and the augmented Lagrangian require only about one-fifth of the time of \algo\ for a single iteration. 
This can be explained by the fact that \algo\ requires the solution of a linear system of size $30,0001 \times 30,0001$ at each iteration (we have exploited parsimony but have not taken into account the special structure of the equality constraints in \eqref{eq:numEx_app}). 
Although  \redalgo~ (and the augmented Lagrangian) have lower execution time per iteration, it also converges much more slowly.

In order to highlight the potential of \redalgo~, we run the same example for $N=200,000$, which results in a large-scale optimization problem with $400,000$ decision variables and $200,001$ non-convex equality constraints. In this case, solving the resulting Karush-Kuhn-Tucker system at every iteration, which is required for \algo~, becomes challenging. However, \redalgo~ can still be applied and requires only about 0.085 seconds for executing a single iteration. The resulting function values and the evolution of the constraint violations are shown in \Cref{fig:large}.

\begin{figure}
\setlength{\figurewidth}{.45\columnwidth}
\setlength{\figureheight}{.125\columnwidth}
\resizebox{75mm}{!}{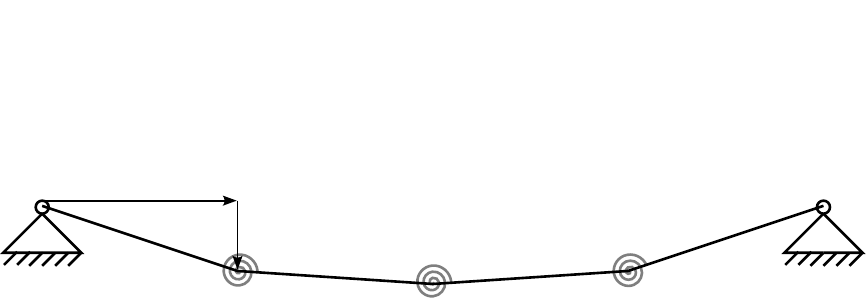} \hspace{-20pt}%
\input{media/hangingChain.tikz}
\caption{The figure shows a sketch of the hanging chain (left), the results arising from optimizing \eqref{eq:numEx_app} (black, right) and the results predicted by the Euler-Bernoulli beam theory (red, right). The predictions from the Euler-Bernoulli beam theory are used as initial guess.}
\label{fig:chain}
\end{figure}

\begin{figure}
\setlength{\figurewidth}{.35\columnwidth}
\setlength{\figureheight}{.2\columnwidth}
\input{media/gnormA.tikz} %
\input{media/exectimeAverage.tikz}
\vspace{-10pt}
\caption{The figure on the left shows that \redalgo~ with a constant step size $\gamma_k=T$ does not converge and leads to significant constraint violations. The figure on the right shows the execution time per iteration of the different algorithms (moving average over past iterations). We note that the curve of the augmented Lagrangian and \redalgo~ are essentially superimposed.}
\label{fig:objfunconst}
\end{figure}
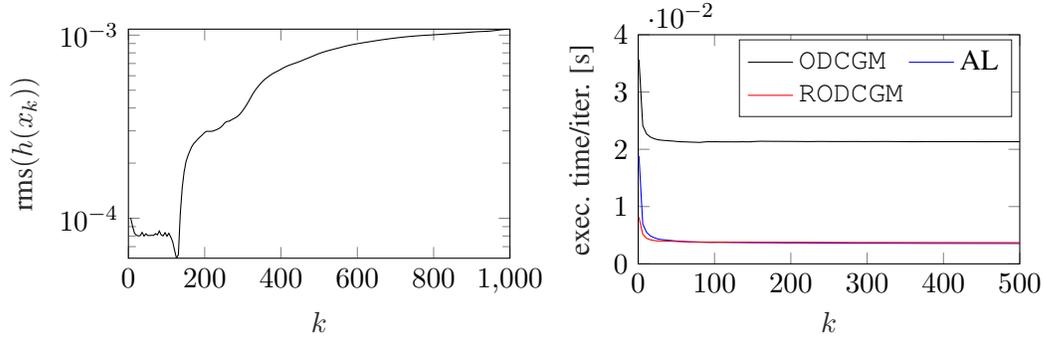


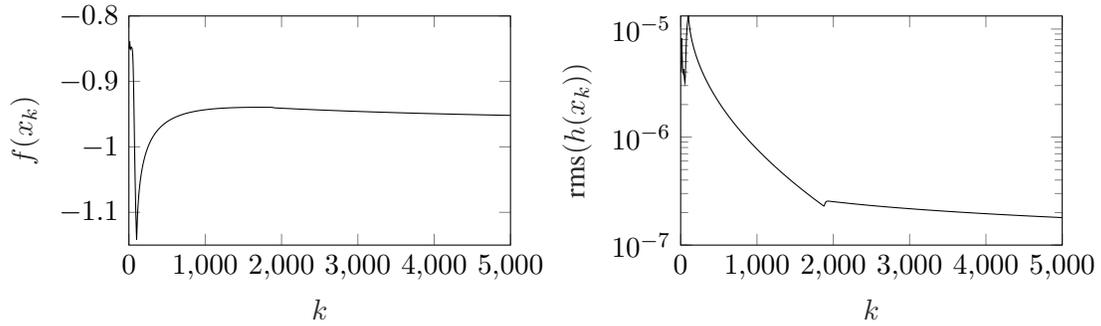
\begin{figure}
\setlength{\figurewidth}{.35\columnwidth}
\setlength{\figureheight}{.2\columnwidth}
\input{media/fvalueLarge.tikz}%
\input{media/gnormLarge.tikz}
\caption{This figure shows the results from applying \redalgo~ with $\gamma_k \sim 1/\sqrt{k}$ for large $k$ to \eqref{eq:numEx_app} with $N=200,000$. This leads to an optimization problem with $400,000$ decision variables and $200,001$ non-convex constraints. The evolution of the function value is shown on the left, whereas the evolution of the constraint violations is shown on the right.}
\label{fig:large}
\end{figure}

%% file: drawingChain.pdf_tex
\begingroup%
  \makeatletter%
  \providecommand\color[2][]{%
    \errmessage{(Inkscape) Color is used for the text in Inkscape, but the package 'color.sty' is not loaded}%
    \renewcommand\color[2][]{}%
  }%
  \providecommand\transparent[1]{%
    \errmessage{(Inkscape) Transparency is used (non-zero) for the text in Inkscape, but the package 'transparent.sty' is not loaded}%
    \renewcommand\transparent[1]{}%
  }%
  \providecommand\rotatebox[2]{#2}%
  \newcommand*\fsize{\dimexpr\f@size pt\relax}%
  \newcommand*\lineheight[1]{\fontsize{\fsize}{#1\fsize}\selectfont}%
  \ifx\svgwidth\undefined%
    \setlength{\unitlength}{249.31066487bp}%
    \ifx\svgscale\undefined%
      \relax%
    \else%
      \setlength{\unitlength}{\unitlength * \real{\svgscale}}%
    \fi%
  \else%
    \setlength{\unitlength}{\svgwidth}%
  \fi%
  \global\let\svgwidth\undefined%
  \global\let\svgscale\undefined%
  \makeatother%
  \begin{picture}(1,0.34381835)%
    \lineheight{1}%
    \setlength\tabcolsep{0pt}%
    \put(0,0){\includegraphics[width=\unitlength,page=1]{drawingChain.pdf}}%
    \put(0.10892166,0.12707191){\color[rgb]{0,0,0}\makebox(0,0)[lt]{\lineheight{1.25}\smash{\begin{tabular}[t]{l}$x_1$\end{tabular}}}}%
    \put(0.28299618,0.07432215){\color[rgb]{0,0,0}\makebox(0,0)[lt]{\lineheight{1.25}\smash{\begin{tabular}[t]{l}$y_1$\end{tabular}}}}%
    \put(0,0){\includegraphics[width=\unitlength,page=2]{drawingChain.pdf}}%
    \put(0.46991705,0.12707191){\color[rgb]{0,0,0}\makebox(0,0)[lt]{\lineheight{1.25}\smash{\begin{tabular}[t]{l}$e_x$\end{tabular}}}}%
    \put(0.05881946,0.31498573){\color[rgb]{0,0,0}\makebox(0,0)[lt]{\lineheight{1.25}\smash{\begin{tabular}[t]{l}$e_y$\end{tabular}}}}%
  \end{picture}%
\endgroup%

%% file: media/gnormA.tikz
%
%
\begin{tikzpicture}

\begin{axis}[%
width=0.951\figurewidth,
height=\figureheight,
at={(0\figurewidth,0\figureheight)},
scale only axis,
xmin=0,
xmax=1000,
xlabel style={font=\color{white!15!black}},
xlabel={$k$},
ymode=log,
ymin=6.04944475842867e-05,
ymax=0.00107789902082252,
yminorticks=true,
ylabel style={font=\color{white!15!black}},
ylabel={$\text{rms}(h(x_k))$},
axis background/.style={fill=white}
]
\addplot [color=black, forget plot]
  table[row sep=crcr]{%
1	0\\
6	0.000100527894253087\\
11	9.02201687063572e-05\\
16	8.28610073113188e-05\\
21	8.07897512672054e-05\\
26	8.00656693428086e-05\\
31	8.03219526202546e-05\\
36	8.36907794066898e-05\\
41	7.9924663485585e-05\\
46	8.22316113981195e-05\\
51	8.04544140899932e-05\\
56	8.050525266801e-05\\
61	8.07447760475883e-05\\
66	8.05728410271541e-05\\
71	8.23981738365558e-05\\
76	8.10792774976084e-05\\
81	8.52007723858638e-05\\
86	8.10915861240303e-05\\
91	8.007781617471e-05\\
96	8.33687048494412e-05\\
101	7.9836897667944e-05\\
106	8.27617752787801e-05\\
111	7.90212297396653e-05\\
116	7.44493985016894e-05\\
121	6.69274119149754e-05\\
126	6.04944475842867e-05\\
131	6.3252284907722e-05\\
136	0.000106054308717524\\
141	0.000146349140211535\\
146	0.000179193894426132\\
151	0.000203876188306075\\
156	0.000219243823381209\\
161	0.000232924175726245\\
166	0.000245821839528317\\
171	0.000255276790835232\\
176	0.0002618904526388\\
181	0.000269094427072533\\
186	0.000275925591726176\\
191	0.000281720049422391\\
196	0.000288422310523213\\
201	0.000295489208091768\\
206	0.000298518635012227\\
211	0.00029815226682191\\
216	0.000298030184352456\\
221	0.000299248037914872\\
226	0.000301625730493635\\
231	0.000304613243450714\\
236	0.000308021354743331\\
241	0.000312435788806174\\
246	0.000319449307311435\\
251	0.000329064869275586\\
256	0.000336010952121934\\
261	0.00033821413560188\\
266	0.000340281975639003\\
271	0.000345664347112523\\
276	0.00034979866552728\\
281	0.00035439532454455\\
286	0.000360044641807162\\
291	0.00036804073454825\\
296	0.000378580526716203\\
301	0.000390842498512269\\
306	0.000404540309266575\\
311	0.000420243863261184\\
316	0.000438113528423406\\
321	0.00045748789672899\\
326	0.000477123990056439\\
331	0.00049552387484862\\
336	0.000511883637668507\\
341	0.000527333291436143\\
346	0.000542439118430806\\
351	0.000556861711761479\\
356	0.000569815400964101\\
361	0.000581145223083535\\
366	0.000591714421809439\\
371	0.000602032454501006\\
376	0.000612033358024262\\
381	0.0006207648296437\\
386	0.000628505134639075\\
391	0.000636159295862795\\
396	0.000644208252663063\\
401	0.000652954208973046\\
406	0.000661986544556824\\
411	0.000670285640045249\\
416	0.000677543513710237\\
421	0.000684040688707937\\
426	0.000690173044468266\\
431	0.000696126879174943\\
436	0.000701882007957443\\
441	0.000707574095637582\\
446	0.000713736047646231\\
451	0.000720618596051307\\
456	0.000727868133793671\\
461	0.000735237751627243\\
466	0.000742509690417571\\
471	0.000749501317094884\\
476	0.000756683407015708\\
481	0.000764108755678512\\
486	0.000771930097751803\\
491	0.00077981764401125\\
496	0.000787537511428728\\
501	0.000794844169168479\\
506	0.00080164012492616\\
511	0.000807845833168892\\
516	0.000813861123253922\\
521	0.000819816400635059\\
526	0.000825350890432398\\
531	0.000830494198804137\\
536	0.000835448290292005\\
541	0.000840667962692266\\
546	0.000845776121688769\\
551	0.000850960352386457\\
556	0.000856327656662895\\
561	0.000861842816408452\\
566	0.00086730323012416\\
571	0.000872558961358554\\
576	0.000877362955580607\\
581	0.000881945515638828\\
586	0.000886298677264657\\
591	0.000890518298337324\\
596	0.000894456723976024\\
601	0.000898193987179577\\
606	0.000901814910696339\\
611	0.000905454917826598\\
616	0.000909085041272326\\
621	0.000912546147061999\\
626	0.000915931500680244\\
631	0.000919423472160521\\
636	0.00092302426742553\\
641	0.000926692899049023\\
646	0.0009303153204646\\
651	0.000933850997069408\\
656	0.000937266298502236\\
661	0.00094060881059894\\
666	0.000943861793151595\\
671	0.000947056764947298\\
676	0.000950173085176018\\
681	0.000953239880525431\\
686	0.000956235933967676\\
691	0.000959191123120714\\
696	0.00096207893559174\\
701	0.00096492408157113\\
706	0.000967693889485267\\
711	0.000970405246023574\\
716	0.000973025849607905\\
721	0.000975567210413372\\
726	0.000978005498598652\\
731	0.000980349222139779\\
736	0.000982586970380645\\
741	0.00098472477070706\\
746	0.000986762623158078\\
751	0.000988700987121399\\
756	0.000990546578667843\\
761	0.000992288962764788\\
766	0.000993943883040909\\
771	0.000995499081618748\\
776	0.000996993506932941\\
781	0.000998423085182317\\
786	0.000999835807923102\\
791	0.00100122580135037\\
796	0.00100263496883972\\
801	0.00100405508446291\\
806	0.00100552148258085\\
811	0.00100702310008556\\
816	0.00100858742801131\\
821	0.00101019604255511\\
826	0.00101186468880859\\
831	0.00101356302499791\\
836	0.00101529327694052\\
841	0.00101701663883118\\
846	0.00101873114252798\\
851	0.00102040865392136\\
856	0.00102206338480073\\
861	0.00102369959308566\\
866	0.00102535290844927\\
871	0.00102705244848214\\
876	0.0010288318385271\\
881	0.0010307146924231\\
886	0.00103270775560501\\
891	0.00103480218156959\\
896	0.00103696404115812\\
901	0.00103912670830278\\
906	0.00104118050817298\\
911	0.00104291518731388\\
916	0.00104426677755078\\
921	0.00104527849215748\\
926	0.00104632774796561\\
931	0.00104740507577423\\
936	0.00104868931621372\\
941	0.00105016182354285\\
946	0.00105194505744318\\
951	0.0010540403206828\\
956	0.00105651742749819\\
961	0.00105932887846144\\
966	0.00106243480962815\\
971	0.00106567911065026\\
976	0.00106888474830274\\
981	0.00107160586640787\\
986	0.00107380482268684\\
991	0.00107579847797852\\
996	0.00107789902082252\\
};
\end{axis}
\end{tikzpicture}%

%% file: media/exectimeAverage.tikz
%
%
\begin{tikzpicture}

\begin{axis}[%
width=0.951\figurewidth,
height=\figureheight,
at={(0\figurewidth,0\figureheight)},
scale only axis,
xmin=0,
xmax=500,
xlabel style={font=\color{white!15!black}},
xlabel={$k$},
ymin=0,
ymax=0.04,
ylabel style={font=\color{white!15!black}},
ylabel={exec. time/iter. [s]},
legend columns=2,
axis background/.style={fill=white},
legend style={legend cell align=left, align=left, draw=white!15!black}
]
\addplot [color=black]
  table[row sep=crcr]{%
1	0.0356742\\
6	0.0241111333333333\\
11	0.0226736727272727\\
16	0.02214296875\\
21	0.0218493904761905\\
26	0.0216643269230769\\
31	0.0215840225806452\\
36	0.0215318722222222\\
41	0.0214773902439024\\
46	0.021437352173913\\
51	0.0213613156862745\\
56	0.0213180821428571\\
61	0.0213003114754098\\
66	0.0212712757575758\\
71	0.0212498873239437\\
76	0.0212233223684211\\
81	0.0212049345679012\\
86	0.0212737174418605\\
91	0.0213229747252747\\
96	0.0213114260416667\\
101	0.0213233801980198\\
106	0.0213068424528302\\
111	0.0213005765765766\\
116	0.021308274137931\\
121	0.0213003669421488\\
126	0.0212957436507937\\
131	0.0213050297709924\\
136	0.0213163566176471\\
141	0.0213066439716312\\
146	0.0213015191780822\\
151	0.0213481993377484\\
156	0.0213891865384615\\
161	0.0214108093167702\\
166	0.0213970240963855\\
171	0.0213975450292398\\
176	0.0213845329545455\\
181	0.0213885834254144\\
186	0.0213852037634409\\
191	0.0213812005235602\\
196	0.0213739897959184\\
201	0.021376028358209\\
206	0.0213666985436893\\
211	0.0213629601895735\\
216	0.0213557981481481\\
221	0.0213480488687783\\
226	0.0213519792035398\\
231	0.0213560285714286\\
236	0.0213749610169491\\
241	0.021365532780083\\
246	0.021363137398374\\
251	0.0213553780876494\\
256	0.02135277109375\\
261	0.0213556720306513\\
266	0.0213498680451128\\
271	0.0213502070110701\\
276	0.0213587489130435\\
281	0.0213548099644128\\
286	0.0213570223776224\\
291	0.0213505865979381\\
296	0.0213452976351351\\
301	0.0213493591362126\\
306	0.0213517212418301\\
311	0.0213503514469453\\
316	0.0213512933544304\\
321	0.0213624420560748\\
326	0.0213603665644172\\
331	0.0213535555891239\\
336	0.0213501544642857\\
341	0.0213441700879765\\
346	0.021338851734104\\
351	0.0213329504273504\\
356	0.0213272216292135\\
361	0.0213249819944598\\
366	0.0213319073770492\\
371	0.0213298630727763\\
376	0.0213401651595745\\
381	0.021337917847769\\
386	0.0213355810880829\\
391	0.0213324324808184\\
396	0.0213286898989899\\
401	0.0213316311720698\\
406	0.021331466502463\\
411	0.0213307462287104\\
416	0.0213385132211538\\
421	0.0213329408551069\\
426	0.0213299654929577\\
431	0.0213315535962877\\
436	0.0213315428899082\\
441	0.0213286464852608\\
446	0.0213249177130045\\
451	0.0213210906873614\\
456	0.0213189015350877\\
461	0.0213277453362256\\
466	0.0213285418454936\\
471	0.0213271789808917\\
476	0.0213258764705882\\
481	0.021325943035343\\
486	0.0213285899176955\\
491	0.0213280940936863\\
496	0.0213257022177419\\
501	0.0213251572854291\\
506	0.0213243509881423\\
511	0.0213269747553816\\
516	0.0213279496124031\\
521	0.0213329236084453\\
526	0.0213305511406844\\
531	0.0213298265536723\\
536	0.021329311380597\\
541	0.0213264064695009\\
546	0.0213273805860806\\
551	0.0213259299455535\\
556	0.0213257784172662\\
561	0.0213256600713012\\
566	0.0213278945229682\\
571	0.0213281211908932\\
576	0.0213286934027778\\
581	0.0213280578313253\\
586	0.0213273784982935\\
591	0.0213265414551608\\
596	0.0213262615771812\\
601	0.0213329747088186\\
606	0.0213316351485149\\
611	0.0213342042553192\\
616	0.0213330798701299\\
621	0.021331499194847\\
626	0.0213285905750799\\
631	0.0213285280507132\\
636	0.021327218081761\\
641	0.0213261277691108\\
646	0.0213251574303406\\
651	0.0213268013824885\\
656	0.0213274088414634\\
661	0.0213324450832073\\
666	0.0213377141141141\\
671	0.0213369372578241\\
676	0.021337075887574\\
681	0.0213376296622614\\
686	0.0213388988338192\\
691	0.0213388150506512\\
696	0.021342011637931\\
701	0.0213442673323823\\
706	0.0213442941926346\\
711	0.0213441675105485\\
716	0.0213448953910615\\
721	0.0213429746185853\\
726	0.0213415180440771\\
731	0.0213416393980848\\
736	0.0213403849184783\\
741	0.0213416798920378\\
746	0.0213422902144772\\
751	0.021340771770972\\
756	0.0213414505291005\\
761	0.0213401844940867\\
766	0.0213382417754569\\
771	0.0213371805447471\\
776	0.0213375157216495\\
781	0.0213370850192061\\
786	0.0213365176844784\\
791	0.0213405750948167\\
796	0.0213406400753769\\
801	0.0213415133583021\\
806	0.0213466470223325\\
811	0.0213479103575832\\
816	0.0213467796568627\\
821	0.0213460209500609\\
826	0.0213445180387409\\
831	0.0213448450060168\\
836	0.0213476942583732\\
841	0.0213476804994055\\
846	0.0213469156028369\\
851	0.0213468655699177\\
856	0.0213481160046729\\
861	0.0213485934959349\\
866	0.0213479012702078\\
871	0.0213487399540757\\
876	0.0213485551369863\\
881	0.0213511413166856\\
886	0.0213520121896162\\
891	0.021351789674523\\
896	0.0213513493303571\\
901	0.0213514366259711\\
906	0.0213520139072848\\
911	0.0213523423710208\\
916	0.0213529114628821\\
921	0.0213516941368078\\
926	0.0213515887688985\\
931	0.0213534457572502\\
936	0.0213538919871795\\
941	0.0213530522848034\\
946	0.0213542178646934\\
951	0.0213586773922187\\
956	0.0213583572175732\\
961	0.0213564119667013\\
966	0.0213551428571428\\
971	0.0213550853759011\\
976	0.0213592590163934\\
981	0.0213599341488277\\
986	0.0213588853955375\\
991	0.0213581759838547\\
996	0.0213583057228915\\
};
\addlegendentry{\algo}

\addplot [color=blue]
  table[row sep=crcr]{%
1	0.018856\\
6	0.0069573\\
11	0.00548946363636364\\
16	0.00487943125\\
21	0.00455890476190476\\
26	0.00433438076923077\\
31	0.00423664838709677\\
36	0.00418134166666667\\
41	0.00410429756097561\\
46	0.00403522826086957\\
51	0.00397808235294118\\
56	0.00393247678571429\\
61	0.00389560655737705\\
66	0.00385694848484849\\
71	0.00383258732394366\\
76	0.00380798421052632\\
81	0.00378247037037037\\
86	0.00376354186046512\\
91	0.00375239120879121\\
96	0.00374022395833333\\
101	0.00374260198019802\\
106	0.00375066698113208\\
111	0.00374667477477477\\
116	0.00374283103448276\\
121	0.00373415867768595\\
126	0.0037306746031746\\
131	0.00372975419847328\\
136	0.00372285441176471\\
141	0.00371334042553192\\
146	0.00370556232876712\\
151	0.00369825894039735\\
156	0.00370346794871795\\
161	0.00369781677018634\\
166	0.00369288855421687\\
171	0.00369264385964912\\
176	0.00368850738636364\\
181	0.00368101436464089\\
186	0.00367092580645161\\
191	0.00366332094240838\\
196	0.00365849489795919\\
201	0.00365922985074627\\
206	0.00366241067961165\\
211	0.00366002701421801\\
216	0.00365666157407408\\
221	0.00365157466063349\\
226	0.00364523672566372\\
231	0.00364131082251082\\
236	0.00363686313559322\\
241	0.00363136929460581\\
246	0.0036303593495935\\
251	0.00362532549800797\\
256	0.003623276171875\\
261	0.00362038314176245\\
266	0.00361591954887218\\
271	0.00361252988929889\\
276	0.00361131811594203\\
281	0.0036075206405694\\
286	0.0036049048951049\\
291	0.0036042147766323\\
296	0.00360272398648649\\
301	0.00360757441860465\\
306	0.00360893235294118\\
311	0.00360681511254019\\
316	0.00360444841772152\\
321	0.00360128380062305\\
326	0.00359949539877301\\
331	0.00359560906344411\\
336	0.00359461904761905\\
341	0.00359265601173021\\
346	0.00358923121387283\\
351	0.00358665014245014\\
356	0.0035844047752809\\
361	0.00358120166204986\\
366	0.00357896174863388\\
371	0.00357610188679245\\
376	0.00357548590425532\\
381	0.00357308057742782\\
386	0.00357226606217617\\
391	0.00357202148337596\\
396	0.0035699595959596\\
401	0.00357130249376559\\
406	0.00357326280788177\\
411	0.00357211313868613\\
416	0.00357063365384615\\
421	0.00356877482185273\\
426	0.00356696291079812\\
431	0.00356806798143852\\
436	0.00356800252293578\\
441	0.00356573696145125\\
446	0.00356641704035875\\
451	0.0035641866962306\\
456	0.00356547828947369\\
461	0.00356573665943601\\
466	0.00356455193133047\\
471	0.00356362993630573\\
476	0.00356248172268908\\
481	0.00356062494802495\\
486	0.00355980308641976\\
491	0.00355820386965377\\
496	0.00355687862903226\\
501	0.00356366966067865\\
506	0.00356607964426878\\
511	0.00356500195694716\\
516	0.00356314224806202\\
521	0.00356133550863724\\
526	0.00355972547528517\\
531	0.00355808210922787\\
536	0.00355697873134328\\
541	0.00355571737523105\\
546	0.00355421355311355\\
551	0.00355267477313975\\
556	0.00355092104316547\\
561	0.00354968092691622\\
566	0.00354882102473498\\
571	0.00354724956217163\\
576	0.00354583472222222\\
581	0.00354575989672978\\
586	0.00354487354948806\\
591	0.00354435160744501\\
596	0.00354315436241611\\
601	0.00354341930116472\\
606	0.00354380627062706\\
611	0.00354297299509001\\
616	0.00354458928571428\\
621	0.0035478074074074\\
626	0.00354936277955271\\
631	0.0035495765451664\\
636	0.00355016289308176\\
641	0.00354990780031201\\
646	0.0035491826625387\\
651	0.00354779354838709\\
656	0.00354719542682926\\
661	0.00354751422087745\\
666	0.00354621801801801\\
671	0.00354491341281669\\
676	0.00354377855029586\\
681	0.00354272790014684\\
686	0.00354160247813411\\
691	0.00354052923299566\\
696	0.00353970158045977\\
701	0.00353990171184023\\
706	0.00353945878186969\\
711	0.00353912784810126\\
716	0.00353799846368715\\
721	0.00353694701803051\\
726	0.00353596019283746\\
731	0.00353459917920656\\
736	0.0035351464673913\\
741	0.00353479095816464\\
746	0.00353455442359249\\
751	0.00353324354194407\\
756	0.00353192433862434\\
761	0.00353069802890933\\
766	0.00352964947780679\\
771	0.00352834824902724\\
776	0.00352723363402062\\
781	0.00352611485275288\\
786	0.00352509351145038\\
791	0.00352419367888748\\
796	0.00352307989949748\\
801	0.00352346803995006\\
806	0.00352343945409429\\
811	0.0035235101109741\\
816	0.00352309362745098\\
821	0.00352346979293544\\
826	0.00352243050847457\\
831	0.00352216064981949\\
836	0.0035216961722488\\
841	0.00352049881093935\\
846	0.00351974751773049\\
851	0.00351848965922444\\
856	0.00351755537383177\\
861	0.00351644738675958\\
866	0.00351549053117783\\
871	0.0035145447761194\\
876	0.00351411883561644\\
881	0.00351325743473326\\
886	0.00351231557562077\\
891	0.00351242940516274\\
896	0.00351181986607143\\
901	0.00351222397336293\\
906	0.00351202847682119\\
911	0.00351129593852909\\
916	0.00351089497816594\\
921	0.00351041248642779\\
926	0.00350960928725702\\
931	0.00350844693877551\\
936	0.00350748632478632\\
941	0.003506874282678\\
946	0.00350599408033827\\
951	0.00350602933753943\\
956	0.00350551715481171\\
961	0.00350462684703434\\
966	0.00350373302277433\\
971	0.0035026508753862\\
976	0.00350190071721311\\
981	0.00350126748216106\\
986	0.00350090253549696\\
991	0.0035010668012109\\
996	0.00350075020080321\\
};
\addlegendentry{AL}

\addplot [color=red]
  table[row sep=crcr]{%
1	0.0081387\\
6	0.00516115\\
11	0.00445087272727273\\
16	0.00419986875\\
21	0.00406921904761905\\
26	0.00399067307692308\\
31	0.00400345806451613\\
36	0.003970825\\
41	0.00395867804878049\\
46	0.00393477826086956\\
51	0.00390208431372549\\
56	0.0038821\\
61	0.0038644393442623\\
66	0.00383866363636364\\
71	0.0038259985915493\\
76	0.00383341710526316\\
81	0.00381521728395062\\
86	0.00380414767441861\\
91	0.0037959010989011\\
96	0.00378964479166667\\
101	0.00379942871287129\\
106	0.00380072075471698\\
111	0.00379484774774775\\
116	0.00379097068965517\\
121	0.00378166446280992\\
126	0.0037778246031746\\
131	0.00376958320610687\\
136	0.00378042058823529\\
141	0.00377758723404255\\
146	0.00377358219178082\\
151	0.00376835695364238\\
156	0.00376385961538461\\
161	0.00376020869565217\\
166	0.0037555343373494\\
171	0.0037573514619883\\
176	0.00375856079545454\\
181	0.00375749281767956\\
186	0.00375290376344086\\
191	0.00375012565445026\\
196	0.00374666683673469\\
201	0.00375396218905473\\
206	0.00375740631067961\\
211	0.00375526398104265\\
216	0.00375301574074074\\
221	0.0037504556561086\\
226	0.00374538495575221\\
231	0.00374237965367965\\
236	0.00374118262711864\\
241	0.00374235767634855\\
246	0.00374226666666667\\
251	0.00373835258964143\\
256	0.003735796484375\\
261	0.00373481072796935\\
266	0.00373400639097744\\
271	0.00373204501845018\\
276	0.0037299231884058\\
281	0.00372625836298932\\
286	0.00372486153846154\\
291	0.0037225116838488\\
296	0.00372071587837838\\
301	0.00372484651162791\\
306	0.0037264977124183\\
311	0.00372552604501608\\
316	0.00372633702531646\\
321	0.00372716043613707\\
326	0.00372688098159509\\
331	0.00372379667673716\\
336	0.00372488898809524\\
341	0.00372526891495601\\
346	0.00372372167630058\\
351	0.00372294245014245\\
356	0.00372322921348315\\
361	0.00372275180055402\\
366	0.0037243718579235\\
371	0.00372516010781671\\
376	0.00372368909574468\\
381	0.00372202414698163\\
386	0.00372167979274612\\
391	0.00372025396419438\\
396	0.00371811868686869\\
401	0.00371923715710723\\
406	0.00372172586206897\\
411	0.00372027688564477\\
416	0.0037185233173077\\
421	0.00371737339667459\\
426	0.00371624389671362\\
431	0.00371475034802784\\
436	0.00371384105504587\\
441	0.00371262902494331\\
446	0.00371185941704036\\
451	0.00371046740576497\\
456	0.00371131842105263\\
461	0.00371417722342733\\
466	0.00371314077253219\\
471	0.00371144394904459\\
476	0.0037105481092437\\
481	0.00370972162162162\\
486	0.00370838847736626\\
491	0.00370657372708758\\
496	0.00370518911290323\\
501	0.00370701377245509\\
506	0.00370843083003953\\
511	0.0037116373776908\\
516	0.00371119360465116\\
521	0.00370950940499041\\
526	0.00370804809885932\\
531	0.0037098252354049\\
536	0.00370910932835821\\
541	0.00370806247689464\\
546	0.00370672032967033\\
551	0.00370566497277677\\
556	0.00370383651079137\\
561	0.00370233101604278\\
566	0.00370066749116608\\
571	0.00369928826619965\\
576	0.00369764045138889\\
581	0.0036959495697074\\
586	0.00369491126279864\\
591	0.00369369763113367\\
596	0.00369210520134228\\
601	0.00369267637271215\\
606	0.00369272326732673\\
611	0.00369414058919804\\
616	0.00369412207792208\\
621	0.00369269082125604\\
626	0.00369181453674122\\
631	0.00369095213946118\\
636	0.00368980801886793\\
641	0.0036887488299532\\
646	0.00368844442724458\\
651	0.00368683164362519\\
656	0.00368661128048781\\
661	0.00368598698940999\\
666	0.00368498768768769\\
671	0.00368443532041729\\
676	0.00368308520710059\\
681	0.00368190587371513\\
686	0.00368100481049563\\
691	0.00368285311143271\\
696	0.00368553548850575\\
701	0.00368942767475036\\
706	0.00368955481586402\\
711	0.00368954838255978\\
716	0.00368860837988827\\
721	0.00368794299583911\\
726	0.00368766267217631\\
731	0.00368710205198359\\
736	0.00368625597826087\\
741	0.00368521174089069\\
746	0.00368442225201072\\
751	0.00368520838881491\\
756	0.00368446917989418\\
761	0.00368368462549277\\
766	0.00368317402088773\\
771	0.00368286575875486\\
776	0.0036818800257732\\
781	0.00368110281690141\\
786	0.00368109580152672\\
791	0.00368113388116309\\
796	0.00368086884422111\\
801	0.00368240986267166\\
806	0.00368314230769231\\
811	0.0036826889025894\\
816	0.00368242990196078\\
821	0.00368190389768575\\
826	0.00368129794188862\\
831	0.00368077280385078\\
836	0.00368059055023923\\
841	0.00368002841854935\\
846	0.00367954373522459\\
851	0.00367875863689777\\
856	0.00367824322429906\\
861	0.00367896411149826\\
866	0.00368401454965358\\
871	0.00368571699196326\\
876	0.00368501883561644\\
881	0.00368424812712826\\
886	0.00368349130925508\\
891	0.00368282323232323\\
896	0.00368343604910714\\
901	0.0036843499445061\\
906	0.0036846114790287\\
911	0.00368478770581778\\
916	0.00368430043668122\\
921	0.00368378925081433\\
926	0.00368313984881209\\
931	0.00368263297529538\\
936	0.0036820797008547\\
941	0.00368186269925611\\
946	0.00368115539112051\\
951	0.00368082302839117\\
956	0.00368049686192469\\
961	0.00367972715920916\\
966	0.00367934213250518\\
971	0.00367883625128733\\
976	0.00367819077868852\\
981	0.00367783088685015\\
986	0.00367742636916836\\
991	0.00367695691220989\\
996	0.00367645441767068\\
};
\addlegendentry{\redalgo}

\end{axis}
\end{tikzpicture}%

%% file: media/fvalueLarge.tikz
%
%
\begin{tikzpicture}

\begin{axis}[%
width=0.951\figurewidth,
height=\figureheight,
at={(0\figurewidth,0\figureheight)},
scale only axis,
xmin=0,
xmax=5000,
xlabel style={font=\color{white!15!black}},
xlabel={$k$},
ymin=-1.15,
ymax=-0.8,
ylabel style={font=\color{white!15!black}},
ylabel={$f(x_k)$},
axis background/.style={fill=white}
]
\addplot [color=black, forget plot]
  table[row sep=crcr]{%
1	-0.84630758915256\\
6	-0.847880298241436\\
11	-0.839022975574149\\
16	-0.847347520504034\\
21	-0.850376498547392\\
26	-0.848435573878151\\
31	-0.84758792580252\\
36	-0.848478412838936\\
41	-0.849085808817055\\
46	-0.852821594139237\\
51	-0.86163249075372\\
56	-0.878450527337311\\
61	-0.904178119148198\\
66	-0.93819495659057\\
71	-0.976484857273121\\
76	-1.01399195398804\\
81	-1.04727495076529\\
86	-1.07639110614094\\
91	-1.10173818330229\\
96	-1.12337351018849\\
101	-1.14150702076845\\
106	-1.12622536793178\\
111	-1.10935220803312\\
116	-1.09698264003487\\
121	-1.08708404826232\\
126	-1.07878918911608\\
131	-1.07163467866336\\
136	-1.06533962375413\\
141	-1.05971915110011\\
146	-1.05464427470826\\
151	-1.05002098223926\\
156	-1.04577838154418\\
161	-1.04186153774998\\
166	-1.03822691969314\\
171	-1.03483938325137\\
176	-1.0316701030954\\
181	-1.02869511295328\\
186	-1.0258942494935\\
191	-1.02325037175827\\
196	-1.0207487735171\\
201	-1.0183767337\\
206	-1.01612316756432\\
211	-1.01397835257372\\
216	-1.01193371047997\\
221	-1.00998163221036\\
226	-1.00811533571635\\
231	-1.00632874946293\\
236	-1.00461641606073\\
241	-1.00297341188101\\
246	-1.00139527947989\\
251	-0.999877970396925\\
256	-0.998417796439257\\
261	-0.997011387972543\\
266	-0.995655658045348\\
271	-0.994347771404726\\
276	-0.993085117635219\\
281	-0.991865287791073\\
286	-0.990686053998975\\
291	-0.989545351595213\\
296	-0.988441263432812\\
301	-0.987372006052092\\
306	-0.986335917457265\\
311	-0.985331446280545\\
316	-0.984357142150375\\
321	-0.983411647106058\\
326	-0.982493687925431\\
331	-0.981602069250745\\
336	-0.98073566741417\\
341	-0.979893424878166\\
346	-0.979074345217071\\
351	-0.978277488576757\\
356	-0.977501967556607\\
361	-0.976746943467159\\
366	-0.976011622920536\\
371	-0.975295254718436\\
376	-0.97459712700577\\
381	-0.973916564662314\\
386	-0.97325292690821\\
391	-0.972605605101847\\
396	-0.971974020711904\\
401	-0.971357623446016\\
406	-0.970755889522194\\
411	-0.970168320068988\\
416	-0.969594439643252\\
421	-0.969033794854452\\
426	-0.968485953086316\\
431	-0.967950501307081\\
436	-0.967427044960746\\
441	-0.966915206932401\\
446	-0.966414626581323\\
451	-0.965924958836425\\
456	-0.965445873349024\\
461	-0.964977053698205\\
466	-0.964518196645325\\
471	-0.964069011433661\\
476	-0.963629219130145\\
481	-0.963198552006416\\
486	-0.962776752956617\\
491	-0.962363574949594\\
496	-0.961958780513368\\
501	-0.961562141250078\\
506	-0.961173437379547\\
511	-0.96079245730989\\
516	-0.960418997233388\\
521	-0.96005286074645\\
526	-0.959693858492249\\
531	-0.95934180782431\\
536	-0.958996532490433\\
541	-0.958657862334937\\
546	-0.958325633018837\\
551	-0.957999685756277\\
556	-0.957679867066192\\
561	-0.95736602853866\\
566	-0.957058026614149\\
571	-0.956755722375671\\
576	-0.956458981352078\\
581	-0.956167673332381\\
586	-0.95588167218993\\
591	-0.955600855715911\\
596	-0.955325105461437\\
601	-0.955054306587711\\
606	-0.954788347723623\\
611	-0.954527120830314\\
616	-0.954270521072199\\
621	-0.954018446694219\\
626	-0.95377079890459\\
631	-0.95352748176326\\
636	-0.953288402075017\\
641	-0.953053469287835\\
646	-0.952822595395381\\
651	-0.952595694844186\\
656	-0.952372684444787\\
661	-0.952153483286685\\
666	-0.951938012657477\\
671	-0.95172619596506\\
676	-0.951517958663911\\
681	-0.951313228183977\\
686	-0.951111933863418\\
691	-0.950914006884008\\
696	-0.950719380209449\\
701	-0.950527988526526\\
706	-0.950339768188951\\
711	-0.950154657163756\\
716	-0.949972594979646\\
721	-0.949793522678404\\
726	-0.949617382767623\\
731	-0.949444119175829\\
736	-0.949273677209537\\
741	-0.949106003512037\\
746	-0.948941046023515\\
751	-0.948778753943504\\
756	-0.948619077693732\\
761	-0.948461968883469\\
766	-0.948307380275173\\
771	-0.948155265752\\
776	-0.948005580286114\\
781	-0.947858279908173\\
786	-0.94771332167782\\
791	-0.947570663654869\\
796	-0.947430264871849\\
801	-0.947292085306921\\
806	-0.947156085858002\\
811	-0.947022228317264\\
816	-0.946890475346709\\
821	-0.94676079045439\\
826	-0.946633137971014\\
831	-0.946507483027719\\
836	-0.946383791534094\\
841	-0.946262030156997\\
846	-0.946142166299901\\
851	-0.946024168082927\\
856	-0.945908004323444\\
861	-0.945793644517113\\
866	-0.945681058819525\\
871	-0.945570218028473\\
876	-0.94546109356656\\
881	-0.945353657464503\\
886	-0.945247882344673\\
891	-0.945143741405495\\
896	-0.945041208405741\\
901	-0.944940257649839\\
906	-0.944840863973266\\
911	-0.944743002728493\\
916	-0.944646649771271\\
921	-0.944551781447344\\
926	-0.944458374579703\\
931	-0.944366406455878\\
936	-0.944275854815915\\
941	-0.944186697840641\\
946	-0.944098914140178\\
951	-0.944012482742864\\
956	-0.943927383084634\\
961	-0.943843594998475\\
966	-0.943761098704558\\
971	-0.943679874800208\\
976	-0.943599904250814\\
981	-0.943521168380422\\
986	-0.943443648862989\\
991	-0.943367327713906\\
996	-0.943292187281604\\
1001	-0.943218210239748\\
1006	-0.943145379579252\\
1011	-0.943073678601029\\
1016	-0.943003090908705\\
1021	-0.94293360040163\\
1026	-0.942865191267996\\
1031	-0.942797847978523\\
1036	-0.942731555279918\\
1041	-0.942666298188889\\
1046	-0.942602061986091\\
1051	-0.942538832210446\\
1056	-0.942476594653487\\
1061	-0.942415335354004\\
1066	-0.942355040592724\\
1071	-0.942295696887173\\
1076	-0.94223729098676\\
1081	-0.942179809867842\\
1086	-0.942123240729077\\
1091	-0.94206757098675\\
1096	-0.94201278827038\\
1101	-0.941958880418201\\
1106	-0.941905835473069\\
1111	-0.941853641678096\\
1116	-0.941802287472663\\
1121	-0.941751761488433\\
1126	-0.941702052545461\\
1131	-0.941653149648292\\
1136	-0.941605041982355\\
1141	-0.94155771891017\\
1146	-0.941511169967922\\
1151	-0.941465384861791\\
1156	-0.941420353464648\\
1161	-0.941376065812683\\
1166	-0.941332512102035\\
1171	-0.941289682685677\\
1176	-0.9412475680702\\
1181	-0.94120615891276\\
1186	-0.941165446018018\\
1191	-0.941125420335202\\
1196	-0.941086072955293\\
1201	-0.941047395108011\\
1206	-0.941009378159253\\
1211	-0.940972013608247\\
1216	-0.940935293084843\\
1221	-0.94089920834715\\
1226	-0.94086375127878\\
1231	-0.940828913886489\\
1236	-0.940794688297617\\
1241	-0.940761066757925\\
1246	-0.940728041629164\\
1251	-0.940695605386639\\
1256	-0.940663750617413\\
1261	-0.940632470017785\\
1266	-0.940601756391325\\
1271	-0.94057160264685\\
1276	-0.940542001796385\\
1281	-0.940512946953165\\
1286	-0.940484431329761\\
1291	-0.940456448236205\\
1296	-0.940428991078351\\
1301	-0.940402053355686\\
1306	-0.9403756286601\\
1311	-0.940349710673837\\
1316	-0.940324293168068\\
1321	-0.940299370001121\\
1326	-0.940274935117136\\
1331	-0.940250982544342\\
1336	-0.940227506393708\\
1341	-0.940204500857432\\
1346	-0.940181960207605\\
1351	-0.940159878794741\\
1356	-0.940138251046574\\
1361	-0.940117071466594\\
1366	-0.940096334632961\\
1371	-0.940076035197074\\
1376	-0.94005616788252\\
1381	-0.940036727483753\\
1386	-0.940017708865125\\
1391	-0.939999106959542\\
1396	-0.939980916767495\\
1401	-0.939963133356023\\
1406	-0.939945751857454\\
1411	-0.939928767468718\\
1416	-0.939912175449925\\
1421	-0.939895971123684\\
1426	-0.939880149874127\\
1431	-0.939864707145694\\
1436	-0.93984963844252\\
1441	-0.939834939327341\\
1446	-0.93982060542062\\
1451	-0.939806632399775\\
1456	-0.939793015998222\\
1461	-0.939779752004441\\
1466	-0.939766836261398\\
1471	-0.939754264665494\\
1476	-0.939742033165825\\
1481	-0.939730137763513\\
1486	-0.939718574510728\\
1491	-0.939707339510094\\
1496	-0.93969642891379\\
1501	-0.939685838923007\\
1506	-0.939675565786941\\
1511	-0.939665605802384\\
1516	-0.939655955312746\\
1521	-0.939646610707582\\
1526	-0.939637568421675\\
1531	-0.939628824934624\\
1536	-0.939620376769976\\
1541	-0.939612220494649\\
1546	-0.939604352718267\\
1551	-0.939596770092561\\
1556	-0.939589469310793\\
1561	-0.939582447106926\\
1566	-0.939575700255286\\
1571	-0.939569225569822\\
1576	-0.9395630199036\\
1581	-0.939557080148057\\
1586	-0.939551403232668\\
1591	-0.939545986124241\\
1596	-0.939540825826443\\
1601	-0.939535919379165\\
1606	-0.939531263858149\\
1611	-0.939526856374287\\
1616	-0.939522694073201\\
1621	-0.939518774134828\\
1626	-0.939515093772775\\
1631	-0.939511650233887\\
1636	-0.939508440797792\\
1641	-0.939505462776431\\
1646	-0.93950271351358\\
1651	-0.93950019038437\\
1656	-0.93949789079493\\
1661	-0.939495812181854\\
1666	-0.939493952011793\\
1671	-0.939492307781087\\
1676	-0.939490877015313\\
1681	-0.939489657268824\\
1686	-0.939488646124438\\
1691	-0.939487841193013\\
1696	-0.939487240113066\\
1701	-0.93948684055028\\
1706	-0.939486640197364\\
1711	-0.939486636773441\\
1716	-0.939486828023846\\
1721	-0.939487211719737\\
1726	-0.939487785657666\\
1731	-0.939488547659374\\
1736	-0.939489495571293\\
1741	-0.939490627264404\\
1746	-0.939491940633761\\
1751	-0.939493433598276\\
1756	-0.939495104100225\\
1761	-0.939496950105182\\
1766	-0.939498969601632\\
1771	-0.939501160600569\\
1776	-0.939503521135291\\
1781	-0.939506049261227\\
1786	-0.939508743055411\\
1791	-0.93951160061643\\
1796	-0.939514620064007\\
1801	-0.939517799538805\\
1806	-0.939521137202275\\
1811	-0.939524631236127\\
1816	-0.93952827984236\\
1821	-0.939532081242874\\
1826	-0.939536033679267\\
1831	-0.939540135412636\\
1836	-0.939544384723638\\
1841	-0.939548779913539\\
1846	-0.939553319313285\\
1851	-0.939558001350871\\
1856	-0.939562825090039\\
1861	-0.939567794903055\\
1866	-0.939572964610655\\
1871	-0.939578889293931\\
1876	-0.93959141935095\\
1881	-0.939656180829872\\
1886	-0.939825960084806\\
1891	-0.939990591845265\\
1896	-0.940117907947508\\
1901	-0.940216598033854\\
1906	-0.940294900677269\\
1911	-0.940358659079537\\
1916	-0.940411954013781\\
1921	-0.940457667290362\\
1926	-0.940497866603353\\
1931	-0.940534061182766\\
1936	-0.940567371173891\\
1941	-0.940598640577095\\
1946	-0.94062851251329\\
1951	-0.940657479428023\\
1956	-0.940685916419602\\
1961	-0.940714103585441\\
1966	-0.940742241553821\\
1971	-0.940770463387098\\
1976	-0.940798845044335\\
1981	-0.940827415769599\\
1986	-0.940856168858616\\
1991	-0.940885072589633\\
1996	-0.940914080614537\\
2001	-0.940943141037278\\
2006	-0.940972203548626\\
2011	-0.941001224348491\\
2016	-0.94103016891436\\
2021	-0.941059012947756\\
2026	-0.941087741960537\\
2031	-0.941116349974344\\
2036	-0.941144837769585\\
2041	-0.941173210992886\\
2046	-0.941201478377989\\
2051	-0.941229650188614\\
2056	-0.941257736989558\\
2061	-0.94128574872321\\
2066	-0.941313694121418\\
2071	-0.941341580362877\\
2076	-0.941369412983061\\
2081	-0.941397195921131\\
2086	-0.941424931716919\\
2091	-0.94145262173932\\
2096	-0.941480266480984\\
2101	-0.941507865807595\\
2106	-0.941535419223772\\
2111	-0.941562926053299\\
2116	-0.941590385619597\\
2121	-0.941617797329827\\
2126	-0.941645160766823\\
2131	-0.941672475691512\\
2136	-0.941699742071839\\
2141	-0.941726960034781\\
2146	-0.941754129864258\\
2151	-0.941781251933258\\
2156	-0.941808326697286\\
2161	-0.941835354627195\\
2166	-0.941862336212702\\
2171	-0.941889271906006\\
2176	-0.941916162140647\\
2181	-0.941943007286766\\
2186	-0.941969807683759\\
2191	-0.941996563603835\\
2196	-0.942023275293933\\
2201	-0.942049942943046\\
2206	-0.942076566728387\\
2211	-0.942103146782526\\
2216	-0.942129683240833\\
2221	-0.942156176205766\\
2226	-0.942182625793787\\
2231	-0.942209032096038\\
2236	-0.942235395224395\\
2241	-0.942261715268795\\
2246	-0.942287992342834\\
2251	-0.942314226538254\\
2256	-0.942340417970857\\
2261	-0.942366566733358\\
2266	-0.942392672941879\\
2271	-0.942418736687367\\
2276	-0.942444758083617\\
2281	-0.942470737217645\\
2286	-0.942496674199048\\
2291	-0.942522569109622\\
2296	-0.942548422054085\\
2301	-0.94257423310876\\
2306	-0.94260000237363\\
2311	-0.942625729920102\\
2316	-0.942651415844239\\
2321	-0.942677060213384\\
2326	-0.942702663120433\\
2331	-0.942728224629629\\
2336	-0.942753744831582\\
2341	-0.942779223788092\\
2346	-0.942804661588085\\
2351	-0.942830058291453\\
2356	-0.942855413985662\\
2361	-0.942880728729\\
2366	-0.942906002607658\\
2371	-0.942931235678249\\
2376	-0.942956428025961\\
2381	-0.942981579705701\\
2386	-0.943006690801403\\
2391	-0.943031761366592\\
2396	-0.943056791483939\\
2401	-0.943081781205352\\
2406	-0.94310673061249\\
2411	-0.943131639755783\\
2416	-0.943156508715612\\
2421	-0.943181337540846\\
2426	-0.943206126310839\\
2431	-0.943230875073121\\
2436	-0.943255583905889\\
2441	-0.943280252855373\\
2446	-0.94330488199858\\
2451	-0.943329471380471\\
2456	-0.943354021077113\\
2461	-0.943378531132172\\
2466	-0.943403001620665\\
2471	-0.943427432585115\\
2476	-0.943451824099703\\
2481	-0.943476176205696\\
2486	-0.943500488976258\\
2491	-0.943524762451525\\
2496	-0.943548996703797\\
2501	-0.943573191772091\\
2506	-0.943597347727671\\
2511	-0.94362146460859\\
2516	-0.94364554248532\\
2521	-0.943669581394762\\
2526	-0.94369358140664\\
2531	-0.943717542556941\\
2536	-0.943741464914489\\
2541	-0.943765348514297\\
2546	-0.943789193424579\\
2551	-0.943812999679347\\
2556	-0.943836767346077\\
2561	-0.943860496457969\\
2566	-0.943884187081783\\
2571	-0.943907839249999\\
2576	-0.943931453028695\\
2581	-0.943955028449539\\
2586	-0.943978565578142\\
2591	-0.944002064445379\\
2596	-0.944025525116314\\
2601	-0.944048947621282\\
2606	-0.944072332024725\\
2611	-0.94409567835644\\
2616	-0.944118986680466\\
2621	-0.944142257026025\\
2626	-0.944165489456826\\
2631	-0.944188684001513\\
2636	-0.944211840723394\\
2641	-0.944234959650779\\
2646	-0.944258040846664\\
2651	-0.944281084338993\\
2656	-0.944304090190423\\
2661	-0.944327058428547\\
2666	-0.944349989115893\\
2671	-0.944372882279755\\
2676	-0.944395737982424\\
2681	-0.944418556250999\\
2686	-0.944441337147703\\
2691	-0.944464080699478\\
2696	-0.944486786968411\\
2701	-0.944509455981331\\
2706	-0.944532087800305\\
2711	-0.944554682452112\\
2716	-0.94457723999872\\
2721	-0.944599760466888\\
2726	-0.944622243918675\\
2731	-0.944644690380782\\
2736	-0.944667099915459\\
2741	-0.944689472549404\\
2746	-0.94471180834494\\
2751	-0.944734107328932\\
2756	-0.944756369563821\\
2761	-0.944778595076622\\
2766	-0.944800783929956\\
2771	-0.944822936150921\\
2776	-0.944845051802426\\
2781	-0.944867130911789\\
2786	-0.944889173542097\\
2791	-0.944911179720912\\
2796	-0.944933149511592\\
2801	-0.944955082941954\\
2806	-0.944976980075671\\
2811	-0.944998840940856\\
2816	-0.945020665601522\\
2821	-0.945042454086078\\
2826	-0.945064206458826\\
2831	-0.945085922748516\\
2836	-0.945107603019898\\
2841	-0.945129247301987\\
2846	-0.945150855659999\\
2851	-0.945172428123263\\
2856	-0.945193964757381\\
2861	-0.945215465592181\\
2866	-0.945236930693611\\
2871	-0.945258360091852\\
2876	-0.945279753853373\\
2881	-0.945301112008711\\
2886	-0.945322434624742\\
2891	-0.945343721732501\\
2896	-0.94536497339931\\
2901	-0.945386189656595\\
2906	-0.94540737057215\\
2911	-0.945428516177841\\
2916	-0.945449626542042\\
2921	-0.945470701696977\\
2926	-0.945491741711455\\
2931	-0.945512746618212\\
2936	-0.945533716486406\\
2941	-0.945554651349348\\
2946	-0.945575551276708\\
2951	-0.945596416302214\\
2956	-0.945617246496013\\
2961	-0.94563804189222\\
2966	-0.945658802561546\\
2971	-0.945679528538471\\
2976	-0.945700219894234\\
2981	-0.945720876663825\\
2986	-0.945741498918882\\
2991	-0.945762086694756\\
2996	-0.945782640063646\\
3001	-0.94580315906134\\
3006	-0.945823643760564\\
3011	-0.945844094197487\\
3016	-0.945864510445184\\
3021	-0.945884892540388\\
3026	-0.945905240556611\\
3031	-0.945925554530913\\
3036	-0.945945834537243\\
3041	-0.945966080613134\\
3046	-0.94598629283292\\
3051	-0.946006471234519\\
3056	-0.946026615892688\\
3061	-0.946046726845686\\
3066	-0.946066804168725\\
3071	-0.946086847900395\\
3076	-0.946106858116287\\
3081	-0.946126834855311\\
3086	-0.946146778193473\\
3091	-0.946166688170013\\
3096	-0.946186564861358\\
3101	-0.946206408307034\\
3106	-0.946226218583815\\
3111	-0.946245995731447\\
3116	-0.946265739827101\\
3121	-0.946285450910851\\
3126	-0.946305129060108\\
3131	-0.946324774315249\\
3136	-0.946344386754038\\
3141	-0.946363966417051\\
3146	-0.946383513382318\\
3151	-0.946403027690627\\
3156	-0.946422509420294\\
3161	-0.946441958612394\\
3166	-0.94646137534545\\
3171	-0.946480759660698\\
3176	-0.946500111637035\\
3181	-0.946519431315764\\
3186	-0.946538718775986\\
3191	-0.946557974059187\\
3196	-0.946577197244644\\
3201	-0.946596388374035\\
3206	-0.946615547526826\\
3211	-0.946634674744845\\
3216	-0.946653770107759\\
3221	-0.946672833657471\\
3226	-0.946691865473773\\
3231	-0.946710865598795\\
3236	-0.946729834112409\\
3241	-0.946748771056745\\
3246	-0.946767676511889\\
3251	-0.946786550520061\\
3256	-0.946805393161415\\
3261	-0.946824204478282\\
3266	-0.946842984550964\\
3271	-0.946861733421749\\
3276	-0.946880451171097\\
3281	-0.946899137841325\\
3286	-0.946917793512969\\
3291	-0.946936418228374\\
3296	-0.94695501206817\\
3301	-0.946973575074659\\
3306	-0.946992107328518\\
3311	-0.947010608872158\\
3316	-0.947029079786278\\
3321	-0.947047520113244\\
3326	-0.947065929933864\\
3331	-0.947084309290481\\
3336	-0.947102658264041\\
3341	-0.947120976896844\\
3346	-0.947139265269828\\
3351	-0.947157523425178\\
3356	-0.947175751443931\\
3361	-0.947193949368316\\
3366	-0.947212117279286\\
3371	-0.947230255219108\\
3376	-0.947248363268845\\
3381	-0.947266441470629\\
3386	-0.947284489905582\\
3391	-0.947302508615859\\
3396	-0.947320497682488\\
3401	-0.947338457147625\\
3406	-0.94735638709241\\
3411	-0.947374287558909\\
3416	-0.947392158628274\\
3421	-0.947410000342541\\
3426	-0.947427812782956\\
3431	-0.94744559599145\\
3436	-0.947463350049317\\
3441	-0.947481074998473\\
3446	-0.947498770920213\\
3451	-0.947516437856438\\
3456	-0.947534075888418\\
3461	-0.947551685058098\\
3466	-0.947569265446805\\
3471	-0.947586817096434\\
3476	-0.947604340088383\\
3481	-0.947621834464458\\
3486	-0.947639300306089\\
3491	-0.947656737655111\\
3496	-0.947674146593038\\
3501	-0.947691527161726\\
3506	-0.947708879442701\\
3511	-0.947726203477772\\
3516	-0.94774349934861\\
3521	-0.947760767096912\\
3526	-0.947778006804408\\
3531	-0.94779521851294\\
3536	-0.947812402304136\\
3541	-0.94782955821992\\
3546	-0.947846686342102\\
3551	-0.947863786712516\\
3556	-0.947880859413164\\
3561	-0.947897904485944\\
3566	-0.947914922012826\\
3571	-0.94793191203567\\
3576	-0.947948874636631\\
3581	-0.947965809857571\\
3586	-0.947982717780741\\
3591	-0.947999598448081\\
3596	-0.948016451941917\\
3601	-0.948033278304257\\
3606	-0.94805007761759\\
3611	-0.948066849923877\\
3616	-0.948083595305804\\
3621	-0.948100313805463\\
3626	-0.948117005505565\\
3631	-0.948133670448238\\
3636	-0.948150308716353\\
3641	-0.948166920352121\\
3646	-0.948183505438571\\
3651	-0.948200064017922\\
3656	-0.948216596173366\\
3661	-0.94823310194724\\
3666	-0.948249581422817\\
3671	-0.948266034642601\\
3676	-0.94828246168999\\
3681	-0.948298862607519\\
3686	-0.948315237478726\\
3691	-0.948331586346306\\
3696	-0.948347909293989\\
3701	-0.94836420636441\\
3706	-0.948380477641441\\
3711	-0.948396723167887\\
3716	-0.948412943027819\\
3721	-0.948429137264077\\
3726	-0.948445305960883\\
3731	-0.948461449161214\\
3736	-0.94847756694944\\
3741	-0.948493659368673\\
3746	-0.948509726503389\\
3751	-0.94852576839674\\
3756	-0.948541785133378\\
3761	-0.948557776756556\\
3766	-0.948573743351003\\
3771	-0.948589684960166\\
3776	-0.948605601668908\\
3781	-0.948621493520697\\
3786	-0.948637360600596\\
3791	-0.948653202952158\\
3796	-0.94866902066052\\
3801	-0.948684813769286\\
3806	-0.948700582363735\\
3811	-0.948716326487636\\
3816	-0.948732046226251\\
3821	-0.94874774162337\\
3826	-0.948763412764504\\
3831	-0.948779059693427\\
3836	-0.948794682495699\\
3841	-0.94881028121518\\
3846	-0.948825855937455\\
3851	-0.948841406706478\\
3856	-0.948856933607871\\
3861	-0.948872436685644\\
3866	-0.948887916025403\\
3871	-0.948903371671129\\
3876	-0.948918803708501\\
3881	-0.948934212181402\\
3886	-0.94894959717558\\
3891	-0.948964958734891\\
3896	-0.948980296945048\\
3901	-0.948995611849826\\
3906	-0.949010903534925\\
3911	-0.949026172044051\\
3916	-0.949041417462801\\
3921	-0.949056639834731\\
3926	-0.949071839245366\\
3931	-0.949087015738192\\
3936	-0.949102169398557\\
3941	-0.949117300269816\\
3946	-0.949132408437178\\
3951	-0.949147493943738\\
3956	-0.949162556874532\\
3961	-0.949177597272421\\
3966	-0.94919261522223\\
3971	-0.949207610766614\\
3976	-0.949222583990154\\
3981	-0.949237534935112\\
3986	-0.949252463685693\\
3991	-0.949267370283895\\
3996	-0.949282254813634\\
4001	-0.949297117316543\\
4006	-0.949311957876127\\
4011	-0.949326776533551\\
4016	-0.949341573371855\\
4021	-0.949356348431788\\
4026	-0.949371101795828\\
4031	-0.949385833504386\\
4036	-0.949400543639364\\
4041	-0.949415232240541\\
4046	-0.949429899389305\\
4051	-0.949444545124848\\
4056	-0.949459169527882\\
4061	-0.949473772636927\\
4066	-0.949488354532044\\
4071	-0.949502915251068\\
4076	-0.949517454873261\\
4081	-0.949531973435889\\
4086	-0.949546471017317\\
4091	-0.949560947653967\\
4096	-0.949575403423551\\
4101	-0.949589838361537\\
4106	-0.949604252544769\\
4111	-0.949618646007952\\
4116	-0.94963301882693\\
4121	-0.949647371035488\\
4126	-0.949661702708473\\
4131	-0.949676013878777\\
4136	-0.949690304620256\\
4141	-0.949704574964786\\
4146	-0.949718824985213\\
4151	-0.949733054712403\\
4156	-0.94974726421809\\
4161	-0.949761453532167\\
4166	-0.949775622725269\\
4171	-0.949789771826258\\
4176	-0.949803900904634\\
4181	-0.949818009988266\\
4186	-0.949832099145516\\
4191	-0.949846168403157\\
4196	-0.949860217828377\\
4201	-0.949874247446989\\
4206	-0.949888257324973\\
4211	-0.949902247487166\\
4216	-0.949916217998438\\
4221	-0.949930168882642\\
4226	-0.949944100203512\\
4231	-0.949958011983898\\
4236	-0.949971904286531\\
4241	-0.949985777133302\\
4246	-0.949999630585861\\
4251	-0.950013464665166\\
4256	-0.950027279431939\\
4261	-0.950041074906306\\
4266	-0.950054851147927\\
4271	-0.950068608176211\\
4276	-0.950082346049987\\
4281	-0.950096064787947\\
4286	-0.950109764448039\\
4291	-0.950123445048398\\
4296	-0.950137106646173\\
4301	-0.950150749259018\\
4306	-0.950164372943358\\
4311	-0.950177977716423\\
4316	-0.950191563634149\\
4321	-0.950205130713329\\
4326	-0.950218679009327\\
4331	-0.950232208538751\\
4336	-0.950245719356664\\
4341	-0.950259211479461\\
4346	-0.950272684961818\\
4351	-0.9502861398201\\
4356	-0.950299576108833\\
4361	-0.950312993844436\\
4366	-0.950326393081284\\
4371	-0.950339773836053\\
4376	-0.950353136163023\\
4381	-0.950366480078981\\
4386	-0.950379805638378\\
4391	-0.950393112858375\\
4396	-0.950406401793574\\
4401	-0.95041967246159\\
4406	-0.950432924917179\\
4411	-0.950446159178582\\
4416	-0.950459375300864\\
4421	-0.950472573302788\\
4426	-0.950485753239773\\
4431	-0.950498915131393\\
4436	-0.950512059033416\\
4441	-0.950525184966125\\
4446	-0.950538292985844\\
4451	-0.950551383113623\\
4456	-0.950564455406381\\
4461	-0.950577509885975\\
4466	-0.950590546609898\\
4471	-0.950603565600872\\
4476	-0.95061656691707\\
4481	-0.950629550581995\\
4486	-0.950642516654488\\
4491	-0.950655465158945\\
4496	-0.950668396154886\\
4501	-0.950681309667754\\
4506	-0.950694205757642\\
4511	-0.950707084450805\\
4516	-0.950719945807964\\
4521	-0.950732789856349\\
4526	-0.950745616657292\\
4531	-0.950758426238989\\
4536	-0.950771218663313\\
4541	-0.950783993959284\\
4546	-0.9507967521893\\
4551	-0.950809493383357\\
4556	-0.950822217604471\\
4561	-0.950834924883189\\
4566	-0.950847615283078\\
4571	-0.950860288835508\\
4576	-0.950872945604444\\
4581	-0.950885585621959\\
4586	-0.950898208952284\\
4591	-0.950910815628171\\
4596	-0.950923405714304\\
4601	-0.950935979243953\\
4606	-0.950948536282024\\
4611	-0.950961076862255\\
4616	-0.950973601049826\\
4621	-0.950986108878836\\
4626	-0.950998600414503\\
4631	-0.951011075691396\\
4636	-0.951023534774682\\
4641	-0.951035977699127\\
4646	-0.951048404529931\\
4651	-0.951060815302113\\
4656	-0.951073210080595\\
4661	-0.951085588900518\\
4666	-0.951097951826652\\
4671	-0.951110298894043\\
4676	-0.951122630167193\\
4681	-0.951134945680976\\
4686	-0.95114724549948\\
4691	-0.951159529657525\\
4696	-0.951171798218587\\
4701	-0.951184051217182\\
4706	-0.951196288716226\\
4711	-0.951208510749832\\
4716	-0.95122071738024\\
4721	-0.951232908641019\\
4726	-0.951245084593562\\
4731	-0.95125724527094\\
4736	-0.951269390733503\\
4741	-0.95128152101377\\
4746	-0.951293636171106\\
4751	-0.951305736237067\\
4756	-0.951317821270103\\
4761	-0.951329891300836\\
4766	-0.951341946386332\\
4771	-0.951353986556289\\
4776	-0.951366011866615\\
4781	-0.951378022345881\\
4786	-0.951390018048568\\
4791	-0.951401999002126\\
4796	-0.951413965259593\\
4801	-0.951425916847228\\
4806	-0.951437853816435\\
4811	-0.951449776192105\\
4816	-0.951461684024057\\
4821	-0.951473577335837\\
4826	-0.951485456175647\\
4831	-0.951497320565495\\
4836	-0.951509170551891\\
4841	-0.951521006155405\\
4846	-0.951532827420766\\
4851	-0.951544634366994\\
4856	-0.951556427036996\\
4861	-0.951568205448267\\
4866	-0.951579969641864\\
4871	-0.95159171963373\\
4876	-0.951603455463203\\
4881	-0.951615177144781\\
4886	-0.95162688471591\\
4891	-0.951638578189522\\
4896	-0.951650257601229\\
4901	-0.951661922962533\\
4906	-0.951673574307455\\
4911	-0.951685211645904\\
4916	-0.951696835010288\\
4921	-0.951708444409186\\
4926	-0.95172003987346\\
4931	-0.95173162141047\\
4936	-0.951743189049637\\
4941	-0.951754742797077\\
4946	-0.951766282680735\\
4951	-0.951777808705768\\
4956	-0.951789320898839\\
4961	-0.951800819264171\\
4966	-0.951812303827389\\
4971	-0.95182377459193\\
4976	-0.9518352315824\\
4981	-0.951846674801674\\
4986	-0.951858104273475\\
4991	-0.951869520000204\\
4996	-0.951880922004933\\
};
\end{axis}
\end{tikzpicture}%

%% file: media/gnormLarge.tikz
%
%
\begin{tikzpicture}

\begin{axis}[%
width=0.951\figurewidth,
height=\figureheight,
at={(0\figurewidth,0\figureheight)},
scale only axis,
xmin=0,
xmax=5000,
xlabel style={font=\color{white!15!black}},
xlabel={$k$},
ymode=log,
ymin=1e-07,
ymax=1.32174965062021e-05,
yminorticks=true,
ylabel style={font=\color{white!15!black}},
ylabel={$\text{rms}(h(x_k))$},
axis background/.style={fill=white}
]
\addplot [color=black, forget plot]
  table[row sep=crcr]{%
1	0\\
6	3.99287330181288e-06\\
11	8.25018242116625e-06\\
16	6.6479195665876e-06\\
21	5.34388196802814e-06\\
26	4.25940357737561e-06\\
31	3.9803325250281e-06\\
36	4.21790697246333e-06\\
41	4.21688697208619e-06\\
46	3.70187107316746e-06\\
51	3.36584563923141e-06\\
56	3.20451287213933e-06\\
61	3.65013917833083e-06\\
66	4.91204430083546e-06\\
71	6.61764638204071e-06\\
76	8.32203081104836e-06\\
81	9.74035227899513e-06\\
86	1.09116219306399e-05\\
91	1.18711403713033e-05\\
96	1.26338293226302e-05\\
101	1.32174965062021e-05\\
106	1.23104739741545e-05\\
111	1.12523334827565e-05\\
116	1.0492677809354e-05\\
121	9.88913128082974e-06\\
126	9.38482594337094e-06\\
131	8.95031582431132e-06\\
136	8.56806764157068e-06\\
141	8.22666983898389e-06\\
146	7.91821948544115e-06\\
151	7.63698622259609e-06\\
156	7.37866514202258e-06\\
161	7.13992971649159e-06\\
166	6.91814967104604e-06\\
171	6.71120513168995e-06\\
176	6.51735978048406e-06\\
181	6.33517166676262e-06\\
186	6.16342888342719e-06\\
191	6.00110214869286e-06\\
196	5.84730917607564e-06\\
201	5.70128744836856e-06\\
206	5.56237310108657e-06\\
211	5.42998432516888e-06\\
216	5.30360816493586e-06\\
221	5.18278990277006e-06\\
226	5.06712443959038e-06\\
231	4.95624923296272e-06\\
236	4.84983846369749e-06\\
241	4.74759818070627e-06\\
246	4.64926223178602e-06\\
251	4.55458883102034e-06\\
256	4.46335764579275e-06\\
261	4.37536731094537e-06\\
266	4.29043329639747e-06\\
271	4.2083860690705e-06\\
276	4.12906950129034e-06\\
281	4.05233948672806e-06\\
286	3.97806273198168e-06\\
291	3.90611569751753e-06\\
296	3.83638366619132e-06\\
301	3.76875992120089e-06\\
306	3.70314501830083e-06\\
311	3.63944613949265e-06\\
316	3.57757651742737e-06\\
321	3.51745492136018e-06\\
326	3.45900519687999e-06\\
331	3.40215585275566e-06\\
336	3.34683968919805e-06\\
341	3.29299346261482e-06\\
346	3.24055758262414e-06\\
351	3.18947583765161e-06\\
356	3.13969514590871e-06\\
361	3.09116532898216e-06\\
366	3.04383890559391e-06\\
371	2.99767090340203e-06\\
376	2.95261868697906e-06\\
381	2.90864180030879e-06\\
386	2.86570182235763e-06\\
391	2.8237622344271e-06\\
396	2.78278829815331e-06\\
401	2.74274694313366e-06\\
406	2.70360666328871e-06\\
411	2.66533742114392e-06\\
416	2.62791055932147e-06\\
421	2.59129871859371e-06\\
426	2.55547576192214e-06\\
431	2.52041670396715e-06\\
436	2.48609764559315e-06\\
441	2.45249571295711e-06\\
446	2.41958900079688e-06\\
451	2.38735651957298e-06\\
456	2.35577814615914e-06\\
461	2.32483457779105e-06\\
466	2.29450728902792e-06\\
471	2.26477849148004e-06\\
476	2.23563109610427e-06\\
481	2.20704867786455e-06\\
486	2.17901544258308e-06\\
491	2.1515161958252e-06\\
496	2.12453631366221e-06\\
501	2.09806171518539e-06\\
506	2.07207883664412e-06\\
511	2.04657460709117e-06\\
516	2.02153642543622e-06\\
521	1.99695213880699e-06\\
526	1.97281002213369e-06\\
531	1.94909875887028e-06\\
536	1.92580742278295e-06\\
541	1.9029254607311e-06\\
546	1.88044267637797e-06\\
551	1.85834921477592e-06\\
556	1.83663554776239e-06\\
561	1.81529246012727e-06\\
566	1.79431103649009e-06\\
571	1.77368264885607e-06\\
576	1.753398944803e-06\\
581	1.7334518362617e-06\\
586	1.71383348885601e-06\\
591	1.69453631176885e-06\\
596	1.67555294810284e-06\\
601	1.65687626570959e-06\\
606	1.63849934845642e-06\\
611	1.62041548790831e-06\\
616	1.60261817540137e-06\\
621	1.58510109448799e-06\\
626	1.56785811372687e-06\\
631	1.55088327980784e-06\\
636	1.53417081098558e-06\\
641	1.51771509080841e-06\\
646	1.50151066212961e-06\\
651	1.48555222138071e-06\\
656	1.46983461309804e-06\\
661	1.4543528246836e-06\\
666	1.43910198139619e-06\\
671	1.42407734155324e-06\\
676	1.40927429193735e-06\\
681	1.39468834339653e-06\\
686	1.38031512662563e-06\\
691	1.36615038812607e-06\\
696	1.35218998632901e-06\\
701	1.33842988787503e-06\\
706	1.32486616404832e-06\\
711	1.31149498735134e-06\\
716	1.29831262821473e-06\\
721	1.28531545184155e-06\\
726	1.27249991517109e-06\\
731	1.25986256396269e-06\\
736	1.24740002999331e-06\\
741	1.23510902835987e-06\\
746	1.22298635488826e-06\\
751	1.21102888363727e-06\\
756	1.19923356449715e-06\\
761	1.18759742087828e-06\\
766	1.17611754748602e-06\\
771	1.16479110817503e-06\\
776	1.1536153338843e-06\\
781	1.1425875206472e-06\\
786	1.1317050276713e-06\\
791	1.12096527548682e-06\\
796	1.11036574416492e-06\\
801	1.09990397159358e-06\\
806	1.08957755181681e-06\\
811	1.0793841334322e-06\\
816	1.06932141804115e-06\\
821	1.05938715875584e-06\\
826	1.0495791587542e-06\\
831	1.03989526988615e-06\\
836	1.03033339132533e-06\\
841	1.02089146826662e-06\\
846	1.01156749066698e-06\\
851	1.0023594920272e-06\\
856	9.93265548216378e-07\\
861	9.84283776331356e-07\\
866	9.75412333594862e-07\\
871	9.66649416289365e-07\\
876	9.57993258725907e-07\\
881	9.49442132242729e-07\\
886	9.40994344240522e-07\\
891	9.32648237242901e-07\\
896	9.2440218799028e-07\\
901	9.16254606561341e-07\\
906	9.0820393551954e-07\\
911	9.00248649087006e-07\\
916	8.92387252346715e-07\\
921	8.84618280464292e-07\\
926	8.7694029793583e-07\\
931	8.69351897858778e-07\\
936	8.6185170122433e-07\\
941	8.54438356229035e-07\\
946	8.47110537610987e-07\\
951	8.39866945999909e-07\\
956	8.32706307291927e-07\\
961	8.25627372039697e-07\\
966	8.18628914860154e-07\\
971	8.11709733859581e-07\\
976	8.04868650076291e-07\\
981	7.98104506938123e-07\\
986	7.91416169735415e-07\\
991	7.84802525109812e-07\\
996	7.7826248055684e-07\\
1001	7.71794963941912e-07\\
1006	7.65398923030677e-07\\
1011	7.59073325031643e-07\\
1016	7.52817156152861e-07\\
1021	7.46629421168716e-07\\
1026	7.40509143000626e-07\\
1031	7.34455362307593e-07\\
1036	7.28467137088109e-07\\
1041	7.22543542294867e-07\\
1046	7.16683669455505e-07\\
1051	7.10886626308023e-07\\
1056	7.05151536443282e-07\\
1061	6.99477538957631e-07\\
1066	6.93863788114472e-07\\
1071	6.88309453015039e-07\\
1076	6.82813717278334e-07\\
1081	6.77375778727215e-07\\
1086	6.71994849084734e-07\\
1091	6.66670153677788e-07\\
1096	6.61400931148139e-07\\
1101	6.56186433170338e-07\\
1106	6.51025924178466e-07\\
1111	6.45918681096048e-07\\
1116	6.40863993078144e-07\\
1121	6.35861161255662e-07\\
1126	6.30909498487379e-07\\
1131	6.26008329118604e-07\\
1136	6.21156988745256e-07\\
1141	6.16354823983015e-07\\
1146	6.11601192244223e-07\\
1151	6.06895461518773e-07\\
1156	6.02237010158929e-07\\
1161	5.97625226673252e-07\\
1166	5.93059509520994e-07\\
1171	5.88539266915356e-07\\
1176	5.84063916627836e-07\\
1181	5.79632885801126e-07\\
1186	5.75245610762012e-07\\
1191	5.70901536843333e-07\\
1196	5.66600118205524e-07\\
1201	5.62340817666751e-07\\
1206	5.58123106533223e-07\\
1211	5.53946464435718e-07\\
1216	5.4981037916962e-07\\
1221	5.45714346536707e-07\\
1226	5.4165787019391e-07\\
1231	5.376404615035e-07\\
1236	5.33661639385596e-07\\
1241	5.29720930175667e-07\\
1246	5.25817867486208e-07\\
1251	5.21951992068359e-07\\
1256	5.18122851680072e-07\\
1261	5.14330000954615e-07\\
1266	5.10573001272976e-07\\
1271	5.06851420639244e-07\\
1276	5.03164833558977e-07\\
1281	4.99512820919611e-07\\
1286	4.95894969873319e-07\\
1291	4.92310873723405e-07\\
1296	4.88760131812006e-07\\
1301	4.85242349411637e-07\\
1306	4.81757137617035e-07\\
1311	4.78304113242217e-07\\
1316	4.74882898716421e-07\\
1321	4.71493121983763e-07\\
1326	4.68134416407234e-07\\
1331	4.64806420671007e-07\\
1336	4.61508778686089e-07\\
1341	4.5824113949948e-07\\
1346	4.55003157203739e-07\\
1351	4.51794490848446e-07\\
1356	4.48614804354354e-07\\
1361	4.45463766428092e-07\\
1366	4.42341050480018e-07\\
1371	4.3924633454267e-07\\
1376	4.36179301190995e-07\\
1381	4.33139637465716e-07\\
1386	4.3012703479547e-07\\
1391	4.2714118892393e-07\\
1396	4.24181799834776e-07\\
1401	4.21248571680846e-07\\
1406	4.18341212714267e-07\\
1411	4.15459435216652e-07\\
1416	4.12602955433246e-07\\
1421	4.09771493503648e-07\\
1426	4.06964773400247e-07\\
1431	4.04182522863356e-07\\
1436	4.01424473338748e-07\\
1441	3.98690359916305e-07\\
1446	3.95979921271036e-07\\
1451	3.93292899603669e-07\\
1456	3.90629040583752e-07\\
1461	3.87988093291714e-07\\
1466	3.85369810165364e-07\\
1471	3.8277394694411e-07\\
1476	3.80200262616853e-07\\
1481	3.77648519368675e-07\\
1486	3.7511848253054e-07\\
1491	3.72609920528386e-07\\
1496	3.70122604833284e-07\\
1501	3.67656309915124e-07\\
1506	3.65210813192265e-07\\
1511	3.62785894987441e-07\\
1516	3.6038133848074e-07\\
1521	3.57996929664796e-07\\
1526	3.55632457300793e-07\\
1531	3.532877128757e-07\\
1536	3.50962490559595e-07\\
1541	3.48656587163518e-07\\
1546	3.46369802099803e-07\\
1551	3.44101937340574e-07\\
1556	3.41852797379186e-07\\
1561	3.39622189191196e-07\\
1566	3.3740992219611e-07\\
1571	3.35215808221155e-07\\
1576	3.33039661463683e-07\\
1581	3.30881298455567e-07\\
1586	3.28740538027236e-07\\
1591	3.26617201274311e-07\\
1596	3.24511111522435e-07\\
1601	3.22422094294211e-07\\
1606	3.203499772764e-07\\
1611	3.18294590286805e-07\\
1616	3.16255765244359e-07\\
1621	3.14233336135208e-07\\
1626	3.12227138985023e-07\\
1631	3.10237011827261e-07\\
1636	3.08262794673937e-07\\
1641	3.06304329486341e-07\\
1646	3.04361460146902e-07\\
1651	3.02434032430821e-07\\
1656	3.00521893978844e-07\\
1661	2.98624894269478e-07\\
1666	2.9674288459321e-07\\
1671	2.94875718025781e-07\\
1676	2.93023249402582e-07\\
1681	2.91185335293526e-07\\
1686	2.89361833978264e-07\\
1691	2.87552605420568e-07\\
1696	2.85757511246371e-07\\
1701	2.83976414718749e-07\\
1706	2.82209180714022e-07\\
1711	2.8045567570042e-07\\
1716	2.787157677155e-07\\
1721	2.76989326343052e-07\\
1726	2.75276222691818e-07\\
1731	2.73576329374245e-07\\
1736	2.71889520485318e-07\\
1741	2.70215671582271e-07\\
1746	2.68554659664354e-07\\
1751	2.66906363151647e-07\\
1756	2.65270661866332e-07\\
1761	2.63647437013515e-07\\
1766	2.62036571161493e-07\\
1771	2.60437948223763e-07\\
1776	2.58851453440031e-07\\
1781	2.57276973358168e-07\\
1786	2.55714395817239e-07\\
1791	2.54163609928054e-07\\
1796	2.52624506057583e-07\\
1801	2.51096975811615e-07\\
1806	2.4958091201775e-07\\
1811	2.48076208708767e-07\\
1816	2.46582761106586e-07\\
1821	2.4510046560705e-07\\
1826	2.43629219763307e-07\\
1831	2.42168922272972e-07\\
1836	2.40719472969725e-07\\
1841	2.39280772857259e-07\\
1846	2.37852724412664e-07\\
1851	2.36435233872644e-07\\
1856	2.35028229314772e-07\\
1861	2.33631815691595e-07\\
1866	2.32247720832556e-07\\
1871	2.30893933706009e-07\\
1876	2.29760497613456e-07\\
1881	2.3048446668548e-07\\
1886	2.3618284366687e-07\\
1891	2.42374990550486e-07\\
1896	2.46841434603742e-07\\
1901	2.49926065841093e-07\\
1906	2.52027337051433e-07\\
1911	2.53430467227047e-07\\
1916	2.54333958474855e-07\\
1921	2.54877243266474e-07\\
1926	2.55159841729733e-07\\
1931	2.55254046592038e-07\\
1936	2.55213292126019e-07\\
1941	2.5507770311549e-07\\
1946	2.54877780724513e-07\\
1951	2.5463685459129e-07\\
1956	2.54372706914431e-07\\
1961	2.54098658411469e-07\\
1966	2.53824314567842e-07\\
1971	2.53556128933306e-07\\
1976	2.53297885635327e-07\\
1981	2.53051173633846e-07\\
1986	2.52815874619959e-07\\
1991	2.52590663119215e-07\\
1996	2.52373484631352e-07\\
2001	2.52161980638775e-07\\
2006	2.51953825925958e-07\\
2011	2.51746967900694e-07\\
2016	2.51539764317168e-07\\
2021	2.51331037601104e-07\\
2026	2.51120061877085e-07\\
2031	2.50906509082877e-07\\
2036	2.5069036971691e-07\\
2041	2.50471867997978e-07\\
2046	2.50251378413938e-07\\
2051	2.50029354506333e-07\\
2056	2.49806269545252e-07\\
2061	2.49582574010607e-07\\
2066	2.49358665542008e-07\\
2071	2.49134873389119e-07\\
2076	2.48911451532478e-07\\
2081	2.48688581562754e-07\\
2086	2.48466379403911e-07\\
2091	2.4824490708536e-07\\
2096	2.4802418430609e-07\\
2101	2.47804201621145e-07\\
2106	2.47584930885317e-07\\
2111	2.47366335551569e-07\\
2116	2.47148377247212e-07\\
2121	2.46931021796471e-07\\
2126	2.46714241663148e-07\\
2131	2.46498018268261e-07\\
2136	2.4628234140689e-07\\
2141	2.46067209263463e-07\\
2146	2.45852626258581e-07\\
2151	2.45638602159306e-07\\
2156	2.45425149578706e-07\\
2161	2.45212283154291e-07\\
2166	2.45000017402126e-07\\
2171	2.44788366433945e-07\\
2176	2.44577342432835e-07\\
2181	2.4436695596972e-07\\
2186	2.44157215036404e-07\\
2191	2.43948125834514e-07\\
2196	2.43739692156283e-07\\
2201	2.43531916428903e-07\\
2206	2.43324799238672e-07\\
2211	2.43118340437177e-07\\
2216	2.42912538658925e-07\\
2221	2.42707392383669e-07\\
2226	2.42502899356718e-07\\
2231	2.42299057557267e-07\\
2236	2.42095864516986e-07\\
2241	2.41893318199557e-07\\
2246	2.41691416225135e-07\\
2251	2.41490156697725e-07\\
2256	2.41289537372244e-07\\
2261	2.41089556461401e-07\\
2266	2.40890211772454e-07\\
2271	2.40691501516148e-07\\
2276	2.40493423436664e-07\\
2281	2.40295975640583e-07\\
2286	2.40099155728157e-07\\
2291	2.39902961643438e-07\\
2296	2.3970739081016e-07\\
2301	2.39512440998591e-07\\
2306	2.39318109463248e-07\\
2311	2.39124393824722e-07\\
2316	2.38931291207929e-07\\
2321	2.38738799122545e-07\\
2326	2.38546914597413e-07\\
2331	2.38355635078608e-07\\
2336	2.38164957541968e-07\\
2341	2.37974879399859e-07\\
2346	2.37785397603014e-07\\
2351	2.37596509555839e-07\\
2356	2.37408212201463e-07\\
2361	2.37220502947466e-07\\
2366	2.37033378738353e-07\\
2371	2.36846836988823e-07\\
2376	2.3666087464944e-07\\
2381	2.36475489146061e-07\\
2386	2.36290677433497e-07\\
2391	2.36106436948684e-07\\
2396	2.35922764655636e-07\\
2401	2.35739658003612e-07\\
2406	2.35557113966242e-07\\
2411	2.35375130005995e-07\\
2416	2.35193703104528e-07\\
2421	2.3501283073714e-07\\
2426	2.34832509899785e-07\\
2431	2.3465273808498e-07\\
2436	2.34473512301018e-07\\
2441	2.3429483005943e-07\\
2446	2.34116688382043e-07\\
2451	2.33939084797545e-07\\
2456	2.33762016345736e-07\\
2461	2.3358548057807e-07\\
2466	2.33409474549817e-07\\
2471	2.33233995833557e-07\\
2476	2.3305904150678e-07\\
2481	2.32884609161531e-07\\
2486	2.3271069589188e-07\\
2491	2.32537299310156e-07\\
2496	2.32364416529011e-07\\
2501	2.32192045182824e-07\\
2506	2.32020182396625e-07\\
2511	2.31848825826419e-07\\
2516	2.31677972621325e-07\\
2521	2.31507620453051e-07\\
2526	2.31337766490398e-07\\
2531	2.3116840842754e-07\\
2536	2.30999543447999e-07\\
2541	2.3083116926631e-07\\
2546	2.30663283087885e-07\\
2551	2.30495882644473e-07\\
2556	2.30328965155812e-07\\
2561	2.30162528374093e-07\\
2566	2.29996569540882e-07\\
2571	2.29831086427942e-07\\
2576	2.29666076291229e-07\\
2581	2.29501536920886e-07\\
2586	2.29337465595614e-07\\
2591	2.2917386012069e-07\\
2596	2.29010717793324e-07\\
2601	2.28848036442479e-07\\
2606	2.2868581337849e-07\\
2611	2.28524046450591e-07\\
2616	2.28362732990821e-07\\
2621	2.28201870866069e-07\\
2626	2.28041457427918e-07\\
2631	2.27881490559789e-07\\
2636	2.27721967632145e-07\\
2641	2.27562886549385e-07\\
2646	2.27404244702603e-07\\
2651	2.27246040014688e-07\\
2656	2.27088269894128e-07\\
2661	2.26930932281793e-07\\
2666	2.26774024607849e-07\\
2671	2.26617544833565e-07\\
2676	2.26461490406618e-07\\
2681	2.26305859308353e-07\\
2686	2.26150649009829e-07\\
2691	2.25995857511521e-07\\
2696	2.25841482304581e-07\\
2701	2.25687521410338e-07\\
2706	2.2553397234121e-07\\
2711	2.25380833139146e-07\\
2716	2.25228101336942e-07\\
2721	2.25075774994922e-07\\
2726	2.24923851669732e-07\\
2731	2.24772329442093e-07\\
2736	2.24621205893758e-07\\
2741	2.2447047912679e-07\\
2746	2.24320146740921e-07\\
2751	2.24170206862142e-07\\
2756	2.24020657114426e-07\\
2761	2.23871495646369e-07\\
2766	2.23722720104924e-07\\
2771	2.23574328657528e-07\\
2776	2.23426318975498e-07\\
2781	2.23278689249606e-07\\
2786	2.231314371762e-07\\
2791	2.22984560967003e-07\\
2796	2.22838058342787e-07\\
2801	2.22691927538059e-07\\
2806	2.22546166298686e-07\\
2811	2.22400772881769e-07\\
2816	2.22255745059912e-07\\
2821	2.22111081112918e-07\\
2826	2.21966778836881e-07\\
2831	2.21822836535444e-07\\
2836	2.2167925203145e-07\\
2841	2.21536023650738e-07\\
2846	2.21393149243959e-07\\
2851	2.21250627158698e-07\\
2856	2.21108455271336e-07\\
2861	2.20966631955735e-07\\
2866	2.20825155114394e-07\\
2871	2.20684023144152e-07\\
2876	2.20543233972845e-07\\
2881	2.20402786022066e-07\\
2886	2.20262677245235e-07\\
2891	2.20122906086764e-07\\
2896	2.19983470529372e-07\\
2901	2.19844369041084e-07\\
2906	2.19705599629284e-07\\
2911	2.19567160787806e-07\\
2916	2.19429050553485e-07\\
2921	2.19291267441136e-07\\
2926	2.19153809513234e-07\\
2931	2.19016675310399e-07\\
2936	2.18879862919806e-07\\
2941	2.1874337090571e-07\\
2946	2.18607197384033e-07\\
2951	2.18471340944938e-07\\
2956	2.1833579973074e-07\\
2961	2.18200572351095e-07\\
2966	2.18065656975474e-07\\
2971	2.17931052237851e-07\\
2976	2.17796756336138e-07\\
2981	2.17662767926972e-07\\
2986	2.17529085233538e-07\\
2991	2.17395706935792e-07\\
2996	2.17262631282626e-07\\
3001	2.17129856977496e-07\\
3006	2.16997382298589e-07\\
3011	2.16865205970198e-07\\
3016	2.16733326294928e-07\\
3021	2.16601742022595e-07\\
3026	2.1647045148131e-07\\
3031	2.16339453441144e-07\\
3036	2.16208746255892e-07\\
3041	2.16078328718404e-07\\
3046	2.15948199210086e-07\\
3051	2.15818356543308e-07\\
3056	2.15688799123442e-07\\
3061	2.15559525785458e-07\\
3066	2.15430534961348e-07\\
3071	2.15301825507087e-07\\
3076	2.15173395877265e-07\\
3081	2.15045244947218e-07\\
3086	2.14917371197906e-07\\
3091	2.14789773526835e-07\\
3096	2.14662450440264e-07\\
3101	2.14535400855406e-07\\
3106	2.14408623301451e-07\\
3111	2.14282116714881e-07\\
3116	2.14155879647678e-07\\
3121	2.14029911056778e-07\\
3126	2.13904209519071e-07\\
3131	2.13778774010071e-07\\
3136	2.13653603128827e-07\\
3141	2.13528695871087e-07\\
3146	2.13404050858478e-07\\
3151	2.13279667101865e-07\\
3156	2.13155543245706e-07\\
3161	2.13031678322847e-07\\
3166	2.12908070997861e-07\\
3171	2.12784720320479e-07\\
3176	2.12661624979453e-07\\
3181	2.12538784040404e-07\\
3186	2.12416196211042e-07\\
3191	2.12293860574204e-07\\
3196	2.12171775859235e-07\\
3201	2.12049941165266e-07\\
3206	2.11928355241644e-07\\
3211	2.11807017205255e-07\\
3216	2.11685925826827e-07\\
3221	2.11565080237246e-07\\
3226	2.11444479226007e-07\\
3231	2.11324121942386e-07\\
3236	2.11204007193625e-07\\
3241	2.11084134141415e-07\\
3246	2.10964501614489e-07\\
3251	2.10845108787791e-07\\
3256	2.10725954508832e-07\\
3261	2.10607037968843e-07\\
3266	2.10488358032554e-07\\
3271	2.10369913904745e-07\\
3276	2.10251704468927e-07\\
3281	2.10133728942379e-07\\
3286	2.10015986225124e-07\\
3291	2.09898475549036e-07\\
3296	2.09781195831647e-07\\
3301	2.09664146315824e-07\\
3306	2.09547325933919e-07\\
3311	2.09430733945763e-07\\
3316	2.09314369297888e-07\\
3321	2.0919823126127e-07\\
3326	2.09082318800669e-07\\
3331	2.08966631198152e-07\\
3336	2.0885116743714e-07\\
3341	2.087359268119e-07\\
3346	2.08620908318622e-07\\
3351	2.08506111259634e-07\\
3356	2.08391534647442e-07\\
3361	2.08277177797074e-07\\
3366	2.08163039734767e-07\\
3371	2.08049119787505e-07\\
3376	2.07935416996348e-07\\
3381	2.07821930698934e-07\\
3386	2.07708659951728e-07\\
3391	2.07595604100933e-07\\
3396	2.0748276221612e-07\\
3401	2.07370133655974e-07\\
3406	2.07257717503215e-07\\
3411	2.07145513125553e-07\\
3416	2.07033519621863e-07\\
3421	2.06921736369048e-07\\
3426	2.06810162479126e-07\\
3431	2.06698797337509e-07\\
3436	2.06587640070885e-07\\
3441	2.06476690073896e-07\\
3446	2.06365946486457e-07\\
3451	2.06255408713806e-07\\
3456	2.06145075906304e-07\\
3461	2.06034947479031e-07\\
3466	2.05925022597255e-07\\
3471	2.05815300685255e-07\\
3476	2.05705780922464e-07\\
3481	2.05596462739192e-07\\
3486	2.05487345325546e-07\\
3491	2.05378428124269e-07\\
3496	2.05269710341456e-07\\
3501	2.05161191427641e-07\\
3506	2.0505287059932e-07\\
3511	2.04944747316883e-07\\
3516	2.04836820810469e-07\\
3521	2.04729090547797e-07\\
3526	2.04621555770786e-07\\
3531	2.04514215957527e-07\\
3536	2.04407070361605e-07\\
3541	2.04300118470337e-07\\
3546	2.04193359551977e-07\\
3551	2.04086793103834e-07\\
3556	2.03980418408241e-07\\
3561	2.03874234971741e-07\\
3566	2.03768242086325e-07\\
3571	2.03662439266137e-07\\
3576	2.0355682581749e-07\\
3581	2.03451401264276e-07\\
3586	2.03346164925058e-07\\
3591	2.03241116332786e-07\\
3596	2.03136254819647e-07\\
3601	2.03031579927283e-07\\
3606	2.02927091001612e-07\\
3611	2.0282278759451e-07\\
3616	2.02718669065662e-07\\
3621	2.0261473497704e-07\\
3626	2.02510984700447e-07\\
3631	2.02407417806661e-07\\
3636	2.02304033680597e-07\\
3641	2.02200831902559e-07\\
3646	2.02097811871089e-07\\
3651	2.01994973177993e-07\\
3656	2.01892315234737e-07\\
3661	2.01789837641668e-07\\
3666	2.01687539825823e-07\\
3671	2.01585421396604e-07\\
3676	2.01483481794438e-07\\
3681	2.01381720639853e-07\\
3686	2.01280137386763e-07\\
3691	2.0117873166578e-07\\
3696	2.01077502945097e-07\\
3701	2.00976450865471e-07\\
3706	2.00875574905973e-07\\
3711	2.00774874718868e-07\\
3716	2.0067434979877e-07\\
3721	2.0057399980928e-07\\
3726	2.00473824259063e-07\\
3731	2.00373822820807e-07\\
3736	2.00273995018049e-07\\
3741	2.0017434053588e-07\\
3746	2.00074858912199e-07\\
3751	1.99975549839957e-07\\
3756	1.99876412870576e-07\\
3761	1.99777447708409e-07\\
3766	1.99678653918602e-07\\
3771	1.99580031217025e-07\\
3776	1.99481579184277e-07\\
3781	1.99383297545846e-07\\
3786	1.99285185895949e-07\\
3791	1.99187243973141e-07\\
3796	1.99089471383348e-07\\
3801	1.98991867873229e-07\\
3806	1.9889443306433e-07\\
3811	1.98797166715976e-07\\
3816	1.98700068461053e-07\\
3821	1.98603138067916e-07\\
3826	1.98506375184854e-07\\
3831	1.98409779589202e-07\\
3836	1.98313350942016e-07\\
3841	1.98217089032343e-07\\
3846	1.9812099353162e-07\\
3851	1.98025064239601e-07\\
3856	1.97929300841534e-07\\
3861	1.97833703145144e-07\\
3866	1.97738270848455e-07\\
3871	1.97643003766831e-07\\
3876	1.97547901609131e-07\\
3881	1.97452964197774e-07\\
3886	1.97358191252801e-07\\
3891	1.97263582607234e-07\\
3896	1.97169137990883e-07\\
3901	1.9707485724191e-07\\
3906	1.96980740100386e-07\\
3911	1.96886786411244e-07\\
3916	1.96792995923136e-07\\
3921	1.96699368485183e-07\\
3926	1.96605903854668e-07\\
3931	1.96512601885338e-07\\
3936	1.96419462342665e-07\\
3941	1.96326485084225e-07\\
3946	1.96233669881534e-07\\
3951	1.96141016593823e-07\\
3956	1.96048524996194e-07\\
3961	1.95956194951236e-07\\
3966	1.95864026238102e-07\\
3971	1.95772018719844e-07\\
3976	1.95680172180088e-07\\
3981	1.95588486477548e-07\\
3986	1.95496961394238e-07\\
3991	1.95405596788474e-07\\
3996	1.95314392445093e-07\\
4001	1.95223348218343e-07\\
4006	1.95132463889763e-07\\
4011	1.95041739308616e-07\\
4016	1.94951174251136e-07\\
4021	1.94860768558665e-07\\
4026	1.94770522003846e-07\\
4031	1.94680434420687e-07\\
4036	1.94590505575553e-07\\
4041	1.94500735292668e-07\\
4046	1.94411123328391e-07\\
4051	1.94321669494066e-07\\
4056	1.94232373536518e-07\\
4061	1.94143235253649e-07\\
4066	1.94054254377604e-07\\
4071	1.93965430690893e-07\\
4076	1.93876763913289e-07\\
4081	1.93788253811843e-07\\
4086	1.93699900088159e-07\\
4091	1.9361170248851e-07\\
4096	1.93523660700169e-07\\
4101	1.93435774447844e-07\\
4106	1.93348043398043e-07\\
4111	1.93260467255719e-07\\
4116	1.93173045666941e-07\\
4121	1.93085778309538e-07\\
4126	1.92998664805186e-07\\
4131	1.92911704808597e-07\\
4136	1.92824897919586e-07\\
4141	1.92738243765166e-07\\
4146	1.92651741918914e-07\\
4151	1.92565391979135e-07\\
4156	1.92479193494831e-07\\
4161	1.92393146035211e-07\\
4166	1.92307249122623e-07\\
4171	1.92221502297802e-07\\
4176	1.92135905054058e-07\\
4181	1.9205045690394e-07\\
4186	1.91965157313737e-07\\
4191	1.91880005764566e-07\\
4196	1.91795001692669e-07\\
4201	1.9171014455168e-07\\
4206	1.91625433749598e-07\\
4211	1.91540868710314e-07\\
4216	1.91456448816007e-07\\
4221	1.91372173463511e-07\\
4226	1.91288042006951e-07\\
4231	1.91204053815851e-07\\
4236	1.91120208222599e-07\\
4241	1.91036504570114e-07\\
4246	1.90952942163527e-07\\
4251	1.90869520322496e-07\\
4256	1.90786238332538e-07\\
4261	1.90703095491438e-07\\
4266	1.90620091062155e-07\\
4271	1.90537224324064e-07\\
4276	1.90454494525486e-07\\
4281	1.90371900928878e-07\\
4286	1.90289442768464e-07\\
4291	1.90207119293342e-07\\
4296	1.90124929724582e-07\\
4301	1.90042873302326e-07\\
4306	1.89960949240325e-07\\
4311	1.89879156773354e-07\\
4316	1.89797495113464e-07\\
4321	1.89715963488917e-07\\
4326	1.89634561107466e-07\\
4331	1.89553287203146e-07\\
4336	1.89472140990629e-07\\
4341	1.89391121704165e-07\\
4346	1.89310228564755e-07\\
4351	1.89229460816398e-07\\
4356	1.89148817692327e-07\\
4361	1.89068298449697e-07\\
4366	1.88987902336856e-07\\
4371	1.88907628628966e-07\\
4376	1.8882747659083e-07\\
4381	1.88747445515919e-07\\
4386	1.88667534692891e-07\\
4391	1.88587743441038e-07\\
4396	1.88508071074372e-07\\
4401	1.88428516938408e-07\\
4406	1.88349080377901e-07\\
4411	1.88269760770849e-07\\
4416	1.88190557493108e-07\\
4421	1.88111469955048e-07\\
4426	1.880324975678e-07\\
4431	1.87953639779297e-07\\
4436	1.87874896037227e-07\\
4441	1.87796265829129e-07\\
4446	1.87717748642216e-07\\
4451	1.87639344002733e-07\\
4456	1.87561051443837e-07\\
4461	1.87482870533802e-07\\
4466	1.87404800846123e-07\\
4471	1.87326841993661e-07\\
4476	1.87248993596678e-07\\
4481	1.87171255308999e-07\\
4486	1.87093626792705e-07\\
4491	1.87016107751091e-07\\
4496	1.86938697891962e-07\\
4501	1.868613969652e-07\\
4506	1.86784204723451e-07\\
4511	1.86707120957628e-07\\
4516	1.86630145464317e-07\\
4521	1.86553278080255e-07\\
4526	1.86476518646166e-07\\
4531	1.86399867044271e-07\\
4536	1.86323323157939e-07\\
4541	1.86246886907623e-07\\
4546	1.86170558217473e-07\\
4551	1.86094337053467e-07\\
4556	1.86018223381677e-07\\
4561	1.85942217202517e-07\\
4566	1.85866318518605e-07\\
4571	1.85790527369464e-07\\
4576	1.85714843793356e-07\\
4581	1.85639267862806e-07\\
4586	1.855637996483e-07\\
4591	1.85488439254147e-07\\
4596	1.85413186784477e-07\\
4601	1.85338042373559e-07\\
4606	1.85263006150848e-07\\
4611	1.85188078276674e-07\\
4616	1.8511325890759e-07\\
4621	1.85038548224443e-07\\
4626	1.84963946405052e-07\\
4631	1.8488945365115e-07\\
4636	1.84815070158011e-07\\
4641	1.84740796140683e-07\\
4646	1.84666631811276e-07\\
4651	1.84592577400906e-07\\
4656	1.8451863312946e-07\\
4661	1.84444799236851e-07\\
4666	1.84371075952426e-07\\
4671	1.84297463519993e-07\\
4676	1.84223962171231e-07\\
4681	1.84150572151009e-07\\
4686	1.84077293690212e-07\\
4691	1.84004127034254e-07\\
4696	1.83931072408658e-07\\
4701	1.83858130050843e-07\\
4706	1.83785300179742e-07\\
4711	1.83712583020781e-07\\
4716	1.83639978780217e-07\\
4721	1.8356748766955e-07\\
4726	1.8349510987771e-07\\
4731	1.83422845595954e-07\\
4736	1.83350694991236e-07\\
4741	1.8327865823372e-07\\
4746	1.83206735467751e-07\\
4751	1.83134926834413e-07\\
4756	1.83063232449051e-07\\
4761	1.82991652423383e-07\\
4766	1.82920186834899e-07\\
4771	1.82848835758544e-07\\
4776	1.82777599240149e-07\\
4781	1.82706477317351e-07\\
4786	1.82635469992485e-07\\
4791	1.82564577262669e-07\\
4796	1.82493799087904e-07\\
4801	1.82423135419304e-07\\
4806	1.82352586168133e-07\\
4811	1.82282151235678e-07\\
4816	1.82211830484118e-07\\
4821	1.82141623765394e-07\\
4826	1.82071530888841e-07\\
4831	1.82001551651394e-07\\
4836	1.81931685809789e-07\\
4841	1.81861933105377e-07\\
4846	1.81792293239439e-07\\
4851	1.81722765895905e-07\\
4856	1.81653350716426e-07\\
4861	1.81584047328392e-07\\
4866	1.81514855314003e-07\\
4871	1.81445774241293e-07\\
4876	1.81376803638192e-07\\
4881	1.81307943018371e-07\\
4886	1.81239191850988e-07\\
4891	1.81170549588369e-07\\
4896	1.81102015641622e-07\\
4901	1.81033589409529e-07\\
4906	1.80965270250582e-07\\
4911	1.80897057507197e-07\\
4916	1.80828950486155e-07\\
4921	1.80760948481115e-07\\
4926	1.8069305074826e-07\\
4931	1.80625256536048e-07\\
4936	1.80557565059602e-07\\
4941	1.80489975520498e-07\\
4946	1.80422487087739e-07\\
4951	1.80355098926546e-07\\
4956	1.80287810170535e-07\\
4961	1.80220619950558e-07\\
4966	1.80153527368387e-07\\
4971	1.80086531529379e-07\\
4976	1.80019631510654e-07\\
4981	1.79952826393571e-07\\
4986	1.79886115234803e-07\\
4991	1.79819497100318e-07\\
4996	1.7975297103353e-07\\
};
\end{axis}
\end{tikzpicture}%

%% file: proof_main.tex
\section{Supplementary proofs}
\subsection{Proof of \Cref{th:gen_rates}}\label{proof:gen_rates}
We preface the proof with two elementary results.
\begin{lemma}\label{lm:aff_proj}
  Consider $n_h \leq n$, $y \in \bbR^n$, $b \in \bbR^{n_h}$ and $W \in \bbR^{n \times n_h}$, with $W$ being of full rank. It holds that:
  \begin{equation*}
    \argmin_{v \in \bbR^n : W^{\top} v = b} \frac{1}{2} \norm{ v + y}^2 = - y + W(W^{\top} W)^{-1} (W^{\top} y + b)
  \end{equation*}
\end{lemma}
\begin{corollary}\label{cor:proj_Vx}
If $x \in \bbR^n$ is such that $\nabla h(x)$ is of full rank, then:
  \begin{equation}\label{eq:proj_auxtan}
    -\nabla_V f(x) = \argmin_{v \in V(x)} \frac{1}{2} \norm{v + \nabla f(x)}^2 = - \nabla f + \nabla h(x) D_h(x) \nabla h(x)^{\top} \nabla f(x)\, ,
  \end{equation}
  where $D_h(x) := (\nabla h(x)^{\top} \nabla h(x))^{-1}$.
\end{corollary}

The following proposition is the key element in our proof. It mainly follows from a Taylor expansion of $\Lambda_M$.

\begin{proposition}[Discrete Lyapunov function]\label{prop:lyap_stoch}
    Let \Cref{hyp:cont_model}--\ref{hyp:disc_model} hold and let $\overline{M}$ be the one of Theorem~\ref{th:cont_time}. If for all $k \in \bbN$, $\gamma_k \leq \alpha_m^{-1}$, then for all $M \geq \overline{M}$, it holds:
\begin{equation}\label{eq:lyap_sto_dec}
    \bbE_k [\Lambda_M(x_{k+1})]- \Lambda_M(x_k) \leq - \gamma_k \norm{ v_k}^2 \left( 1 - \frac{L_f + M L_h}{2} \gamma_k\right) + \frac{L_f + ML_h}{2}\sigma^2 \gamma_k^2 \, .
\end{equation}
\end{proposition}
\begin{proof}
Since $f$ is gradient Lipschitz continuous on $K$ and $\bbE_k[\eta_{k+1}] = 0$, it holds that
\begin{equation}\label{eq:f_disc_decr}
    \begin{split}
        \bbE_k[f(x_{k+1})] - f(x_k) &\leq \gamma_k \nabla f(x_k)^{\top}  v_k + \frac{L_f}{2}\gamma_k^2 \bbE_k[\norm{v_k + \eta_{k+1}}^2] \\
        &\leq \gamma_k \nabla f(x_k)^{\top}  v_k  + \frac{L_f}{2}\gamma_k^2 (\norm{v_k}^2 + \sigma^2) \\
        &\leq - \gamma_k \norm{v_k}^2\left( 1 - \frac{L_f}{2}\gamma_k\right) + \gamma_k(\nabla f(x_k) + v_k)^{\top} v_k + \frac{L_f}{2}\gamma_k^2\sigma^2\\
        &\leq - \gamma_k \norm{v_k}^2\left( 1 - \frac{L_f}{2}\gamma_k\right) + \gamma_k M_1 \norm{h(x_k)} + \frac{L_f}{2}\gamma_k^2\sigma^2 \, ,
    \end{split}
\end{equation}
where the second inequality follows from \Cref{hyp:disc_model}-\ref{hyp:var_bound} and the last inequality follows from Equation~\eqref{eq:err_f}.

Similarly, since $h$ is gradient Lipschitz on $K$, we obtain:
\begin{equation}\label{eq:h_disc_decr}
\begin{split}
        \bbE_k[\norm{h(x_{k+1})}]  &\leq \bbE_k[\norm{h(x_{k}) + \gamma_k \nabla h(x_k)^{\top}( v_k + \eta_{k+1})}] + \frac{L_h}{2} \gamma^2_k \bbE_k[\norm{v_k + \eta_{k+1}}^2] \\
            &\leq \norm{h(x_{k}) + \gamma_k \nabla h(x_k)^{\top}v_k} +  \frac{L_h}{2} \gamma^2_k \norm{v_k}^2 + \frac{L_h}{2} \gamma_k^2 \sigma^2 \\
            &\leq \norm{ h(x_k) - \gamma_k \nabla h(x_k)^{\top} \nabla h(x_k) A(x_k) h(x_k)} + \frac{L_h}{2} \gamma^2_k (\norm{v_k}^2 + \sigma^2) \\
            &\leq (1 - \alpha_m \gamma_k) \norm{h(x_k)} + \frac{L_h}{2} \gamma^2_k (\norm{v_k}^2 + \sigma^2) \, ,
\end{split}    
\end{equation}
where in the second inequality we have used that $\eta_{k+1} \in V(x_k)$ and in the last inequality our choice of $(\gamma_k)$ with \Cref{hyp:cont_model}-\ref{hyp:hA_eigenv}.

Combining Equations~\eqref{eq:f_disc_decr} and \eqref{eq:h_disc_decr} with the fact that $M \geq \overline{M} = M_1/\alpha_m$ completes the proof.
\end{proof}

The following corollary is obtained by telescoping Equation~\eqref{eq:lyap_sto_dec}.
\begin{corollary}\label{cor:lyap_telesc}
    Under the assumptions of Theorem~\ref{th:gen_rates}, for $N >0$, it holds that: 
    \begin{equation}
        \sum_{i=0}^{N-1} \bbE[\norm{v_i}^2] \leq 2 \frac{\bbE[\Lambda_M(x_0)] - \bbE[\Lambda_M(x_{N-1})]}{\gamma} + N \gamma (L_f + M L_h) \sigma^2 \,.
    \end{equation}
\end{corollary}
\subsubsection{Deterministic case: $\sigma = 0$}\label{proof:det_rates}
Fix $\sigma = 0$, from Corollary~\ref{cor:lyap_telesc} we obtain:
 \begin{equation*}
    \gamma \sum_{i=0}^{N-1} \norm{v_i}^2 \leq 2(\Lambda_{M}(x_0) - \Lambda_{M}(x_{N-1}))\, .
  \end{equation*}
  This implies Equation~\eqref{eq:det_rates} and shows that $ \norm{v_k} \rightarrow 0$. Now notice that
  \begin{equation*}
    \norm{v_k}^2 = \norm{\nabla_V f(x_k)}^2 + \norm{\nabla h(x_k) A(x_k) h(x_k)}^2 \, .
  \end{equation*}
  Since by ~\Cref{hyp:cont_model} both $A$ and $\nabla h$ are of full rank on $K$, this implies that $\norm{h(x_k)} \rightarrow 0$. Thus, if $x^*$ is an accumulation point of $(x_k)$, then it must satisfy $h(x^*) = 0$, or, in other words, $x^* \in \cM$. Finally, by continuity of $\OF$, we also have $0 = \lim_{k \rightarrow \infty}\norm{\OF(x_k)} = \norm{\OF(x^*)} = \norm{\Grad f(x^*)}$, which completes the proof.

\subsubsection{The general case}

Using Corollary~\ref{cor:lyap_telesc}, we obtain:
\begin{equation*}
\begin{split}
        \bbE\left[\norm{v_{\hat k}}^2\right] = \frac{1}{N}\sum_{i=0}^{N-1} \bbE[\norm{v_i}^2] &\leq 2\frac{\bbE[\Lambda_M(x_0)] - \bbE[\Lambda_M(x_{N-1})]}{N \gamma} + \gamma(L_f + M L_h) \sigma^2 \\
        &\leq  \frac{2 D_M}{N  \gamma} + \gamma(L_f + M L_h) \sigma^2
\end{split}
\end{equation*}
Therefore, 
\begin{equation*}
    \begin{split}
        \bbE\left[\norm{v_{\hat k}}^2\right] &\leq \frac{2 D_M}{N  \gamma} + \bar D (L_f + M L_h) \frac{\sigma}{\sqrt{N}} \\
        &\leq \frac{2 D_M}{N} \max(\alpha_m, L_f + M L_h, \bar D^{-1} \sigma \sqrt{N}) + \bar D (L_f + M L_h) \frac{\sigma}{\sqrt{N}} \\
        &\leq \frac{2 D_M}{N}\left( \alpha_m + L_f + M L_h\right) + \frac{\sigma}{\sqrt{N}} \left( \bar D (L_f + M L_h) + \frac{2 D_M}{\bar D}\right) \, ,
    \end{split}
\end{equation*}
which completes the proof.

\subsection{Safe step size}\label{proof:safe_step}
In this section, we discuss \Cref{hyp:disc_model}-\ref{hyp:iter_bound}. We establish that if the sequence $(\eta_{k+1})$ is bounded (which is the case in both the deterministic and finite-sum settings), a sufficiently small step size forces the algorithm to stay in $K$. To formulate this theorem, we denote $C_h, C_f$ as the Lipschitz constants of $h,f$ on $K$ and define $ C_A:= \sup_{x \in K} \norm{\nabla h A}$.

\begin{proposition}[safe step-size]\label{prop:safe_step}
  Assume \Cref{hyp:cont_model} and \Cref{hyp:disc_model}-\ref{hyp:fh_Lipgrad} and that there is a constant $b>0$ such that $\sup_{k} \norm{\eta_{k+1}} \leq b$. Consider $\delta >0$ and let $\gamma$ be defined as
  \begin{equation}\label{eq:safe_step}
    \gamma =  \min\left(\frac{1}{\alpha_m}, \frac{\delta}{C_h \sqrt{2 \left(C^2_A r_1^2 + C_f^2 + b^2\right)}}, \frac{\alpha_m}{2 L_h C_A^2 r_1}, \frac{2\alpha_m (r_1-\delta)}{3 L_h (C_f^2+ b^2)}\right)\, .
  \end{equation}
  If $(\gamma_k)$ is bounded by $\gamma$ and $\norm{h(x_0)} \leq r_1- \delta$, then the sequence $(x_k)$ produced by \algo\ remains in $K$.
\end{proposition}
\begin{proof}
Let $k \in \bbN$ be such that $\norm{h(x_k)} \leq r_1 - \delta$, we will show that $\norm{h(x_{k+1})} \leq r_1 - \delta$, which will complete the proof by an immediate induction. 

Denote $C_h$ the Lipschitz constant of $h$ on $K$ and notice that if $\norm{x_{k+1} - x_k} \leq \delta/C_h$, then $x_{k+1} \in K$. Indeed, assume that $x_{k+1} \notin K$ and denote for $t \in [0, 1]$, $x_t = x_k + t(x_{k+1}- x_k)$. Let $u = \inf \{t \in [0,1]: x_t \notin K \}$ and note that by continuity of $h$, $\norm{h(x_u)} = r_1$. This implies that
\begin{equation*}
    \delta \leq \norm{h(x_u)} - \norm{h(x_k)} \leq \norm{h(x_u) -h(x_k)} \leq C_h\norm{x_u - x_k} \leq u \delta \, .
\end{equation*}
Thus, $u$ must be equal to $1$, which is a contradiction. 

Now, 
\begin{equation*}
    \norm{x_{k+1} - x_k}^2 \leq 2\gamma_k^2 (\norm{v_k}^2 + b^2) \, ,
\end{equation*}
and since $\nabla_V f$ is orthogonal to $\nabla h A h$, we obtain:
\begin{equation}\label{eq:safe_vkbound}
    \norm{v_k}^2 \leq \norm{\nabla h(x_k) A(x_k) h(x_k)}^2 + \norm{\nabla_V f}^2 \leq C^2_A \norm{h(x_k)}^2 + C_f^2  \leq C^2_A r_1^2+ C_f^2\, ,
\end{equation}
Hence, we get,
\begin{equation*}
    \norm{x_{k+1} - x_k}^2 \leq 2 \gamma^2 (C^2_A r_1^2+ C_f^2 + b^2) \leq \frac{\delta^2}{C_h^2} \, ,
\end{equation*}
which shows that $x_{k+1}$ remain in $K$. Now, since $x_k, x_{k+1} \in K$ and $\nabla h$ is $L_h$-Lipschitz on $K$, it holds that:
\begin{equation*}
  \norm{h(x_{k+1})- h(x_k) - \gamma_k \nabla h(x_k)^{\top} v_k} \leq \frac{L_h}{2} \norm{x_{k+1} - x_k}^2 \, ,
\end{equation*}
where we have used the fact that $\eta_{k+1} \in V(x_k)$.
Thus,
\begin{equation*}
  \norm{h(x_{k+1})} \leq \norm{h(x_k) + \gamma_k \nabla h(x_k)^{\top} v_k} +\frac{L_h}{2} \norm{x_{k+1} - x_k}^2  \, .
\end{equation*}
Recall that $\cI \in \bbR^{n_h\times n_h}$ denotes the identity matrix, it holds that:
\begin{equation*}
  \begin{split}
        \norm{h(x_{k+1})} &\leq \norm{( \cI - \gamma_k \nabla h(x_k)^{\top} \nabla h(x_k) A(x_k)) h(x_k)}  +\frac{L_h}{2} \norm{x_{k+1} - x_k}^2  \\
         &\leq (1 - \alpha_m \gamma_k) \norm{h(x_k)}  +\frac{L_h}{2} \norm{x_{k+1} - x_k}^2 \,.
  \end{split}
\end{equation*}
Examining Equation~\eqref{eq:safe_vkbound} we can actually obtain a tighter upper bound on $\norm{x_{k+1} - x_k}^2$
\begin{equation*}
    \norm{x_{k+1} - x_k}^2 \leq 2 b^2 + 2 C^2_A r_1 \norm{h(x_k)} + 2 C_f^2\, .
\end{equation*}
And finally: 
\begin{equation*}
    \begin{split}
           \norm{h(x_{k+1})} \leq \norm{h(x_k)} + \gamma_k \norm{h(x_k)}( \gamma_k L_h C_A^2 r_1 -  \alpha_m ) + \gamma_k^2 L_h(C_f^2 + b^2) \, ,
    \end{split}
\end{equation*}
Since $\gamma \leq \alpha_m / (2L_h C_A^2 r_1)$, it holds that:
\begin{equation*}
  \norm{h(x_{k+1})} \leq \norm{h(x_k)} - \alpha_m \gamma_k \norm{h(x_k)}/2 + \gamma_k^2 L_h(C_f^2 + b^2) \, .
\end{equation*}
Therefore, if $\alpha_m \norm{h(x_{k}) } \geq 2 \gamma_k(L_h C_f^2 + b^2)$, then $\norm{h(x_{k+1})} \leq \norm{h(x_k)}$. Otherwise,
\begin{equation*}
  \norm{h(x_{k+1})} \leq \norm{h(x_{k})} +\gamma_k^2 L_h(C_f^2 + b^2)  \leq (L_h C_f^2 +b^2)\gamma_k \left(\frac{2}{\alpha_m} + \gamma_k \right) \leq \frac{3(L_h C_f^2 +b^2) \gamma}{2\alpha_m} \leq r_1 -\delta \, ,
\end{equation*}
where the last inequality comes from our choice of $\gamma$.
\end{proof}
\begin{remark}\label{rm:safe_step}
Although equation ~\eqref{eq:safe_step} may be intractable, it shows that the iterates remain in $K$ for a sufficiently small $\gamma$. Therefore, we can combine our algorithm with standard line search techniques (see \cite{NoceWrig06}). For example, if we set a threshold $\overline{\gamma}$, we check whether iterates with step sizes smaller than the threshold remain in $K$. If this is not the case, the threshold is divided by $2$. Such a change of the threshold value can only occur finitely often, so that the convergence rates of \Cref{th:gen_rates} remain true.
\end{remark}

%% file: proof_reduced.tex
\subsection{Proof of \Cref{th:reduced_rates}}\label{app:reduced_proofs}

We start from the observation that explains how looks the solution to the projection on $\tilde{V}(x) = \{ v \in \bbR^n : \nabla H(x)^\top v = 0 \}$, where $H(x) = \norm{h(x)}^2 / 2$. 
\begin{corollary}\label{cor:tildeV_projections}
    Let $f\colon \bbR^n \to \bbR$ be a differentiable function and let $\cM = \{ x \in \bbR^n : H(x) = 0\}$. Then for any $x$ such that $\nabla h(x)$ is of full rank it holds
    \[
        \nabla_{\tV} f(x) = \nabla f(x) + \lambda(x) \cdot \nabla H(x),
    \]
    where 
    \[
        \lambda(x) = \begin{cases}
            0, & x \in \cM \\
            -\frac{\nabla H(x)^\top \nabla f(x)}{\norm{\nabla H(x)}^2}, & x \not \in \cM
        \end{cases}
    \]
\end{corollary}
\begin{proof}
    Apply \Cref{lm:aff_proj} with $n_h = 1, W = \nabla H(x)$ and $b = 0$ for $x \not \in \cM$.
\end{proof}
\begin{remark}
    Even through $\lambda(x) \to -\infty$ as $x \to \cM$, we have that projected gradients are always bounded:
    \[
        \norm{\nabla_{\tV} f(x)} \leq \norm{\nabla f(x)} + \frac{\vert \nabla H(x)^\top \nabla f(x) \vert }{\norm{\nabla H(x)}^2} \cdot \norm{\nabla H(x)} \leq 2 \norm{\nabla f(x)}.
    \]
\end{remark}

Next we provide a lemma that guarantees for $H(x_k)$ that under the specific choice of $A(x) = \alpha(x)\cI$ where $\alpha(x) =  \alpha \cdot \frac{H(x)}{\norm{\nabla H(x)}^2}$ for a constant $\alpha > 0$ if $x \not \in \cM$ and $\alpha(x) = 0$ if $x \in \cM$. 

\begin{lemma}\label{lm:reduced_constrain_rate}
    Assume \Cref{hyp:cont_model}'-\ref{hyp:disc_model}. Assume that for any $k > 0$, $\alpha \gamma_k \leq 1$. Define $v_k = -\alpha(x_k) \nabla H(x_k) - \nabla_{\tV} f(x_k)$. It holds
    \begin{align*}
        \bbE[H(x_{k+1})] &\leq H(x_0)  \cdot \prod_{j=0}^k (1 - \alpha \gamma_j)  + \bbE\left[\frac{L_H}{2} \sum_{j=0}^k \gamma_j^2 (\norm{v_j}^2 +\sigma^2)\prod_{\ell=j+1}^k (1 - \alpha \gamma_\ell)\right]\, .
    \end{align*}
    Furthermore, if $(\gamma_k)$ is a non-increasing sequence, then for all $N \in \bbN$,
    \begin{equation*}
        \bbE\left[\sum_{k=0}^{N-1} \gamma_k H(x_{k})\right] \leq \frac{1}{\alpha} H(x_0) + \frac{L_H}{2 \alpha} \bbE\left[ \sum_{k=0}^{N-1} \gamma_k^2 \norm{v_k}^2\right]  + \frac{L_H}{2\alpha} \sum_{k=0}^{N-1} \gamma_k^2 \sigma^2 \, .
    \end{equation*}
\end{lemma}
\begin{proof}
    Since $h(x)$ is Lispchitz and has Lipschitz-continuous gradients we have that $H(x)$ also has Lipschitz-continuous gradients with constant $L_H$, thus for any $k \in \bbN$
    \[
        \bbE_k[H(x_{k+1})] \leq H(x_k) + \gamma_k \nabla H(x_k)^\top v_k + \frac{L_H \gamma_k^2}{2} \bbE_k[\norm{v_k + \eta_{k+1}}^2].
    \]
    Notice that $\nabla_{\tV} f(x_k)$ is orthogonal to $\nabla H(x_k)$, thus 
    \[
        \nabla H(x_k)^\top v_k  = -\alpha(x_k) \norm{\nabla H(x_k)}^2 = -\alpha H(x_k)
    \] by definition of $\alpha(x_k)$. Also notice that $\bbE_k[\norm{v_k + \eta_{k+1}}^2] \leq \norm{v_k}^2 + \sigma^2$. Therefore
    \[
        \bbE_k[H(x_{k+1})] \leq (1 - \alpha \gamma_k) H(x_k) + \frac{L_H \gamma_k^2}{2} \left(\norm{v_k}^2 + \sigma^2 \right).
    \]
    Rolling out this inequality we conclude the first statement. Next we sum all inequalities for all $k=0,\ldots,N-1$ with weights $\gamma_k$
    \[
        \bbE\left[\sum_{k=0}^{N-1} \gamma_k H(x_k)\right] \leq  H(x_0) \cdot \sum_{k=0}^{N-1} \gamma_k \prod_{j=0}^{k-1} (1 - \alpha \gamma_j)  + \frac{L_H}{2} \sum_{k=0}^{N-1} \gamma_k \sum_{j=0}^{k-1} \gamma_j^2 (\norm{v_j}^2 + \sigma^2) \prod_{\ell=j+1}^{k-1} (1 - \alpha \gamma_\ell)\, .
    \]
    First we apply \Cref{lm:sums_gamma_rev_bound} for the first term. Next we change the order of summation and apply \Cref{lm:sums_gamma_rev_bound} again
    \begin{align*}
        \sum_{k=0}^{N-1} \gamma_k \sum_{j=0}^{k-1} \gamma_j^2 (\norm{v_j}^2 + \sigma^2) \prod_{\ell=j+1}^{k-1} (1 - \alpha \gamma_\ell) &=
        \sum_{j=0}^{N-1} \gamma_j^2 (\norm{v_j}^2 + \sigma^2) \sum_{k=j+1}^{N-1} \gamma_k  \prod_{\ell=j+1}^{k-1} (1 - \alpha \gamma_\ell)
        \\
        &\leq \frac{1}{\alpha} \sum_{j=0}^{N-1} \gamma_j^2 (\norm{v_j}^2 + \sigma^2).
    \end{align*}
\end{proof}

\begin{lemma}\label{lm:sums_gamma_rev_bound}
    Let $b > 0$ and $(\gamma_k)_{k \geq 0}$ be a non-increasing sequence such that $\gamma_0 \leq 1/b$. Then
    \[
        \sum_{k=0}^{n} \gamma_k \prod_{j=0}^{k-1} (1 - b \gamma_j) = \frac{1 - \prod_{j=0}^n (1 - b \gamma_j)}{b}
    \]
\end{lemma}
\begin{proof}
    Introduce $u_{i:j} = \prod_{\ell=i}^j (1- b \gamma_{\ell})$. Notice that $u_{0:k-1} - u_{0:k} = u_{0:k-1} \cdot b \gamma_k$. Summing this equation from $0$ to $n$ we conclude the statement.
\end{proof}

To provide rates of convergence for the final algorithm we have to proof the following proposition.
\begin{proposition}\label{prop:reduced_rates}
    Assume \Cref{hyp:cont_model}'-\ref{hyp:disc_model}. Let $x_0 \in \cM$. If for all $k \in \bbN$ $\gamma_k \equiv \gamma$ where $\gamma \leq \min(\alpha^{-1}, (L_f + \alpha L_H \mu_h^{-2})^{-1})$, then for any $N \in \bbN$ the following holds
    \begin{align*}
        \bbE\left[ \sum_{k=0}^{N-1} \gamma \norm{v_k}^2 \right] &\leq 4\Delta_N + 2\left(L_f + \alpha L_H \mu_h^{-2}  \right) \cdot \sigma^2 \cdot \gamma^2 N \\
        &+ 4 \widetilde{C}^2 \cdot \alpha L_H \cdot \gamma^2 N  + 4 \widetilde{C} \cdot \sqrt{\frac{\alpha \gamma N}{2}} \sqrt{\frac{L_H \sigma^2}{2} \cdot \gamma^2 N},
    \end{align*}
    where $D_0 = f(x_0) - \inf_{x\in K} f(x)$ and $\widetilde{C} = B_f M_h \mu_h^{-2}$. 
\end{proposition}
\begin{proof}
    First we use the definition of smoothness of the function $f$
    \[
        \bbE_k[f(x_{k+1})] \leq f(x_k) + \gamma_k \nabla f(x_k)^\top v_k + \frac{L_f \gamma_k^2}{2} (\norm{v_k}^2 + \sigma^2).
    \]
    Next we notice that $\nabla f(x_k) + (\alpha(x_k) + \lambda(x_k)) \nabla H(x_k) = -v_k$, thus
    \begin{equation}\label{eq:reduced_f_difference}
        \begin{split}
            \bbE_k[f(x_{k+1})] &\leq f(x_k) - \gamma_k \left( 1 - \frac{L_f \gamma_k}{2} \right) \norm{v_k}^2 + \frac{L_f \gamma_k^2 \sigma^2}{2} \\
            &- \gamma_k (\alpha(x_k) + \lambda(x_k)) \nabla H(x_k)^\top v_k.
        \end{split}
    \end{equation}
    By orthogonality property and choice of $\alpha(x_k)$ we have
    \[
        \nabla H(x_k)^\top v_k = - \alpha(x_k) \norm{\nabla H(x_k)}^2 = -\alpha H(x_k),
    \]
    therefore, rolling out inequality for any $N \in \bbN$
    \begin{equation}\label{eq:reduced_f_difference_final}
        \begin{split}
        \bbE[f(x_N)] &\leq f(x_0) + \bbE\left[\sum_{k=0}^{N-1}\left( \frac{L_f \gamma_k^2}{2} - \gamma_k\right) \norm{v_k}^2 \bbE\right] + \sum_{k=0}^{N-1} \frac{L_f \gamma_k^2 \sigma^2}{2} \\
        &- \alpha \bbE\left[\sum_{k=0}^{N-1}  \gamma_k (\alpha(x_k) + \lambda(x_k)) H(x_k)\right].
        \end{split}
    \end{equation}
    Next we have to analyze the last sum. To do it, we start from definitions of $\alpha(x_k)$, $\lambda(x_k)$ and $H(x_k)$
    \[
        \left\vert \sum_{k=0}^{N-1} \gamma_k (\alpha(x_k) + \lambda(x_k)) H(x_k) \right\vert  \leq \sum_{k=0}^{N-1} \gamma_k \frac{\vert \alpha H(x_k) - \nabla f(x_k)^\top \nabla H(x_k)\vert }{\norm{ \nabla h(x) h(x_k)}^2 } \cdot \frac{1}{2} \norm{h(x_k)}^2.
    \]
    Next we apply Cauchy-Schwartz inequality combined with definition of $B_f$, $\mu_h$ and $M_h$ we have $\norm{ \nabla h(x) h(x_k)}^2 \geq \mu_h \norm{h(x)}^2$ and
    \[
        \left\vert \sum_{k=0}^{N-1} \gamma_k (\alpha(x_k) + \lambda(x_k)) H(x_k) \right\vert \leq \frac{\alpha}{2\mu_h^{2}} \sum_{k=0}^{N-1} \gamma_k H(x_k) + \frac{B_f M_h}{\sqrt{2} \cdot \mu_h^{2}} \sum_{k=0}^{N-1} \gamma_k \sqrt{H(x_k)}.
    \]
    By Cauchy-Schwartz inequality
    \[
        \sum_{k=0}^{N-1} \gamma_k \sqrt{H(x_k)} \leq \sqrt{\sum_{k=0}^{N-1} \gamma_k} \cdot \sqrt{\sum_{k=0}^{N-1} \gamma_k H(x_k)}.
    \]
    Next we are going to deal with expectation. By Jensen's inequality applied to a square root
    \begin{align*}
        \bbE\left[\biggl\vert \sum_{k=0}^{N-1} \gamma_k (\alpha(x_k) + \lambda(x_k)) H(x_k) \biggl\vert\right] &\leq \frac{\alpha}{2\mu_h^2} \bbE\left[ \sum_{k=0}^{N-1} \gamma_k H(x_k) \right] \\
        &+ \frac{B_f M_h}{\sqrt{2} \cdot \mu_h^2} \sqrt{\sum_{k=0}^{N-1} \gamma_k} \sqrt{\bbE\left[ \sum_{k=0}^{N-1} \gamma_k H(x_k) \right]}
    \end{align*}
    By assumption we have $\gamma_k \leq \alpha^{-1}$, so we can apply \Cref{lm:reduced_constrain_rate} and obtain
    \begin{align*}
        \bbE\biggl[\biggl\vert \sum_{k=0}^{N-1} \gamma_k &(\alpha(x_k) + \lambda(x_k)) H(x_k) \biggl\vert\biggl] \leq \frac{1}{2 \mu_h^2}\left( H(x_0) +  \frac{L_H}{2} \bbE\left[\sum_{k=0}^{N-1} \gamma_k^2 \norm{v_k}^2\right] + \frac{L_H}{2} \sum_{k=0}^{N-1} \gamma_k^2 \sigma^2 \right) \\
        &+ \frac{B_f M_h}{\sqrt{2\alpha} \mu_h^{2}} \sqrt{\sum_{k=0}^{N-1} \gamma_k} \cdot \sqrt{H(x_0) + \frac{L_H}{2} \bbE\left[\sum_{k=0}^{N-1}  \gamma^2_k \norm{v_k}^2\right] + \frac{L_H}{2} \sum_{k=0}^{N-1} \gamma_k^2 \sigma^2 }.
    \end{align*}
    For simplicity we assume $H(x_0) = 0$ and that $\gamma_k \equiv \gamma$ that satisfies the following inequality
    \[
        \gamma \leq \frac{1}{L_f + \alpha L_H \mu_h^{-2}}.
    \]
    Define $\Delta f_N = f(x_0) - f(x_N)$ and $S_N = \bbE\left[\sum_{k=0}^{N-1}\norm{v_k}^2\right]$, then by rearranging term in \eqref{eq:reduced_f_difference_final} and applying inequality $\sqrt{a+b} \leq \sqrt{a} + \sqrt{b}$ for positive $a,b$
    \begin{align*}
        \frac{\gamma}{2} S_N &\leq \Delta f_N + \frac{(L_f + \alpha L_H \mu_h^{-2}) \cdot \sigma^2}{2} \gamma^2 N \\
        &+ \frac{B_f M_h L_H}{2\mu_h^2} \cdot (\alpha^{1/2} \gamma^{3/2} N) \cdot \sigma  + \frac{B_f M_h}{\mu_h^2} \cdot \sqrt{\frac{\alpha \gamma^2 N \cdot L_H}{2}} \sqrt{\gamma S_N}.
    \end{align*}
    We have the quadratic inequality in $\sqrt{\gamma S_N}$ that could be easily solved. Using the fact that if $x^2  \leq 2ax + 2b$ then $x \leq a + \sqrt{a^2 + 2b} \leq 2a + \sqrt{2b}$ and a numeric inequality $(2a+\sqrt{2b})^2 \leq 8 a^2 + 4b$
    \begin{align*}
        \bbE\left[ \sum_{k=0}^{N-1} \gamma \norm{v_k}^2 \right] &\leq 4 \Delta f_N + 2\left( L_f + \alpha L_H \mu_h^{-2} \right)\cdot \sigma^2  \cdot \gamma^2 N \\
        &+ \left(\frac{4B_f^2 M_h^2 \cdot \alpha L_H}{\mu_h^{4}}\right) \cdot \gamma^2 N  + \frac{2B_f M_h L_H}{\mu_h^2} \cdot (\alpha^{1/2} \gamma^{3/2} N) \cdot \sigma
    \end{align*}
    Finally, we notice that $D_0 = f(x_0) - \inf_{x\in K} f(x)$ is an upper bound on $\Delta f_N$.
\end{proof}

Now we are ready to prove the main convergence results. It will be divided into two independent propositions.

\begin{proposition}[Convergence in deterministic case]\label{prop:reduced_deterministic_rates}
    Assume \Cref{hyp:cont_model}'-\ref{hyp:disc_model} and let $x_0 \in \cM$. 
     Let $\sigma^2 = 0$ and also define $\bar D$ as a known constant. If for all $k \in \bbN$, $\gamma_k \equiv \bar\gamma$ where $\bar\gamma =\min(\alpha^{-1}, (L_f + \alpha L_H \mu_h^{-2})^{-1}, \bar D \cdot N^{-1/3})$,  and $\alpha = \bar\gamma$ then for any $N \in \bbN$ the following holds
    \[
        \min_{k=0,\ldots,N-1} \left\{\norm{\nabla_{\tV} f(x_k)}^2 + \frac{1}{2} \norm{h(x_k)}^2\right\} \leq \frac{8D_0 (L_f + L_H \mu_h^{-2})}{N} + \left(\frac{8 D_0}{\bar D} + 8 \widetilde{C} L_H \cdot \bar D \right) \cdot N^{-2/3},
    \]
    where $D_0 = f(x_0) - \inf_{x\in K} f(x)$ and $\widetilde{C} = B_f M_h \mu_h^{-2}$. 
    
    In particular, \redalgo\ outputs a point $\hat x$ for which the minima in the left-hand side attains such that $\norm{h(\hat x)} \leq \varepsilon$ and $\norm{\nabla_{\tV} f(\hat x)} \leq \varepsilon$ in $\cO(\varepsilon^{-3})$ iterations.
\end{proposition}
\begin{proof}
    Apply \Cref{prop:reduced_rates} with $\sigma^2 = 0$ and without expectations
    \[
        \sum_{k=0}^{N-1} \gamma \norm{v_k}^2 \leq 4D_0 + 4 \widetilde{C}^2 \cdot \alpha L_H \cdot  \gamma^2 N.
    \]
    and, at the same time, by combination of \Cref{lm:reduced_constrain_rate} and \Cref{prop:reduced_rates}
    \[
        \sum_{k=0}^{N-1} \gamma \left( \norm{v_k}^2 + H(x_k) \right) \leq 4D_0\left(1 + \frac{\gamma}{\alpha}\right) + 4 \widetilde{C}^2 \cdot \alpha L_H \cdot  \gamma^2 N \left( 1 + \frac{\gamma}{\alpha} \right).
    \]
    By taking $\alpha = \gamma$ and using orthongonality property of $v_k$ we have
    \[
        \min_{k=0,\ldots,N-1} \left\{\norm{\nabla_{\tV} f(x_k)}^2 + \frac{1}{2} \norm{h(x_k)}^2\right\} \leq \frac{8 D_0}{\gamma N} + 8 \widetilde{C}^2 L_H \cdot \gamma^2.
    \]
    To balance these two terms we choose $\gamma_k \equiv \bar\gamma = \min(1, (L_f + \bar\gamma L_H \mu_h^{-2})^{-1}, \bar D \cdot N^{1/3})$ and obtain
    \[
        \min_{k=0,\ldots,N-1} \left\{\norm{\nabla_{\tV} f(x_k)}^2 + \frac{1}{2} \norm{h(x_k)}^2\right\} \leq \frac{8D_0 (L_f + L_H \mu_h^{-2})}{N} + \left(\frac{8 D_0}{\bar D} + 8 \widetilde{C} L_H \cdot \bar D \right) \cdot N^{-2/3}.
    \]
\end{proof}

\begin{proposition}[Convergence in stochastic case]\label{prop:reduced_stochastic_rates}
    Assume \Cref{hyp:cont_model}'-\ref{hyp:disc_model} and let $x_0 \in \cM$. 
    Let $\sigma^2 > 0$ and also define $\bar D$ as a known constant. If for all $k \in \bbN$ $\gamma_k \equiv \bar\gamma$ where $\bar\gamma =\min(\alpha^{-1}, (L_f + \alpha L_H \mu_h^{-2})^{-1}, \bar D \cdot N^{-1/2})$. Fix number of steps $N > 0$. Let $\hat k$ be a uniform index sampled from the set $\{0,\ldots, N-1\}$ then the following holds
    \begin{align*}
        \bbE\left[\norm{\nabla_{\tV} f(x_{\hat k})}^2 + \frac{1}{2}\norm{h(x_{\hat k})}^2\right] &\leq \frac{4D_0 (L_f + L_h \mu_h^{-2})}{N} + \frac{4 D_0}{\bar D \cdot N^{1/2}}  +  \frac{4 \widetilde{C}^2 \bar D^2 \cdot L_H}{N}\\
        &+ \frac{\bar D}{N^{1/2}}\left( 2\left(L_f + \gamma L_H \mu_h^{-2}  \right) \cdot \sigma^2 + 2 \widetilde{C} \cdot \sqrt{ \frac{L_H \sigma^2}{2}} \right)
    \end{align*}
    where $D_0 = f(x_0) - \inf_{x\in K} f(x)$ and $\widetilde{C} = B_f M_h \mu_h^{-2}$. 
    
    In particular, \redalgo\ outputs a point $\hat x = x_{\hat k}$ such that $\bbE[\norm{h(\hat x)}] \leq \varepsilon$ and $\bbE[\norm{\nabla_{\tV} f(\hat x)}] \leq \varepsilon$ in $\cO(\varepsilon^{-4})$ iterations.
\end{proposition}
\begin{proof}
    Let us start from \Cref{prop:reduced_rates} with taking $\alpha = \gamma$
    \begin{align*}
        \bbE\left[ \sum_{k=0}^{N-1} \gamma \norm{v_k}^2 \right] &\leq 4D_0 + 2\left(L_f + \gamma L_H \mu_h^{-2}  \right) \cdot \sigma^2 \cdot \gamma^2 N \\
        &+ 4 \widetilde{C}^2 \cdot L_H \cdot \gamma^3 N  + 2 \widetilde{C} \cdot \sqrt{ \frac{L_H \sigma^2}{2}} \cdot \gamma^2 N.
    \end{align*}
    Combining \Cref{lm:reduced_constrain_rate} with \Cref{prop:reduced_rates} and using orthogonality property
    \begin{align*}
        \bbE\left[\frac{1}{N}\sum_{k=0}^{N-1}\left\{ \norm{\nabla_{\tV} f(x_k)}^2 + \frac{1}{2}\norm{h(x_k)}^2\right\} \right] &\leq \frac{4 D_0}{\gamma N} + \gamma \cdot \left( 2\left(L_f + \gamma L_H \mu_h^{-2}  \right) \cdot \sigma^2 + 2 \widetilde{C} \cdot \sqrt{ \frac{L_H \sigma^2}{2}} \right) \\
        &+ 4 \widetilde{C}^2 \cdot L_H \cdot \gamma^2.
    \end{align*}
    Notice that in the left-hand side we have exactly expectation over $\hat k$. Thus taking $\gamma_k \equiv \bar \gamma = \min(1, (L_f + \bar \gamma L_H \mu_h^{-2})^{-1}, \bar D \cdot N^{-1/2})$ we obtain
    \begin{align*}
        \bbE\left[\norm{\nabla_{\tV} f(x_{\hat k})}^2 + \frac{1}{2}\norm{h(x_{\hat k})}^2\right] &\leq \frac{4D_0 (L_f + L_h \mu_h^{-2})}{N} + \frac{4 D_0}{\bar D \cdot N^{1/2}}  +  \frac{4 \widetilde{C}^2 \bar D^2 \cdot L_H}{N}\\
        &+ \frac{\bar D}{N^{1/2}}\left( 2\left(L_f + \gamma L_H \mu_h^{-2}  \right) \cdot \sigma^2 + 2 \widetilde{C} \cdot \sqrt{ \frac{L_H \sigma^2}{2}} \right)
    \end{align*}
\end{proof}

%% file: proof_stiefel.tex
\subsection{Proof of Theorem~\ref{th:stief_rates}}\label{app:geometric}

\begin{lemma}\label{lem:riemannian_gradient_by_projection}
    Let $(\cM, g)$ be a Riemannian manifold with a Riemannian metric $g_x(\xi, \eta) = \langle \xi, G_x \eta \rangle$ and let $f:\bbR^n \rightarrow \bbR$ be continuously differentiable. Then for any $x \in \cM$ we have
    \begin{equation}\label{eq:riem_gradient_opt_problem}
        \Grad_{\cM} f(x) = \argmin_{v \in \cT_x(\cM)} \frac{1}{2} \norm{v - G^{-1}_x \nabla f(x)}^2_{g_x}.
    \end{equation}
\end{lemma}
\begin{proof}
    Let $v^\star$ be a solution to \eqref{eq:riem_gradient_opt_problem}. Then it could be written as a solution to the following variational inequality
    \[
        \forall v \in \cT_x(\cM): \langle \nabla F(v^\star) , v - v^\star \rangle \geq 0,
    \]
    where $F(v) = \frac{1}{2} \norm{v - G^{-1}_x \nabla f(x)}^2_{g_x}$. By a direct computation we have
    \[
            \forall v \in \cT_x(\cM): \langle G_x v^\star - \nabla f(x) , v - v^\star \rangle \geq 0.
    \]
    Fix an arbitrary $\xi \in \cT_x(\cM)$. Since $\cT_x(\cM)$, we have that $v_1 = \xi + v^\star$ and $v_2 = -\xi + v^\star$ lies in $\cT_x(\cM)$. Thus
    \[
         \langle G_x v^\star - \nabla f(x) , \xi \rangle \geq 0, \quad \langle G_x v^\star - \nabla f(x) , -\xi \rangle \geq 0.
    \]
    Therefore, by arbitrary choice of $\xi$ we have
    \[
        \forall \xi \in \cT_x(\cM): \langle G_x v^\star - \nabla f(x) , \xi \rangle = 0.
    \]
\end{proof}


\begin{proposition}\label{prop:gof_expression}
  Let $x \in \bbR^n$ be such that $\nabla h(x)$ is of full rank. It holds that:
  \begin{equation*}
    \GOF(x) =  -\nabla h Ah -Q^{-1} \nabla f + Q^{-1} \nabla hB \nabla f \, ,
  \end{equation*}
  with $B(x) \in \bbR^{n_h \times n}$ defined as:
  \begin{equation*}
    B(x)= (\nabla h(x)^{\top} Q^{-1}(x) \nabla h(x))^{-1} \nabla h(x)^{\top} Q^{-1}(x)\, .
  \end{equation*}
\end{proposition}
\begin{proof}
It holds that:
  \begin{equation}
      \argmin_{v \in V} \norm{v + Q^{-1}\nabla f}_{q}^2 = \argmin_{v \in V}\norm{Q^{1/2} v + Q^{-1/2} \nabla f}^2 = Q^{-1/2}\argmin_{v \in Q^{1/2}V} \norm{ v + Q^{-1/2} \nabla f}^2 \, .
  \end{equation}
  Noticing that $Q^{1/2} V := \{v \in \bbR^n : (Q^{-1/2} \nabla h)^{\top}v = 0 \}$, we obtain our claim by applying Lemma~\ref{lm:aff_proj} with $W = Q^{-1/2} \nabla h$, $y = Q^{-1/2} \nabla f$ and $b= 0$.
\end{proof}

Denote $M_q$ the following constant:
  \begin{equation}
      M_q := \sup_{x \in K} \norm{(\nabla h A h)^{\top} Q \nabla h A + \nabla f^{\top} \nabla h A - 2(\nabla h B \nabla f)^{\top} \nabla h A }\, .
  \end{equation}
It will play the same role as $M_1$ (notice that $M_q = M_1$, if $Q = \cI_n$) in the proof of Theorem~\ref{th:gen_rates}.
\begin{lemma}\label{lm:f+v_riem}
  Let \Cref{hyp:cont_model}-\Cref{hyp:riem_proj}. Denoting $v = \GOF(x)$, it holds that
  \begin{equation*}
  \norm{ (Q^{-1}\nabla f + v)^{\top} Q v} \leq M_q \norm{h} \, .
  \end{equation*}
\end{lemma}
\begin{proof}
Note that, as previously, $\nabla h^{\top} v = - \nabla h^{\top} \nabla h A h$. Thus, using Proposition~\ref{prop:gof_expression}, we obtain:
\begin{equation*}
  \begin{split}
    (Q^{-1}\nabla f + v)^{\top} Qv &= -(\nabla h A h)^{\top} Q v + (B\nabla f)^{\top} \nabla h^{\top} v \\
&=-(\nabla h A h)^{\top} Q v  - (B \nabla f)^{\top} \nabla h^{\top}\nabla h A h \\
&= (\nabla h A h)^{\top} Q (\nabla h A h) + (\nabla h A h)^{\top} \nabla f - 2  (\nabla h B \nabla f)^{\top}(\nabla h A h) \\
&= \left( (\nabla h A h)^{\top} Q \nabla h A + \nabla f^{\top} \nabla h A - 2(\nabla h B \nabla f)^{\top} \nabla h A \right) h \, .
  \end{split}
\end{equation*}
This completes the proof by the definition of $M_q$.
\end{proof}

Denote $\overline{M_q}:= M_q/\alpha_m$. The following is an extension of Proposition~\ref{prop:lyap_stoch} to the present case.
\begin{proposition}[Geometry aware discrete Lyapunov function]\label{prop:lyap_sto_geom}
    Let \Cref{hyp:cont_model}--\Cref{hyp:riem_proj} hold. If for all $k$, $\gamma_k \leq \alpha_m^{-1}$, then for all $M \geq \overline{M}_q$, it holds:
\begin{equation}\label{eq:lyap_sto_riem}
    \bbE_k [\Lambda_M(x_{k+1})]- \Lambda_M(x_k) \leq - \gamma_k \norm{ v_k}^2 \left( \frac{1}{C_q} - \frac{L_f + M L_h}{2} \gamma_k\right) + \frac{L_f + ML_h}{2}\sigma^2 \gamma_k^2 \, .
\end{equation}
\end{proposition}
\begin{proof}
    Following the same path as in the proof of Proposition~\ref{prop:lyap_stoch}, we obtain a generalization of Equation~\eqref{eq:f_disc_decr}:
      \begin{equation*}
    \begin{split}
          \bbE_k[f(x_{k+1})] - f(x_k) &\leq \gamma_k \nabla f(x_k)^{\top} v_k + \frac{L_f}{2} \gamma_k^2 \bbE_k[\norm{v_k + \eta_{k+1}}]^2  \\
          &\leq \gamma_k  (Q^{-1}(x_k) \nabla f(x_k))^{\top} Q(x_k) v_k  + \frac{L_f}{2} \gamma_k^2 (\norm{v_k}^2 +\sigma^2)\\
          &\leq - \gamma_k v_k^{\top} Q(x_k) v_k+ (Q^{-1} (x_k) \nabla f(x_k) + v_k)^{\top}Q(x_k) v_k + \frac{L_f}{2} \gamma_k^2 (\norm{v_k}^2 +\sigma^2) \\
          &\leq - \gamma_k \norm{v_k}^2 \left( \frac{1}{C_q} - \frac{L_f}{2} \gamma_k \right) + M_q \norm{h(x_k)}  + \frac{L_f}{2}\gamma_k^2 \sigma^2 \, ,
    \end{split}
  \end{equation*}
  where we have used Lemma~\ref{lm:f+v_riem} for the third and \Cref{hyp:riem_proj} for the fourth inequality.
Since Equation~\eqref{eq:h_disc_decr} remains unchanged, we obtain the claimed inequality.
\end{proof}

The end of the proof of Theorem~\ref{th:stief_rates} then follows, \emph{mutatis mutatis}, the one of Theorem~\ref{th:gen_rates}, upon replacing Proposition~\ref{prop:lyap_stoch} with Proposition~\ref{prop:lyap_sto_geom}.